\documentclass[twoside,a4paper,11pt]{amsart}

\usepackage[T1]{fontenc}
\usepackage{mathpazo} % add possibly `sc` and `osf` options
\usepackage{eulervm}

\usepackage[utf8]{inputenc}

\usepackage{fullpage}
\usepackage{amscd}
\usepackage{amsmath}
\usepackage{amssymb}
\usepackage{amsthm}
\usepackage{color}
\usepackage{graphicx}
\usepackage{hyperref}
\usepackage[noabbrev,capitalize]{cleveref}
\usepackage{mathrsfs}
\usepackage{mathtools}
\usepackage{pict2e}
\usepackage{stackrel}
\usepackage[T1]{fontenc}
\usepackage{tikz}
\usepackage{tikz-cd}
\usetikzlibrary{decorations.pathmorphing,decorations.markings,arrows,calc,shapes.geometric,arrows.meta,positioning}
\usepackage[utf8]{inputenc}
\usepackage{wasysym}
\usepackage[all]{xy}
\usepackage{xcolor}
\usepackage{xy}
\usepackage{CJKutf8}
\usepackage{ipaex-type1}
\usepackage{epigraph}

\hypersetup{
    colorlinks,
    linkcolor=[rgb]{0.733, 0.149, 0.286},
    citecolor=[rgb]{0.733, 0.149, 0.286},
    urlcolor=[rgb]{0.733, 0.149, 0.286}
}

\usepackage{palatino}
\usepackage{mathpazo} 
\allowdisplaybreaks

\newtheorem{lemma}{Lemma}[section]
\newtheorem{theorem}[lemma]{Theorem}
\newtheorem{Corollary}[lemma]{Corollary}
\newtheorem{Proposition}[lemma]{Proposition}

\newtheorem*{Notation}{Notation}

\newtheorem{theoremintro}{Theorem}

\theoremstyle{definition}
\newtheorem{Definition}[lemma]{Definition}
\newtheorem{Remark}[lemma]{\sc Remark}
\newtheorem{Counterexample}[lemma]{\sc Counter-Example}
\newtheorem{Example}[lemma]{\sc Example}

\def\colim{\mathop{\mathrm{colim}}}

\newcommand{\smod}{\mathbb{S}\textsf{-}\mathsf{mod}}

\newcommand{\ac}{\scriptstyle \textrm{!`}}

\newcommand{\qi}{\xrightarrow{ \,\smash{\raisebox{-0.65ex}{\ensuremath{\scriptstyle\sim}}}\,}}
\newcommand{\lqi}{\xleftarrow{ \,\smash{\raisebox{-0.65ex}{\ensuremath{\scriptstyle\sim}}}\,}}

\newcommand{\C}{\mathcal{C}}

\newcommand{\kk}{\Bbbk}

\newcommand{\PP}{\mathcal{P}}

\input cyracc.def
\font\tencyr=wncysc10
\def\cyr{\tencyr\cyracc}
\def\diracComb{\mbox{\cyr SH}}

\author{Victor Roca i Lucio}
\title{Absolute algebras, contramodules, and duality squares}
\date{\today}

\address{Victor Roca i Lucio, Université Paris Cité and Sorbonne Université, CNRS, IMJ-PRG, F-75013 Paris, France}
\email{\href{mailto:roca-lucio@imj-prg.fr}{roca-lucio@imj-prg.fr}}



\thanks{The author was partially supported by the projects ANR-20-CE40-0016 HighAGT and ANR-22-CE40-0008 SHoCoS funded by the Agence Nationale pour la Recherche.}

\begin{document}
	
\begin{abstract}
Absolute algebras are a new type of algebraic structures, endowed with a meaningful notion of infinite sums of operations without supposing any underlying topology. Opposite to the usual definition of operadic calculus, they are defined as algebras over cooperads. The goal of this article is to develop this new theory. First, we relate the homotopy theory of absolute algebras to the homotopy theory of usual algebras via a \textit{duality square}. It intertwines bar-cobar adjunctions with linear duality adjunctions. In particular, we show that linear duality functors between types of coalgebras and types of algebras are Quillen functors and that they induce equivalences between objects with finiteness conditions on their homology. 
We give general comparison results between absolute types of algebras and their classical counterparts. We work out examples of this theory such as absolute associative algebras and absolute Lie algebras, and show that it includes the theory of contramodules. Finally, in \cite{campos2020lie}, the authors showed that two nilpotent Lie algebras whose universal enveloping algebras are isomorphic as associative algebras must be isomorphic. As an application of our results, we generalize their theorem to the setting of absolute Lie algebras and absolute $\mathcal{L}_\infty$-algebras. 
\end{abstract}

\maketitle

\setcounter{tocdepth}{1}

\tableofcontents

\section*{Introduction}
Let $A$ be a vector space. The data of an associative algebra structure amounts to the data of a "multiplication table" of elements of $A$: given two elements, it states which element in $A$ is their product. One can compile this "multiplication table" into a single morphism
\[
\left\{
\begin{tikzcd}[column sep=3pc,row sep=0pc]
\gamma_A: \displaystyle \bigoplus_{n \geq 1} A^{\otimes n} \arrow[r]
&A \\
a_1 \otimes \cdots \otimes a_n \arrow[r,mapsto]
&\gamma_A(a_1 ,\cdots,a_n)~,
\end{tikzcd}
\right.
\]
which assigns to any $n$-tuple $(a_1, \cdots, a_n)$ the value of their product $\gamma_A(a_1,\cdots,a_n)$. The conditions that an associative algebra structure has to satisfy are encoded by the fact that $\gamma_A$ defines a structure of an \textit{algebra over a monad}. But in this classical algebraic framework, infinite collections of elements do not have an assigned value by this "multiplication table". In order to give a value to expressions like $\sum_{n \geq 1} a_{1,n} \otimes \cdots \otimes a_{n,n}$, the classical approach used so far has been to add the data of a topology on $A$, so that the values assigned to any partial sum \textit{converge} to a well-defined element in $A$. This solution has several issues, specially when one tries to mix \textit{homotopical algebra} and \textit{topology}. For instance, neither the category of topological abelian groups nor the category of topological modules over a topological ring are abelian. This type of problems is the main motivation for the recent approaches of \cite{scholze} and \cite{barwick}.

\medskip

Absolute algebras are a new type of algebraic structures where infinite sums of operations have well-defined images \textit{by definition}. For instance, the data of a \textit{(non-unital) absolute associative algebra structure} on a vector space $A$ amounts to the data of a "transfinite multiplication table", where any series of elements in $A$ is assigned a value in $A$. This information compiles into a single structural map
\[
\left\{
\begin{tikzcd}[column sep=2.5pc,row sep=-0.5pc]
\gamma_A: \displaystyle \prod_{n \geq 1} A^{\otimes n} \arrow[r]
&A \\
\displaystyle \sum_{n \geq 1} a_{1,n} \otimes \cdots \otimes a_{n,n} \arrow[r,mapsto]
&\displaystyle \gamma_A \left(\sum_{n \geq 1} a_{1,n} \otimes \cdots \otimes a_{n,n} \right)~.
\end{tikzcd}
\right.
\]
The bright point is that this type of structures is defined as an algebra over a monad, hence the category of absolute associative algebras enjoys many desirable properties. For example, it is complete and cocomplete. Any absolute associative algebra structure gives an associative algebra structure in the classical sense, simply by restricting the structural map to finite sums. There are many "absolute analogues" of standard types of algebraic structures: absolute Lie algebras, absolute $\mathcal{A}_\infty$-algebra, absolute $\mathcal{L}_\infty$-algebras, etc. This stems from the fact that this type of structures are defined as \textit{algebras over a cooperad}. 

\medskip

Recall that most of classical types of algebraic structures can be successfully encoded by an operad. See \cite{LodayVallette12} for instance. Thus, given an operad $\mathcal{P}$, we get a category of dg $\mathcal{P}$-algebras. The operadic calculus provides us with a powerful tool to study the homotopy category of these dg $\mathcal{P}$-algebras. The key element is the Koszul duality for operads, that allows one to construct a Koszul dual conilpotent cooperad $\mathcal{P}^{\hspace{1pt}\ac}$. This data allows us to construct a bar-cobar adjunction 
\[
\begin{tikzcd}[column sep=5pc,row sep=5pc]
\mathsf{dg}~\mathcal{P}\text{-}\mathsf{alg} \arrow[r,"\mathrm{B}"{name=B},shift left=1.1ex] 
&\mathsf{dg}~\mathcal{P}^{\hspace{1pt}\ac}\text{-}\mathsf{coalg} \arrow[l,"\Omega"{name=C},shift left=1.1ex] \arrow[phantom, from=C, to=B, , "\dashv" rotate=90]
\end{tikzcd}
\]
between the category of dg $\mathcal{P}$-algebras and the category of dg $\mathcal{P}^{\hspace{1pt}\ac}$-coalgebras. Contrary to popular belief, any coalgebra over a cooperad is, by definition, \textit{conilpotent}. Hence we omit this adjective when possible. Using this adjunction, one can transfer the model structure where weak equivalences are given by quasi-isomorphism onto the category of conilpotent $\mathcal{P}^{\hspace{1pt}\ac}$-coalgebras and obtain a Quillen equivalence. This approach, first developed by V. Hinich for dg Lie algebras in \cite{Hinich01} and by K. Lefèvre-Hasegawa in \cite{LefevreHasegawa03} for dg associative algebras, and was later extended to all Koszul operads by B. Vallette in \cite{Vallette14}. Conilpotent coalgebraic types of structures are given by decomposition maps, which assign finite sums to any element. Let $C$ be a vector space. For example, the data of a non-counital \textit{conilpotent} coassociative coalgebra structure on $C$ is equivalent to the data of a "decomposition table", compiled into a map
\[
\left\{
\begin{tikzcd}[column sep=3pc,row sep=-0.5pc]
\Delta_C: C \arrow[r]
&\displaystyle \bigoplus_{n \geq 1} C^{\otimes n}\\
c \arrow[r,mapsto]
&\displaystyle \Delta_C(c) = \sum_{i \in \mathrm{I}} c_1^{(i)} \otimes \cdots \otimes c_n^{(i)}~,
\end{tikzcd}
\right.
\]
where $\mathrm{I}$ is a finite set. It is \textit{conilpotent} precisely because this sum is finite, and therefore $\Delta_C$ lands on the direct sum instead of the product. But most coassociative coalgebras that appear in nature are not conilpotent, e.g: consider a base field $\kk$ together with the diagonal map. 

\medskip

In the same way as algebras are "the Koszul dual notion" to conilpotent coalgebras, since they both are "finitary types of structures", absolute types of algebras are the "Koszul dual notion" to non-conilpotent coalgebras. The infinite decompositions of elements in these non-conilpotent coalgebras are reflected in the "transfinite multiplication tables" of absolute algebras. The notion of an algebra over a cooperad was introduced in an abstract context by B. Le Grignou and D. Lejay in \cite{grignoulejay18}. Their goal was to study the homotopy theory of non necessarily conilpotent coalgebras. Since these types of structures are encoded as \textit{coalgebras over an operad}, the reasonable thing to do was to look at what an algebra over a cooperad looks like. Starting from a cooperad, one can construct a monad by considering a dual version of the Schur functor. Algebras over a cooperad are defined as algebras over its associated monad. Given an operad $\mathcal{P}$ and its Koszul dual cooperad $\mathcal{P}^{\hspace{1pt}\ac}$, the authors of \textit{loc.cit} construct a \textit{complete bar-cobar adjunction}
\[
\begin{tikzcd}[column sep=5pc,row sep=5pc]
\mathsf{dg}~\mathcal{P}\text{-}\mathsf{coalg} \arrow[r,"\widehat{\Omega}"{name=B},shift left=1.1ex] 
&\mathsf{dg}~\mathcal{P}^{\hspace{1pt}\ac}\text{-}\mathsf{alg}~, \arrow[l,"\widehat{\mathrm{B}}"{name=C},shift left=1.1ex] \arrow[phantom, from=C, to=B, , "\dashv" rotate=-90]
\end{tikzcd}
\]
and show that in some cases, the homotopy theory of these non necessarily conilpotent coalgebras can be recovered from the homotopy theory of the dual absolute algebras with a transferred model structure along this complete bar-cobar adjunction.

\medskip

\textbf{Main results.} The goal of this article is to develop the theory of algebras over cooperads, which we call \textit{absolute algebras}. Our first result answers a very natural question: how do these two aforementioned bar-cobar constructions relate to each other? We show that there are duality functors that intertwine both of these adjunctions in a duality square of commuting adjunctions.

\begin{theoremintro}[Duality square, Theorem \ref{thm: magical square}]
There exists a square of adjunctions 
\[
\begin{tikzcd}[column sep=5pc,row sep=5pc]
\left(\mathsf{dg}~\mathcal{P}\text{-}\mathsf{alg}\right)^{\mathsf{op}} \arrow[r,"\mathrm{B}^{\mathsf{op}}"{name=B},shift left=1.1ex] \arrow[d,"(-)^\circ "{name=SD},shift left=1.1ex ]
&\left(\mathsf{dg}~\mathcal{P}^{\hspace{1pt} \mathsf{\ac}} \text{-}\mathsf{coalg}\right)^{\mathsf{op}} \arrow[d,"(-)^*"{name=LDC},shift left=1.1ex ] \arrow[l,"\Omega^{\mathsf{op}}"{name=C},,shift left=1.1ex]  \\
\mathsf{dg}~\mathcal{P}\text{-}\mathsf{coalg} \arrow[r,"\widehat{\Omega}"{name=CC},shift left=1.1ex]  \arrow[u,"(-)^*"{name=LD},shift left=1.1ex ]
&\mathsf{dg}~\mathcal{P}^{\hspace{1pt} \mathsf{\ac}} \text{-}\mathsf{alg}~, \arrow[l,"\widehat{\mathrm{B}}"{name=CB},shift left=1.1ex] \arrow[u,"(-)^\vee"{name=TD},shift left=1.1ex] \arrow[phantom, from=SD, to=LD, , "\dashv" rotate=0] \arrow[phantom, from=C, to=B, , "\dashv" rotate=-90]\arrow[phantom, from=TD, to=LDC, , "\dashv" rotate=0] \arrow[phantom, from=CC, to=CB, , "\dashv" rotate=-90]
\end{tikzcd}
\] 
which commutes in the following sense: right adjoints going from the top right corner to the bottom left corner are naturally isomorphic.
\end{theoremintro}

The technical part of the above results is the construction of the linear duality adjunctions. In particular, the functor $(-)^\circ$ is a generalization of the Sweedler dual functor constructed in \cite{Sweedler69}. The above theorem admits a much more general formulation: it also holds for any curved twisting morphism in the sense of \cite{lucio2022curved}. Conceptually, it shows that absolute algebras appear every time one considers the linear dual of a conilpotent coalgebra. For example, taking the \textit{linear dual of the Bar construction}, which occurs in many instances in the literature: it is a key in ingredient in J. Lurie's proof of that formal moduli problems are encoded by dg Lie algebras \cite{Lurie11} and in the subsequent generalization of his result by D. Calaque, R. Campos and J. Nuiten in \cite{CCN19}.

\medskip

This duality square can in favorable cases be made compatible with the respective model structures of each of these categories. Then, this square becomes a square of Quillen adjunctions. This allows us to have, for the first time, a \textit{homotopical understanding} of the linear duality functor $(-)^*$ between types of coalgebras and types of algebras over a cofibrant dg operad $\mathcal{P}$.

\medskip

\begin{theoremintro}[Theorem \ref{thm: equivalence infini cat cog et alg}]
Let $\mathcal{P}$ be a cofibrant dg operad. There is an equivalence of $\infty$-categories

\[
\begin{tikzcd}[column sep=5pc,row sep=3pc]
         \mathsf{dg}~\PP\text{-}\mathsf{coalg}^{\mathsf{f.d},\mp}~\left[\mathsf{Q.iso}^{-1}\right] \arrow[r, shift left=1.1ex, "(-)^\ast"{name=A}]
         &\mathsf{dg}~\PP \text{-}\mathsf{alg}^{\mathsf{f.d},\pm}~\left[\mathsf{Q.iso}^{-1}\right]^\mathsf{op}~, \arrow[l, shift left=.75ex, "(-)^\circ"{name=B}] \arrow[phantom, from=A, to=B, , "\dashv" rotate=-90]
\end{tikzcd}
\]

between the $\infty$-category of dg $\PP$-algebras with degree-wise finite dimensional bounded below (resp. bounded above) homology and the $\infty$-category of $\PP$-coalgebras with degree-wise finite dimensional bounded above (resp. bounded below) homology. 
\end{theoremintro}

The above two theorems play a key role in \cite{lucio2022integration}, where we develop the integration theory of curved absolute $\mathcal{L}_\infty$-algebras. There, they allow us to prove that curved absolute $\mathcal{L}_\infty$-algebras provide us with rational models, which are dual to Sullivan's models, and to develop our approach to derived algebraic geometry using curved absolute $\mathcal{L}_\infty$-algebras.

\medskip

In Section 3, our goal is to establish several theoretical results about algebras over a general cooperad $\C$, and in particular, to compare these algebras to algebras over the linear dual operad $\C^*$. This amounts to comparing the absolute and the classical versions of the same algebraic structure. In general, we construct an adjunction between algebras over $\C$ and algebras over $\C^*$, where the right adjoint restricts the structural morphism of an absolute algebra to finite sums. We show that, under mild conditions on $\C$, this restriction functor is fully faithful. It would be natural to wonder then if absolute algebras can be obtained as some kind of completion of classical algebras. We explain why this is not the case in general: the basis example of the $I$-adic completion of augmented commutative algebras allows us to show that absolute commutative algebras differ from $I$-adically complete ones. This follows from the fact that the $I$-adic completion is not, in general, idempotent. 

\medskip

In Section 4, we work out examples and explain the theory in particular cases of interest. The first example is that of \textit{contramodules}. Contramodules over coassociative coalgebras where first introduced by S. Eilenberg and J. C. Moore in \cite{eilenbergmoore65} but later somewhat forgotten until they were extensively studied by L. Positselski, see for instance \cite{positselski2021contramodules}. Since a coassociative coalgebra is a cooperad concentrated in arity one, we show that contramodules are a particular example of absolute algebras. Engulfing this theory provides us with illuminating examples and counterexamples that shed a light into what is to be expected of this new type of algebraic structures. After, we treat \textit{in extenso} the cases of dg absolute associative algebras and dg absolute Lie algebras. 

\medskip

Finally, we apply this new framework to Lie theory. In \cite{campos2020lie}, the authors proved the following theorem: two nilpotent Lie algebras are isomorphic as Lie algebras if and only if their universal enveloping algebras are isomorphic as associative algebras. This result followed from a more general statement: two homotopy complete dg Lie algebras have quasi-isomorphic universal enveloping algebras if and only if they are quasi-isomorphic. The key element in the proof of this previous statement is, in turn, a comparison result between the deformation complexes of $\mathcal{C}_\infty$ and $\mathcal{A}_\infty$-coalgebras in general. In this last section, we show how to reinterpret this result on deformation complexes in the more general context of complete absolute Lie algebras and their universal enveloping absolute algebras. 

\begin{theoremintro}[Theorem \ref{thm: iso envelopantes Lie absolues} and Theorem \ref{thm: isos envelopantes absolues de L infinies}]\label{thm C intro}
Let $\kk$ be a field of characteristic zero and let $\mathfrak{g}$ and $\mathfrak{h}$ be two complete graded absolute Lie algebras. They are isomorphic as complete graded absolute Lie algebras if and only if their universal enveloping absolute algebras are isomorphic.
\end{theoremintro}

This result follows from a more general statement concerning complete dg absolute Lie algebras. By transporting via Koszul duality the main technical result of \cite[Theorem 4.27]{campos2020lie}, we immediately get that two such algebras are linked by a zig-zag of weak equivalences if and only if their universal enveloping constructions are. Since these weak equivalences are in particular quasi-isomorphisms, the above result follows when the differential is zero. As we explain in the last section of this paper, the result about homotopy complete Lie algebras (and thus about nilpotent Lie algebras) then follows by the fact that a dg Lie algebra is homotopy complete if and only if is homotopically fully faithfully included into complete absolute dg Lie algebras. 

\medskip

This absolute framework is more general, as for instance complete graded absolute Lie algebras include nilpotent graded Lie algebras without degree restrictions. In this particular case, we compute explicitly the universal enveloping absolute algebra, which is given by the completed tensor algebra modulo the standard relation. These methods also allow us to develop an analogue of Theorem \ref{thm C intro} for minimal $\mathcal{L}_\infty$-algebras using $\infty$-isomorphisms, where we show that if their universal enveloping abslute $\mathcal{A}_\infty$-algebras are weakly equivalent, then there exists an $\infty$-isomorphism between them. In particular, this applies to nilpotent (minimal) $\mathcal{L}_\infty$-algebras in the sense of \cite{Getzler09}, which are examples of absolute (minimal) $\mathcal{L}_\infty$-algebras. 
\medskip

\subsection*{Acknowledgments}
I would like to thank my former PhD. advisor Bruno Vallette for the numerious discussion we had and for his careful readings of this paper. I would also like to thank Damien Calaque, Geoffroy Horel, Brice Le Grignou, Johan Leray, Joost Nuiten, L. Positseski, Maxime Ramzi and Friedrich Wagemann for interesting discussions, as well as Dan Petersen for stimulating correspondence which led to Theorem 3.15.

\medskip

I would like to acknowledge the warm hospitality of Svenska KullagerFabriken and its employees in Göteborg, which provided great working conditions in order to finish this paper. This paper was written during my PhD. thesis at the Université Sorbonne Paris Nord, I would like to thank its great mathematical community. Finally, I would like to express my gratitude to the referee of this paper for all the comments, remarks and corrections that helped improve this paper, in particular for suggesting a simple proof of Proposition 2.15.

\medskip

\subsection*{Conventions}
Let $\Bbbk$ be a ground field of characteristic $0$. The ground category is the symmetric monoidal category $\left(\mathsf{pdg}\textsf{-}\mathsf{mod}, \otimes, \Bbbk\right)$ of pre-differential graded (pdg) $\Bbbk$-modules: these are graded modules $V$ endowed with a degree $-1$ endomorphism $d_V$. We work with the \textit{homological} degree conventions. The tensor product $\otimes$ of pdg modules is given by the graded module
\[
(A \otimes B)_n \coloneqq \displaystyle \bigoplus_{p+q=n}~A_p \otimes~ B_q ~,
\]
together with the pre-differential
\[
d_{A \otimes B} (a \otimes b) \coloneqq d_A(a) \otimes b + (-1)^{|a|}\hspace{1pt} a \otimes d_B(b)~.
\]
The isomorphism $\tau_{A,B}: A \otimes B \longrightarrow B \otimes A$ is given by the Koszul sign rule $\tau (a \otimes b) \coloneqq (-1)^{|a|.|b|} \hspace{1pt} b\otimes a$ on homogeneous elements. This symmetric monoidal category is \textit{closed}, meaning the tensor $\otimes$ of pdg modules admits a right adjoint. It is given by the graded module $\mathrm{hom}(A,B)$ of graded maps between two pdg modules $A$ and $B$, together with the pre-differential
\[
\partial(f) \coloneqq d_B \circ f -(-1)^{|f|}f \circ d_A~,
\]
where $f:A \longrightarrow B$ is a graded map of degree $|f|$. The suspension of a $V$ is denoted by $sV$, given by $(sV)_{p} \coloneqq V_{p-1}~.$ 

\medskip

A \textit{differential graded (dg) module} is a pre-differential graded module $V$ such that $d_V^2 = 0$. The category of dg modules is a full subcategory of pdg modules, preserved by the tensor product and the internal hom. A (p)dg module $V$ is \textit{degree-wise finite dimensional} if every $V_n$ is a finite dimensional $\Bbbk$-module. It is \textit{bounded above} if there exists an $m \in \mathbb{Z}$ such that $V_n = 0$ for $n \geq m$ and it is \textit{bounded below} if there exists an $m \in \mathbb{Z}$ such that $V_n = 0$ for $n \leq m$.

\medskip

A pdg $\mathbb{S}$-module  $M$ is a collection $\{M(n)\}_{n \in \mathbb{N}}$ of pdg $\kk[\mathbb{S}_n]$-modules, where $\mathbb{S}_n$ stands for the symmetric group of permutations of $n$ elements. This category is denoted by $\mathsf{pdg}~\mathbb{S}\textsf{-}\mathsf{mod}$. We denote $(\mathsf{pdg}~\mathbb{S}\textsf{-}\mathsf{mod}, \circ, \I)$ the monoidal category of pre-differential graded $\mathbb{S}$-modules endowed with the composition product $\circ$. Again, a dg $\mathbb{S}$-module is a pdg $\mathbb{S}$-module whose pre-differential squares to zero; they form a full subcategory of pdg $\mathbb{S}$-modules. We refer to \cite{LodayVallette12} for most of the notations, and to \cite{mathez} for a more detailed explanation of this framework and of the recollections. 

\vspace{1.5pc}
 
\section{Recollections}

\vspace{1.5pc}

The goal of this section is to briefly recall the recent developments in operadic calculus made in \cite{grignoulejay18}. Their main idea is encode \textit{non-necessarily conilpotent} types of coalgebras using operads, and relate them via Koszul duality with \textit{algebras over cooperads}. These algebras over cooperads define new types of algebraic structures which we will call "absolute algebras" in the present article. For a more thorough exposition, see also \cite[Chapter 1]{mathez}.

\medskip 

\subsection{Operads and curved cooperads}
The classical version of Koszul duality for operads, developed in \cite{GinzburgKapranov94,GetzlerJones94}, relates augmented operads with conilpotent coaugmented cooperads. See \cite[Chapter 6 and 7]{LodayVallette12} for a comprehensive account. A more general version of this duality, developed in \cite{HirshMilles12, grignou2019}, relates non-necessarily augmented operads with conilpotent curved cooperads.

\begin{Definition}[dg operad]\label{def: operad}
A \textit{dg operad} $\mathcal{P}$ amounts to the data $(\PP, \gamma_\PP, \eta, d_\PP)$ of a monoid in $(\mathsf{dg}~\mathbb{S}\textsf{-}\mathsf{mod}, \circ, \I)$.   
\end{Definition}

\begin{Definition}[augmented dg operad]
An \textit{augmented dg operad} $\PP$ amounts to the data of a dg operad $(\mathcal{P},\gamma,\eta)$ equipped with a morphism of dg operads $\nu: \PP \longrightarrow \I$ such that $\nu \circ \eta = \mathrm{id}.$
\end{Definition}

We denote by $\overline{\mathcal{P}}$ the kernel of the augmentation map $\nu$, which has the structure of a non-unital dg operad. 

\begin{Remark}
Many interesting operads \textit{do not} admit an augmentation. For instance, $u\mathcal{A}ss$, the operad which encodes \textit{unital} associative algebras, or $u\mathcal{C}om$, the operad which encodes \textit{unital} commutative algebras. In general, \textit{unital} types of structures are encoded by operads which do not admit augmentations. 
\end{Remark}

\begin{Definition}[pdg cooperad]\label{def: cooperad}
A \textit{pdg cooperad} $\C$ amounts to the data $(\C,\Delta,\epsilon,d_\C)$ of a comonoid in the monoidal category $(\mathsf{pdg}~\mathbb{S}\textsf{-}\mathsf{mod},\allowbreak \circ, \I)~.$ 
\end{Definition}

\begin{Definition}[coaugmented pdg cooperad]
A \textit{coaugmented pdg cooperad} $\C$ amounts to the data of a pdg cooperad $(\C, \Delta, \epsilon)$ equipped together with a morphism of pdg cooperads $\mu: \I \longrightarrow \C$ such that $\epsilon \circ \mu = \mathrm{id}$. 
\end{Definition}

We denote by $\overline{\C}$ the kernel of the counit map, which has the structure of a non-counital pdg cooperad when $\C$ admits a coaugmentation. 

\medskip

Curved cooperads are particular examples of pdg cooperads, endowed with a \textit{curvature} that controls how far the pre-differential is from being a differential.

\begin{Definition}[curved cooperad]\label{def curved cooperad}
A \textit{curved cooperad} $\C$ amounts to the data $(\C,\Delta,\epsilon,d_\C,\Theta_\C)$ of a pdg cooperad $(\C,\Delta,\epsilon,d_\C)$ and a morphism of pdg $\mathbb{S}$-modules $\Theta_\C: (\C,d_\C) \longrightarrow (\I,0)$ of degree $-2$ called the \textit{curvature}, such that the following diagram commutes: 
\[
\begin{tikzcd}[column sep=8.5pc,row sep=3pc]
\C \arrow[r,"\Delta_{(1)}"] \arrow[rrd,"d_\C^2", bend right =10]
&\C \circ_{(1)} \C \arrow[r,"(\mathrm{id}~ \circ_{(1)} ~ \Theta_\C)~-~(\Theta_\C~ \circ_{(1)}~ \mathrm{id})~"] 
&(\C \circ_{(1)} \I) \oplus (\I \circ_{(1)} \C) \cong  \C \oplus \C  \arrow[d,"+"]\\
&
&\C~,
\end{tikzcd}
\]
where $+$ is given by $+(\mu,\nu) \coloneqq \mu + \nu$ and where $\Delta_{(1)}$ denotes the partial decomposition of the cooperad $\C$, that is, the decompositions which only involve two non-trivial elements.
\end{Definition}

The \textit{coradical filtration} of a coaugmented cooperad (possibly pdg or curved) is the increasing filtration defined using iterations of the partial decompositions $\Delta_{(1)}$. Since it is coaugmented, one can remove trivial decompositions from the partial decomposition maps. An element is in the $\omega$-term of this filtration if its image by any possible composition of $\omega$ (non-trivial) partial decomposition maps is zero. A cooperad is said to be \textit{conilpotent} if any arbitrary iteration of these (non-trivial) partial decompositions maps of the cooperad ends up being trivial, or equivalently, if the coradical filtration is exhaustive. For a precise definition, see \cite[Section 1.3]{lucio2022curved} or \cite[Chapter 1, Section 4]{mathez}.

\medskip

Koszul duality relates dg operads with conilpotent coaugmented curved cooperads. For a dg operad $\mathcal{P}$ and a coaugmented conilpotent curved cooperad $\C$, the fact that they are Koszul dual is encoded by the existence of a curved twisting morphism, which needs to satisfy some extra homological properties, see \cite[Section 4]{HirshMilles12}. 

\medskip

A \textit{curved twisting morphism} $\alpha: \C \longrightarrow \mathcal{P}$ is the data of a degree $-1$ map of graded $\mathbb{S}$-modules between the coaugmentation ideal $\overline{\C}$ and $\mathcal{P}$ which satisfies the following equation: 

\[
\partial(\alpha) + \gamma_{(1)} \circ (\alpha \otimes \alpha) \circ \Delta_{(1)} = \Theta_{\mathcal{H}om}~,
\]
\vspace{0.1pc}

where $\Delta_{(1)}$ is the partial decomposition map of $\C$, $\gamma_{(1)}$ the partial composition map of $\mathcal{P}$, and 
$\Theta_{\mathcal{H}om}$ is the map given by composing the unit of $\mathcal{P}$ and the curvature of $\C$. The set of curved twisting morphisms is denoted by $\mathrm{Tw}(\C,\PP)$. The above equation can be interpreted as a Maurer--Cartan equation in a convolution curved pre-Lie algebra, see \cite[Section 6]{lucio2022curved} for more details. 

\medskip

Following \cite{grignou2019}, there is an operadic bar-cobar adjunction, 
\[
\begin{tikzcd}[column sep=5pc,row sep=3pc]
            \mathsf{curv}~\mathsf{Coop}^{\mathsf{conil}} \arrow[r, shift left=1.1ex, "\Omega"{name=F}] &\mathsf{dg}~\mathsf{Op}~.  \arrow[l, shift left=.75ex, "\mathrm{B}"{name=U}]
            \arrow[phantom, from=F, to=U, , "\dashv" rotate=-90]
\end{tikzcd}
\]

between the category of coaugmented conilpotent curved cooperads and the category of dg operads, which represents and corepresents curved twisting morphisms in the following sense: for any coaugmented conilpotent curved cooperad $\C$ and any dg operad $\mathcal{P}$, there are natural isomorphisms: 
\[
\mathrm{Hom}_{\mathsf{dg}~\mathsf{Op}}(\Omega\C,\PP) \cong \mathrm{Tw}(\C,\PP) \cong \mathrm{Hom}_{\mathsf{curv}~\mathsf{Coop}^{\mathsf{conil}}}(\C,\mathrm{B}\PP)~.
\]

This gives two \textit{universal curved twisting morphisms}, $\iota: \C \longrightarrow \Omega\C$ and $\pi: \mathrm{B}\PP \longrightarrow \PP$, induced by the identity maps. We refer to \cite[Section 4]{grignou2019} for more details on these constructions and results.

\begin{Remark}
When one restricts to \textit{augmented} dg operads, the curvature is zero on the cooperad side, and one obtains an operadic Koszul duality between augmented dg operads and conilpotent  coaugmented \textit{dg} cooperads, which is essentially equivalent to the theory developed in \cite{GinzburgKapranov94, GetzlerJones94}. 
\end{Remark}

\subsection{Classical bar-cobar adjunction between algebras and conilpotent coalgebras}
The main point of the operadic Koszul duality explained in the previous subsection is that it allows us to relate types of algebras encoded by operads and types of \textit{conilpotent} coalgebras encoded by cooperads. 

\medskip

Let us recall some constructions. To any dg (resp. pdg) $\mathbb{S}$-module one can associate an endofunctor in the category of dg (resp. pdg) modules via the \textit{Schur realization functor}:
\[
\begin{tikzcd}[column sep=4pc,row sep=0.5pc]
\mathscr{S} : \mathsf{(p)dg}~\smod \arrow[r]
&\mathsf{End}(\mathsf{(p)dg}~\mathsf{mod}) \\
M \arrow[r,mapsto]
&\mathscr{S}(M)(-) \coloneqq \displaystyle \bigoplus_{n \geq 0} M(n) \otimes_{\mathbb{S}_n} (-)^{\otimes n}~.
\end{tikzcd}
\]

The realization functor $\mathscr{S}(-)$ is strong monoidal. Thus, for any dg operad $\mathcal{P}$ its Schur functor $\mathscr{S}(\mathcal{P})$ is a monad in dg modules and for any pdg cooperad $\C$ its Schur functor $\mathscr{S}(\mathcal{C})$ is a comonad in pdg modules.

\begin{Definition}[dg $\PP$-algebra]
Let $\PP$ be an operad. A dg $\PP$-\textit{algebra} $B$ amounts to the data $(B, \gamma_B, d_B)$ of an algebra over the monad $\mathscr{S}(\PP)$. 
\end{Definition}

\begin{Definition}[pdg $\C$-coalgebra]\label{def: C-coalgebra}
Let $\C$ be a pdg cooperad. A pdg $\C$-\textit{coalgebra} $D$ amounts to the data $(D,\Delta_D;d_D)$ of a coalgebra over the comonad $\mathscr{S}(\C)$. 
\end{Definition}

When $\C$ is a \textit{curved} cooperad, one furthermore asks that the coalgebra structure on $D$ is compatible with the curvature of the cooperad, which leads to the definition of a \textit{curved} $\C$-coalgebra. 

\begin{Definition}[curved $\C$-coalgebra]
Let $\C$ be a curved cooperad. A pdg coalgebra $D$ is said to be \textit{curved} if the following diagram commutes
\[
\begin{tikzcd}[column sep=3pc,row sep=3pc]
D  \arrow[r,"\Delta_D "] \arrow[rd,"-d_D^2",swap]
&\mathscr{S}(\C)(D) \arrow[d,"\mathscr{S}(\Theta_\C)(\mathrm{id})"]\\
&D \cong \mathscr{S}(\I)(D)~.
\end{tikzcd}
\]
\end{Definition}

The data of a curved twisting morphism $\alpha: \C \longrightarrow \PP$ between a conilpotent coaugmented curved cooperad $\C$ and a dg operad $\PP$ induces a bar-cobar adjunction relative to $\alpha$
\[
\begin{tikzcd}[column sep=5pc,row sep=3pc]
          \mathsf{curv}~\mathcal{C}\text{-}\mathsf{coalg} \arrow[r, shift left=1.1ex, "\Omega_{\alpha}"{name=F}] & \mathsf{dg}~\mathcal{P}\text{-}\mathsf{alg}, \arrow[l, shift left=.75ex, "\mathrm{B}_{\alpha}"{name=U}]
            \arrow[phantom, from=F, to=U, , "\dashv" rotate=-90]
\end{tikzcd}
\]
between the category of dg $\PP$-algebras and the category of curved $\C$-coalgebras. 

\medskip

This bar-cobar adjunction has interesting homotopical properties. Usually, one considers the homotopy theory of dg $\PP$-algebras up to quasi-isomorphism, that is, the model category structure on dg $\PP$-algebras where fibrations are given by epimorphisms and weak equivalences are given by quasi-isomorphisms. Using the above adjunction, it can be transferred onto the category of curved $\C$-coalgebras and, under certain hypothesis on the curved twisting morphism $\alpha$, this Quillen adjunction becomes a Quillen equivalence. It is in particular the case for the universal curved twisting morphism $\iota: \C \longrightarrow \Omega\C$. This allows one to study the homotopy theory of dg $\PP$-algebras purely in terms of the Koszul dual coalgebras. See \cite[Section 6 and 7]{grignou2019}. 

\subsection{Complete bar-cobar adjunction between coalgebras and absolute algebras}
A dg operad $\PP$ also encodes types of \textit{coalgebras}, which are \textit{non-necessarily conilpotent}. The goal of this subsection is to explain how, given a curved twisting morphism $\alpha: \C \longrightarrow \PP$, one can also construct a \textit{another bar-cobar adjunction} between these coalgebras encoded by $\PP$ and a \textit{algebras encoded by the cooperad} $\C$. We follow mainly \cite{grignoulejay18} for these constructions. 

\subsubsection{Algebras over a cooperad}
There is a \textit{dual Schur realization functor} which is given by
\[
\begin{tikzcd}[column sep=4pc,row sep=0.5pc]
\widehat{\mathscr{S}}^c : \mathsf{(p)dg}~\smod^{\mathsf{op}} \arrow[r]
&\mathsf{End}(\mathsf{(p)dg}~\mathsf{mod}) \\
M \arrow[r,mapsto]
&\widehat{\mathscr{S}}^c(M)(-) \coloneqq \displaystyle \prod_{n \geq 0} \mathrm{Hom}_{\mathbb{S}_n}(M(n),(-)^{\otimes n})~.
\end{tikzcd}
\]

\begin{lemma}[{\cite[Corollary~3.4]{grignoulejay18}}]\label{lemma: Schur lax contravariant functor}
The dual Schur realization functor $\widehat{\mathscr{S}}^c(-)$ can be endowed with a lax monoidal structure, that is, there exists a natural transformation 
\[
\varphi_{M,N}: \widehat{\mathscr{S}}^c(M) \circ \widehat{\mathscr{S}}^c(N) \longrightarrow \widehat{\mathscr{S}}^c(M \circ N)~.
\]
which satisfies associativity and unitality compatibility conditions with respect to the monoidal structures. Furthermore, $\varphi_{M,N}$ is a degree-wise monomorphism for all pdg $\mathbb{S}$-modules $M,N$.
\end{lemma}

\begin{Remark}
The construction of this monomorphism is also explained in \cite[Chapter 1, Section 5]{mathez}.
\end{Remark}

This allows us to construct from any pdg cooperad a monad in the category of pdg modules.

\begin{Definition}[pdg $\C$-algebra]\label{def pdg C algebra}
A \textit{pdg} $\mathcal{C}$\textit{-algebra} $A$ amounts to the data $(A,\gamma_A,d_A)$ of an algebra over the monad $\widehat{\mathscr{S}}^c(\mathcal{C})$. This data is equivalent to the data of a map
\[
\gamma_A: \prod_{n \geq 0} \mathrm{Hom}_{\mathbb{S}_n}(\C(n),A^{\otimes n}) \longrightarrow A~,
\]

such that the following diagram commutes: 
\[
\begin{tikzcd}[column sep=4pc,row sep=3.5pc]
\widehat{\mathscr{S}}^c(\C) \circ \widehat{\mathscr{S}}^c(\C)(A) \arrow[d, "\widehat{\mathscr{S}}^c(\mathrm{id}_\C)(\gamma_A)",swap] \arrow[r,"\varphi_{\C,\C}(A)"]
&\widehat{\mathscr{S}}^c(\C \circ \C)(A) \arrow[r,"\widehat{\mathscr{S}}^c(\Delta)"]
&\widehat{\mathscr{S}}^c(\C)(A) \arrow[d,"\gamma_A "] \\
\widehat{\mathscr{S}}^c(\C)(A) \arrow[rr,"\gamma_A "]
&
&A~.
\end{tikzcd}
\]
\end{Definition}

\begin{Remark}
The notion of an algebra over a cooperad defines a new type of algebraic structures. The reason is that the structural map
\[
\gamma_A: \displaystyle \prod_{n \geq 0} \mathrm{Hom}_{\mathbb{S}_n}(\mathcal{C}(n),A^{\otimes n}) \longrightarrow A 
\]
gives, for any infinite series of operations in $\C$ labeled with elements of $A$, a well-defined image in $A$. Thus algebras over a cooperad are endowed with infinite summation of structural operations \textit{by definition}, without presupposing any underlying topology. 
\end{Remark}

\subsubsection{The canonical filtration on algebras over conilpotent cooperads} The notion of an algebra over a cooperad admits a further description in the case where the cooperad is \textit{conilpotent}. For the rest of this paragraph, let $\C$ be a conilpotent pdg cooperad. Each term of the coradical filtration $\mathscr{R}_\omega \C$ defines a pdg sub-cooperad, and there is a short exact sequence of pdg $\mathbb{S}$-modules 
\[
\begin{tikzcd}
0 \arrow[r]
&\mathscr{R}_\omega \C \arrow[r,"\iota_\omega",hook]
&\C \arrow[r,"\pi_\omega"]
&\C / \mathscr{R}_\omega \C \arrow[r]
&0~.
\end{tikzcd}
\]

\begin{Definition}[Canonical filtration on a pdg $\mathcal{C}$-algebra]\label{def canonical filtration}
Let $\C$ be a conilpotent pdg cooperad and let $A$ be a pdg $\C$-algebra. The \textit{canonical filtration} of $A$ is the decreasing filtration of given by 
\[ 
\mathrm{W}_\omega A \coloneqq \mathrm{Im}\left(\gamma_A \circ \widehat{\mathscr{S}}^c(\pi_\omega)(\mathrm{id}_A): \widehat{\mathscr{S}}^c(\C / \mathscr{R}_\omega \C)(A) \longrightarrow A \right)
\]
where $\mathscr{R}_\omega \C$ denotes the $\omega$-th term of the coradical filtration, for all $\omega \geq 0~.$ Notice that we have 
\[
A = \mathrm{W}_0 A \supseteq \mathrm{W}_1 A \supseteq \mathrm{W}_2 A \supseteq \cdots \supseteq \mathrm{W}_\omega A \supseteq \cdots.
\]
\end{Definition}
	
Each step $\mathrm{W}_k A$ of the canonical filtration of a pdg $\C$-algebra $A$ is an ideal, meaning the quotient $A/\mathrm{W}_k A$ has a canonical pdg $\C$-algebra structure induced by the original pdg $\C$-algebra structure of $A$. 

\begin{Definition}[Completion of a pdg $\C$-algebra]
Let $A$ be a pdg $\C$-algebra. Its \textit{completion} is given by 
\[ 
\widehat{A} \coloneqq \lim_{\omega} A/\mathrm{W}_\omega A~,
\]
where the limit is taken in the category of pdg $\C$-algebras.
\end{Definition}

It comes equipped with a canonical morphism of pdg $\C$-algebras $\varphi_A: A \longrightarrow \widehat{A}~.$ 

\begin{Proposition}[{\cite[Proposition~4.24]{grignoulejay18}}]\label{Prop: varphi is an epi}
Let $A$ be a pdg $\C$-algebra. The canonical morphism 
\[
\varphi_A: A \longrightarrow \widehat{A}
\]
is an epimorphism.
\end{Proposition}
 
\begin{Remark}\label{Remark: varphi is an epimorphism}
Conceptually, this comes from the fact that any pdg $\mathcal{C}$-algebra already carries a meaningful notion of infinite summation. Thus, "nothing needs to be added" when one applies the completion functor. On the other hand, the topology induced by the canonical filtration of a pdg $\mathcal{C}$-algebra might not be Hausdorff. Meaning that the canonical morphism $\varphi_B$ might not be a monomorphism. Therefore the action of the completion functor is to force the canonical topology to be Hausdorff. 
\end{Remark}

\begin{Definition}[Complete pdg $\mathcal{C}$-algebra]
The pdg $\C$-algebra $A$ is said to be \textit{complete} if the morphism $\varphi_A$ is an isomorphism.
\end{Definition}

\begin{Example}
Any \textit{free} pdg $\C$-algebra is complete, see \cite[Section 4]{grignoulejay18}. 
\end{Example}

\begin{Proposition}\label{prop: complete C algebras are a reflexiv subcat}
Let $\C$ be a conilpotent pdg cooperad. The category of complete pdg $\C$-algebras forms a reflective subcategory of the category of pdg $\C$-algebras, where the reflector is given by the completion functor.
\end{Proposition}

\begin{Remark}
Contrary to $\mathcal{C}$-coalgebras, which are always conilpotent (meaning they are cocomplete with respect to the canonical filtration induced by the coradical filtration of the cooperad $\C$), there are examples of $\mathcal{C}$-algebras which are not complete. See the counterexample given in \cite[Section 4.5]{grignoulejay18}.
\end{Remark}

\subsubsection{Curved algebras over curved cooperads} When one considers algebras over a curved cooperad, it is natural to restrict to the full subcategory of such algebras which are compatible with the curvature of the cooperad.

\begin{Definition}[Curved $\C$-algebra]\label{def curved alg over a coop}
Let $\C$ be a curved cooperad and let $A$ be a pdg $\C$-algebra. It is a \textit{curved} $\C$\textit{-algebra} if the following diagram commutes: 
\[
\begin{tikzcd}[column sep=4pc,row sep=3pc]
A \cong \widehat{\mathscr{S}}^c(\I)(A) \arrow[r,"\widehat{\mathscr{S}}^c(\Theta_\C)(\mathrm{id}) "] \arrow[rd,"d_A^2",swap]
&\widehat{\mathscr{S}}^c(\C)(A) \arrow[d,"\gamma_A"]\\
&A ~.
\end{tikzcd}
\]
\end{Definition}

The category of curved $\mathcal{C}$-algebras is a full sub-category of the category of pdg $\mathcal{C}$-algebras, and the inclusion functor admits a reflector.

\begin{Proposition}[{\cite[Theorem 7.5]{grignoulejay18}}]\label{Prop: curved reflective subcategory}
Let $\C$ be a curved cooperad. The category of curved $\C$-algebras is a reflective sub-category of the category of pdg $\C$-algebras. It is thus presentable.
\end{Proposition}

Notice that the condition of being \textit{conilpotent} for curved cooperads is the same as for pdg cooperads, hence for a conilpotent curved cooperad $\C$, one can consider the full sub-category of complete curved $\C$-algebras. By composing the two reflectors, it is also a reflective sub-category of the category of pdg $\C$-algebras, which is presentable. Thus it is presentable as well.

\begin{Remark}
Any map $g: \mathcal{C} \longrightarrow \mathcal{D}$ induces a morphism $\widehat{\mathscr{S}}^c(g): \widehat{\mathscr{S}}^c(\mathcal{D}) \longrightarrow \widehat{\mathscr{S}}^c(\mathcal{C})$ of monads, which in turn produces an adjunction $\mathrm{Ind}_g \dashv \mathrm{Res}_g$ between pdg $\C$-algebras and pdg $\mathcal{D}$-algebras. One can check that this adjunction can be restricted to an adjunction between complete curved $\C$-algebras and complete curved $\mathcal{D}$-algebras.
\end{Remark}

\subsubsection{Coalgebras over operads}
Even though that, for an operad $\PP$, its dual Schur functor $\widehat{\mathscr{S}}^c(\mathcal{P})$ fails to be a comonad, one can still define a notion of a coalgebra over an operad. 

\begin{Definition}[dg $\mathcal{P}$-coalgebra]
A \textit{dg} $\mathcal{P}$\textit{-coalgebra} $C$ amounts to the data $(C, \Delta_C, d_C)$ of a dg module $(C,d_C)$ endowed with a structural map
\[
\Delta_C: C \longrightarrow \displaystyle \prod_{n \geq 0} \mathrm{Hom}_{\mathbb{S}_n}(\mathcal{P}(n),C^{\otimes n})~,
\]
such that the following diagram commutes 
\[
\begin{tikzcd}[column sep=4.5pc,row sep=3pc]
C \arrow[r,"\Delta_C"] \arrow[d,"\Delta_C",swap] 
&\widehat{\mathscr{S}}^c(\PP)(C) \arrow[r,"\widehat{\mathscr{S}}^c(\mathrm{id})(\Delta_C)"]
&\widehat{\mathscr{S}}^c(\PP)(C) \circ \widehat{\mathscr{S}}^c(\PP)(C) \arrow[d,"\varphi_{\PP,\PP}(C)"] \\
\widehat{\mathscr{S}}^c(\PP)(C) \arrow[rr,"\widehat{\mathscr{S}}^c(\gamma_\PP)(\mathrm{id})"]
&
&\widehat{\mathscr{S}}^c(\PP \circ \PP)(C)~.
\end{tikzcd}
\]
\end{Definition}

\begin{Remark}
The data of a $\mathcal{P}$-coalgebra $C$ is equivalent to the data of a morphism of dg operads $\mathcal{P} \longrightarrow \mathrm{Coend}_C$, where $\mathrm{Coend}_C$ is stands for the coendomorphisms operad of $C$, given by the dg $\mathbb{S}$-module
\[
\mathrm{Coend}_C(n) \coloneqq \mathrm{Hom}\left(C, C^{\otimes n} \right)~,
\]

where the operad structure is given by the standard composition of morphisms.
\end{Remark}

\begin{theorem}[{\cite[Theorem 2.7.11]{anelcofree2014}}]
Let $\PP$ be a dg operad. The category of dg $\PP$-coalgebras is comonadic. There exists a comonad $(\mathscr{C}(\PP), \omega, \zeta)$ in the category of dg modules such that the category of $\mathscr{C}(\PP)$-coalgebras is equivalent to the category of dg $\PP$-coalgebras.
\end{theorem}

In particular, this entails the existence of a cofree dg $\PP$-coalgebra. While in the general setting of \cite{anelcofree2014}, the construction of the comonad $\mathscr{C}(\PP)$ is given by an infinite recursion, the construction of $\mathscr{C}(\PP)$ in the category of dg modules stops at the first step. 

\begin{theorem}[{\cite[Theorem 3.3.1]{anelcofree2014}}]\label{thm: existence of the cofree P cog}
The endofunctor $\mathscr{C}(\PP)$ is given by the pullback
\[
\begin{tikzcd}[column sep=3pc,row sep=3pc]
\mathscr{C}(\PP) \arrow[r,"p_2"] \arrow[d,"p_1",swap,rightarrowtail] \arrow[dr, phantom, "\ulcorner", very near start]
&\widehat{\mathscr{S}}^c(\PP) \circ \widehat{\mathscr{S}}^c(\PP) \arrow[d,"\varphi_{\PP,\PP}",rightarrowtail] \\
\widehat{\mathscr{S}}^c(\PP) \arrow[r,"\widehat{\mathscr{S}}^c(\gamma)"]
&\widehat{\mathscr{S}}^c(\PP \circ \PP)
\end{tikzcd}
\]
in the category of endofunctors. Notice that $p_1$ is a degree-wise monomorphism since $\varphi_{\PP,\PP}$ is a degree-wise monomorphism. The structural map of the comonad  
\[
\omega: \mathscr{C}(\PP) \longrightarrow \mathscr{C}(\PP) \circ \mathscr{C}(\PP)
\]
is induced by the map $p_2$ in the previous pullback. The counit of the comonad $\mathscr{C}(\PP)$ is given by
\[
\begin{tikzcd}[column sep=4pc,row sep=0.5pc]
\xi: \mathscr{C}(\PP) \arrow[r,"p_1"] 
&\widehat{\mathscr{S}}^c(\PP) \arrow[r,"\widehat{\mathscr{S}}^c(\eta)"]
&\mathrm{Id}~.
\end{tikzcd}
\]
\end{theorem}

\begin{Remark}
Let $\PP$ be a dg operad. For any dg module $V$, the cofree dg $\PP$-coalgebra on $V$ is given by $\mathscr{C}(\PP)(V)$.
\end{Remark}

\begin{Remark}
The subspace $\mathscr{C}(\PP)(V)$ of $\widehat{\mathscr{S}}^c(\PP)(V)$ admits an explicit description in terms of \textit{representative functions}. See \cite[Section 3.1]{anelcofree2014} or \cite{blockleroux} for the original reference about representative functions in the case of coassociative and cocommutative coalgebras.
\end{Remark}

\begin{Remark}
Any morphism $f: \PP \longrightarrow \mathcal{Q}$ induces a morphism of comonads $\mathscr{C}(f): \mathscr{C}(\mathcal{Q})  \longrightarrow \mathscr{C}(\mathcal{P})$, which in turn produces an adjunction $\mathrm{Res}_f \dashv \mathrm{Coind}_f$ between dg $\mathcal{P}$-coalgebras and dg $\mathcal{Q}$-coalgebras.
\end{Remark}

\subsubsection{Complete bar-cobar adjunctions and model structures}
Here, we explain the construction of the complete bar-cobar adjunction relative to a curved twisting morphism made in \cite[Section 8 to 11]{grignoulejay18}. Beware that we use different notations than those used in \textit{loc.cit.} We fix $\mathcal{P}$ a dg operad, $\mathcal{C}$ a conilpotent curved cooperad, and $\alpha: \mathcal{C} \longrightarrow \mathcal{P}$ a curved twisting morphism. 

\begin{Notation}\label{not: le dirac}
Let $f: X \longrightarrow Y$ be a map of degree $0$ and $g: X \longrightarrow Y$ be a map of degree $p$ between graded modules $X,Y$. We denote
\[
\diracComb_n(f,g) \coloneqq \sum_{i=0}^n f^{\otimes i-1} \otimes g \otimes f^{\otimes n-i} : X^{\otimes n} \longrightarrow Y^{\otimes n}
\]
the resulting $\mathbb{S}_n$-equivariant map of degree $p$. Let $M$ be a graded $\mathbb{S}$-module. It induces a map of degree $p$
\[
\prod_{n \geq 0} \mathrm{Hom}_{\mathbb{S}_n}(M(n),X^{\otimes n}) \longrightarrow \prod_{n \geq 0} \mathrm{Hom}_{\mathbb{S}_n}(M(n),Y^{\otimes n})
\]
by applying $\mathrm{Hom}(\mathrm{id},\diracComb_n(f,g))$ at each arity. By a slight abuse of notation, this map will be denoted by $\widehat{\mathscr{S}}^c(\mathrm{id})(\diracComb(f,g))~.$
\end{Notation} 

\begin{Definition}[Complete bar construction relative to $\alpha$]\label{def: complete bar}
Let $(A,\gamma_A,d_A)$ be a curved $\C$-algebra. The \textit{complete bar construction relative to} $\alpha$ of $A$ is given by
\[
\widehat{\mathrm{B}}_{\alpha} A \coloneqq (\mathscr{C}(\PP)(A), d_{\mathrm{bar}} \coloneqq d_1 + d_2)~,
\]
where $\mathscr{C}(\PP)(A)$ denotes the cofree graded $\PP$-coalgebra generated by $A$. The differential $d_{\mathrm{bar}}$ is given by the sum of two terms $d_1$ and $d_2$. The term $d_1$ is given by
\[
d_1 = \mathscr{C}(d_\PP)(\mathrm{id}) +  \mathscr{C}(\mathrm{id})(\diracComb(\mathrm{id},d_A))~.
\]
The term $d_2$ is given by the unique coderivation extending  
\[
\begin{tikzcd}[column sep=4pc,row sep=1pc]
\mathscr{C}(\PP)(A) \arrow[r,"p_1 (A)",rightarrowtail]
&\widehat{\mathscr{S}}^c(\PP)(A)\arrow[r,"\widehat{\mathscr{S}}^c(\alpha)(\mathrm{id})"]
&\widehat{\mathscr{S}}^c(\C)(A)  \arrow[r,"\gamma_A "]
&A~.
\end{tikzcd}
\]
\end{Definition}

\begin{Proposition}[{\cite[Proposition 9.2]{grignoulejay18}}]
For any curved $\C$-algebra $A$, the complete bar construction $\widehat{\mathrm{B}}_{\alpha}A$ forms a dg $\PP$-coalgebra, and it defines a functor
\[
\widehat{\mathrm{B}}_{\alpha}: \mathsf{curv}~\C\text{-}\mathsf{alg} \longrightarrow \mathsf{dg}~\PP\text{-}\mathsf{coalg}~.
\]
\end{Proposition}

\begin{Definition}[Complete cobar construction relative to $\alpha$]\label{def: complete cobar}
Let $(C,\delta_C,d_C)$ be a dg $\PP$-coalgebra. The \textit{complete cobar construction relative to} $\alpha$ of $C$ is given by
\[
\widehat{\Omega}_\alpha C \coloneqq (\widehat{\mathscr{S}}^c(\C)(C), d_{\mathrm{cobar}} \coloneqq d_1 - d_2)~,
\]
where $\widehat{\mathscr{S}}^c(\C)(C)$ denotes the free complete pdg $\C$-algebra generated by $C$. The differential $d_{\mathrm{cobar}}$ is given by the difference of two terms $d_1$ and $d_2$. The term $d_1$ is given by
\[
d_1 = -\widehat{\mathscr{S}}^c(d_\C)(\mathrm{id}) + \widehat{\mathscr{S}}^c(\mathrm{id})(\diracComb(\mathrm{id},d_C))~.
\]
The term $d_2$ is given by the unique derivation extending  
\[
\begin{tikzcd}[column sep=4pc,row sep=1pc]
C \arrow[r," \Delta_C"]
&\widehat{\mathscr{S}}^c(\PP)(C) \arrow[r,"\widehat{\mathscr{S}}^c(\alpha)(\mathrm{id})"]
&\widehat{\mathscr{S}}^c(\C)(C)~.
\end{tikzcd} 
\]
\end{Definition}

\begin{Proposition}[{\cite[Proposition 9.1]{grignoulejay18}}]
For any dg $\PP$-coalgebra $C$, the complete cobar construction $\widehat{\Omega}_{\alpha}C$ forms a complete curved $\C$-algebra, and it defines a functor
\[
\widehat{\Omega}_\alpha: \mathsf{dg}~\PP\text{-}\mathsf{coalg} \longrightarrow \mathsf{curv}~\C\text{-}\mathsf{alg}^{\mathsf{comp}}~.
\]
\end{Proposition}

\begin{Definition}[Curved twisting morphism relative to $\alpha$]
Let $(C,\delta_C,d_C)$ be a dg $\PP$-coalgebra and let $(A,\gamma_A,d_A)$ be a curved $\C$-algebra. A graded morphism 
\[
\nu: C \longrightarrow A
\]
is said to be a \textit{curved twisting morphism relative to} $\alpha$ if it satisfies the following equation 
\[
\gamma_A \cdot \widehat{\mathscr{S}}^c(\alpha)(\nu) \cdot \Delta_C + \partial(\nu) = 0~,
\]
where $\partial$ is the pre-differential on the graded module $\mathrm{hom}(C,A)$. The set of curved twisting morphisms relative to $\alpha$ are denoted by $\mathrm{Tw}^{\alpha}(C,A)$.
\end{Definition}

\begin{Remark}
The set of curved twisting morphisms relative to $\alpha$ can be encoded as the set of Maurer--Cartan elements in a convolution curved absolute $\mathcal{L}_\infty$-algebra. For more on this, see \cite[Section 4]{lucio2022integration}.
\end{Remark}

\begin{Proposition}[{\cite[Section 8]{grignoulejay18}}]
There are bijections
\[
\mathrm{Hom}_{\mathsf{curv}~\C\text{-}\mathsf{alg}^{\mathsf{comp}}}\left(\widehat{\Omega}_{\alpha}C, A\right) \cong  \mathrm{Tw}^{\alpha}(C,A)  \cong \mathrm{Hom}_{\mathsf{dg}~\PP\text{-}\mathsf{coalg}} \left(C, \widehat{\mathrm{B}}_{\alpha}A \right)~,
\]

which are natural in $C$ and $A$. 
\end{Proposition}

In particular, we get an adjunction 
\[
\begin{tikzcd}[column sep=5pc,row sep=3pc]
            \mathsf{dg}~\PP\text{-}\mathsf{coalg} \arrow[r, shift left=1.1ex, "\widehat{\Omega}_{\alpha}"{name=F}] & \mathsf{curv}~\C\text{-}\mathsf{alg}^{\mathsf{comp}} \arrow[l, shift left=.75ex, "\widehat{\mathrm{B}}_{\alpha}"{name=U}]
            \arrow[phantom, from=F, to=U, , "\dashv" rotate=-90]
\end{tikzcd}
\]
between the category of dg $\PP$-coalgebras and the category of complete curved $\C$-algebras. We call this adjunction the \textit{complete bar-cobar adjunction} relative to $\alpha$.   
 
\medskip

In some case, one can promote this adjunction into a Quillen adjunction. But first, one needs to put a model category structure on the category of dg $\PP$-coalgebras. 

\begin{theorem}[{\cite[Section 9]{grignoulejay18}}]\label{thm: model structure on P-cog}
Let $\mathcal{P}$ be a cofibrant dg operad. There is a model category structure on the category of dg $\mathcal{P}$-coalgebras, left-transferred along the cofree-forgetful adjunction
\[
\begin{tikzcd}[column sep=7pc,row sep=3pc]
\mathsf{dg}\textsf{-}\mathsf{mod} \arrow[r, shift left=1.1ex, "\mathscr{C}(\mathcal{P})(-)"{name=F}]      
&\mathsf{dg}~\PP\textsf{-}\mathsf{coalg}~, \arrow[l, shift left=.75ex, "U"{name=U}]
\arrow[phantom, from=F, to=U, , "\dashv" rotate=90]
\end{tikzcd}
\]
where 
\begin{enumerate}
\item the class of weak equivalences is given by quasi-isomorphisms,
\item the class of cofibrations is given by degree-wise monomorphisms,
\item the class of fibrations is given by right lifting property with respect to acyclic cofibrations.
\end{enumerate}
\end{theorem}

\begin{Remark}
These results do not hold without the assumption that $\PP$ is a \textit{cofibrant} dg operad. Indeed, one can show that the category of dg $ u\mathcal{C}om$-coalgebras does not admit a model structure where weak equivalences are given by quasi-isomorphisms and where cofibrations are given by degree-wise monomorphisms, where $u\mathcal{C}om$ encodes (co)unital (co)commutative (co)algebras. 
\end{Remark}

Model category structures on coalgebras over dg operads behave well with respect to quasi-isomorphisms at the operadic level.

\begin{theorem}[{\cite[Section 9]{grignoulejay18}}]
Let $f: \mathcal{P} \qi \mathcal{Q}$ be a quasi-isomorphism of cofibrant dg operads. The induced adjunction 
\[
\begin{tikzcd}[column sep=7pc,row sep=3pc]
\mathsf{dg}~\mathcal{P}\textsf{-}\mathsf{coalg} \arrow[r, shift left=1.1ex, "\mathrm{Coind}_{f}"{name=F}]      
&\mathsf{dg}~\mathcal{Q}\textsf{-}\mathsf{coalg}~, \arrow[l, shift left=.75ex, "\mathrm{Res}_{f}"{name=U}]
\arrow[phantom, from=F, to=U, , "\dashv" rotate=90]
\end{tikzcd}
\]

is a Quillen equivalence.
\end{theorem}

When this model category structure exists, it can be transferred along the complete bar-cobar adjunction relative to $\alpha$. 

\begin{theorem}[{\cite[Section 10]{grignoulejay18}}]
Let $\mathcal{P}$ be a cofibrant dg operad and let $\alpha: \C \longrightarrow \mathcal{P}$ be a curved twisting morphism. There is a model structure on the category of complete $\mathcal{C}$-algebras, right-transferred along the complete bar-cobar adjunction
\[
\begin{tikzcd}[column sep=7pc,row sep=3pc]
            \mathsf{dg}~\mathcal{P}\textsf{-}\mathsf{coalg} \arrow[r, shift left=1.1ex, "\widehat{\Omega}_{\alpha}"{name=F}] &\mathsf{curv}~\mathcal{C}\textsf{-}\mathsf{alg}^{\mathsf{comp}}~, \arrow[l, shift left=.75ex, "\widehat{\text{B}}_{\alpha}"{name=U}]
            \arrow[phantom, from=F, to=U, , "\dashv" rotate=-90]
\end{tikzcd}
\]
where 
\begin{enumerate}
\item the class of weak equivalences is given by morphisms $f$ such that $\widehat{\text{B}}_\alpha(f)$ is a quasi-isomorphism,
\item the class of fibrations is given by degree-wise epimorphisms,
\item  and the class of cofibrations is given by left lifting property with respect to acyclic fibrations.
\end{enumerate}
\end{theorem}

In the case where $\PP$ is the cobar construction $\Omega \C$ of $\mathcal{C}$ in the sense of \cite{grignou2021}, this adjunction can be promoted to a Quillen equivalence.

\begin{theorem}[{\cite[Section 11]{grignoulejay18}}]
The complete bar-cobar adjunction relative to the universal curved twisting morphism $\iota: \C \longrightarrow \Omega \C$
\[
\begin{tikzcd}[column sep=7pc,row sep=3pc]
            \mathsf{dg}~\Omega\mathcal{C}\textsf{-}\mathsf{coalg} \arrow[r, shift left=1.1ex, "\widehat{\Omega}_{\iota}"{name=F}] 
            &\mathsf{curv}~\mathcal{C}\textsf{-}\mathsf{alg}^{\mathsf{comp}}~. \arrow[l, shift left=.75ex, "\widehat{\text{B}}_{\iota}"{name=U}]
            \arrow[phantom, from=F, to=U, , "\dashv" rotate=-90]
\end{tikzcd}
\]
is a Quillen equivalence. 
\end{theorem}

\begin{Remark}
The model category structure on the category of complete curved $\mathcal{C}$-algebras transferred using the complete bar-cobar adjunction relative to the curved twisting morphism $\iota: \mathcal{C} \longrightarrow \Omega \mathcal{C}$ is called the \textit{canonical model structure}.
\end{Remark}

Furthermore, bar-cobar adjunctions are "functorial" in the following sense:

\begin{Proposition}[{\cite[Lemma 9.8]{grignoulejay18}}]\label{prop: compatibility conditions}
Let $\mathcal{P}$ and $\mathcal{Q}$ be two cofibrant dg operads, and let $\mathcal{C}$ and $\mathcal{D}$ be conilpotent curved cooperads and 

\begin{enumerate}
\item let $\alpha: \C \longrightarrow \PP$ and $\beta: \mathcal{D} \longrightarrow \mathcal{Q}$ be two curved twisting morphisms, 

\item let $f: \PP \longrightarrow \mathcal{Q}$ be a morphism of dg operads and let $g: \mathcal{C} \longrightarrow \mathcal{D}$ be a morphism of conilpotent curved cooperads, 
\end{enumerate}

such that the following diagram commutes 
\[
\begin{tikzcd}[column sep=4pc,row sep=4pc]
\C \arrow[r,"\alpha"] \arrow[d,"g",swap] 
&\mathcal{P} \arrow[d,"f"]\\
\mathcal{D} \arrow[r,"\beta"]
&\mathcal{Q}~.
\end{tikzcd}
\]

The following square 
\[
\begin{tikzcd}[column sep=5pc,row sep=5pc]
\mathsf{dg}~\mathcal{P}\text{-}\mathsf{coalg} \arrow[r,"\widehat{\Omega}_\alpha"{name=B},shift left=1.1ex] \arrow[d,"\mathrm{Coind}_f "{name=SD},shift left=1.1ex ]
&\mathsf{curv}~\mathcal{C}\text{-}\mathsf{alg}^{\mathsf{comp}} \arrow[d,"\mathrm{Res}_g"{name=LDC},shift left=1.1ex ] \arrow[l,"\widehat{\mathrm{B}}_\alpha"{name=C},,shift left=1.1ex]  \\
\mathsf{dg}~\mathcal{Q}\text{-}\mathsf{coalg} \arrow[r,"\widehat{\Omega}_\beta "{name=CC},shift left=1.1ex]  \arrow[u,"\mathrm{Res}_f"{name=LD},shift left=1.1ex ]
&\mathsf{curv}~\mathcal{D}\text{-}\mathsf{alg}^{\mathsf{comp}} \arrow[l,"\widehat{\mathrm{B}}_\beta"{name=CB},shift left=1.1ex] \arrow[u,"\mathrm{Ind}_g"{name=TD},shift left=1.1ex] \arrow[phantom, from=SD, to=LD, , "\dashv" rotate=0] \arrow[phantom, from=C, to=B, , "\dashv" rotate=-90]\arrow[phantom, from=TD, to=LDC, , "\dashv" rotate=0] \arrow[phantom, from=CC, to=CB, , "\dashv" rotate=-90]
\end{tikzcd}
\] 

of Quillen adjunctions commutes. 
\end{Proposition}

\vspace{1.5pc}

\section{Duality squares}

\vspace{1.5pc}

Let us fix a dg operad $\mathcal{P}$ and a conilpotent curved cooperad $\mathcal{C}$, together with a curved twisting morphism $\alpha: \mathcal{C} \longrightarrow \mathcal{P}$. The goal of this section is to construct two duality adjunctions that interrelate the "classical" bar-cobar constructions relative to $\alpha$ with the complete bar-cobar constructions relative to $\alpha$. This will allow us to understand the linear duality functor from a homotopy theoretical point of view.

\begin{Notation}
We use $B$ for a generic dg $\mathcal{P}$-algebra, $C$ for a generic dg $\mathcal{P}$-coalgebra, $A$ for a generic curved $\C$-algebra and $D$ for a generic curved $\C$-coalgebra.
\end{Notation}

\subsection{Sweedler functor}
The linear dual of a coalgebra over a given operad is naturally an algebra over the same operad. 

\begin{lemma}
The linear dual defines a functor
\[
\begin{tikzcd}[column sep=4pc,row sep=0pc]
\left(\mathsf{dg}~\mathcal{P}\text{-}\mathsf{coalg}\right)^{\mathsf{op}} \arrow[r,"(-)^*"] 
&\mathsf{dg}~\mathcal{P}\text{-}\mathsf{alg}~.
\end{tikzcd}
\]
\end{lemma}

\begin{proof}
There is a canonical morphism of operads $\mathrm{Coend}_C \longrightarrow \mathrm{End}_{C^*}$ which sends $f: C \longrightarrow C^{\otimes n}$ to the map 
\[
\begin{tikzcd}[column sep=4pc,row sep=0pc]
(C^*)^{\otimes n} \arrow[r,rightarrowtail] 
&(C^{\otimes n})^* \arrow[r,"f^* "] 
&C^*~. 
\end{tikzcd}
\]
Thus any dg $\mathcal{P}$-coalgebra structure $\Gamma_C: \mathcal{P} \longrightarrow \mathrm{Coend}_C$ induces a dg $\mathcal{P}$-algebra structure on $C^*$ by composing this structural morphism with the previous map. 
\end{proof}

\begin{Proposition}\label{prop: adjoint à droite}
The linear duality functor 
\[
\begin{tikzcd}[column sep=4pc,row sep=0pc]
\left(\mathsf{dg}~\mathcal{P}\text{-}\mathsf{coalg}\right)^{\mathsf{op}} \arrow[r,"(-)^*"] 
&\mathsf{dg}~\mathcal{P}\text{-}\mathsf{alg}~.
\end{tikzcd}
\]
admits a left adjoint.
\end{Proposition}

\begin{proof}
Consider the following square of functors
\[
\begin{tikzcd}[column sep=4pc,row sep=4pc]
\left(\mathsf{dg}~\mathcal{P}\text{-}\mathsf{coalg}\right)^{\mathsf{op}} \arrow[r,"(-)^*"]  \arrow[d,"\mathrm{U}^\mathsf{op}"{name=SD},shift left=1.1ex ]
&\mathsf{dg}~\mathcal{P}\text{-}\mathsf{alg} \arrow[d,"\mathrm{U}"{name=LDC},shift left=1.1ex ] \\
\mathsf{dg}~\mathsf{mod}^{\mathsf{op}} \arrow[r,"(-)^*"{name=CC},,shift left=1.1ex] \arrow[u,"\left(\mathscr{C}(\mathcal{P})(-)\right)^\mathsf{op}"{name=LD},shift left=1.1ex ] \arrow[phantom, from=SD, to=LD, , "\dashv" rotate=0]
&\mathsf{dg}~\mathsf{mod}~,  \arrow[l,"(-)^*"{name=CB},shift left=1.1ex] \arrow[u,"\mathscr{S}(\mathcal{P})(-)"{name=TD},shift left=1.1ex] \arrow[phantom, from=TD, to=LDC, , "\dashv" rotate=0] \arrow[phantom, from=CC, to=CB, , "\dashv" rotate=90]
\end{tikzcd}
\] 
where $\mathscr{S}(\mathcal{P})(-)$ is the free dg $\mathcal{P}$-algebra functor and where $\mathscr{C}(\mathcal{P})(-)$ is the cofree dg $\mathcal{P}$-coalgebra functor given by Theorem \ref{thm: existence of the cofree P cog}. The left hand side adjunction is monadic, since it is the opposite of a comonadic adjunction. All categories involved are complete and cocomplete. We also have that $(-)^* \cdot \mathrm{U}^\mathsf{op} \cong \mathrm{U} \cdot (-)^*~.$ Thus we can apply the Adjoint Lifting Theorem \cite[Theorem 2]{AdjointLifting}, which concludes the proof. 
\end{proof}

\begin{Definition}[Sweedler dual]\label{def: Sweedler dual functor}
The \textit{Sweedler duality functor}
\[
\begin{tikzcd}[column sep=4pc,row sep=0pc]
\mathsf{dg}~\mathcal{P}\text{-}\mathsf{alg} \arrow[r,"(-)^\circ"] 
&\left(\mathsf{dg}~\mathcal{P}\text{-}\mathsf{coalg}\right)^{\mathsf{op}}
\end{tikzcd}
\]
is defined as the functor left adjoint of the linear dual functor. 
\end{Definition}

\begin{Remark}
The proof of the Adjoint Lifting Theorem \cite[Theorem 2]{AdjointLifting} gives an explicit construction of this left adjoint. Let $(B,\gamma_B,d_B)$ be a dg $\mathcal{P}$-algebra. The Sweedler dual dg $\mathcal{P}$-coalgebra $B^\circ$ is given by the following equalizer: 
\[
\begin{tikzcd}[column sep=4pc,row sep=4pc]
\mathrm{Eq}\Bigg(\mathscr{C}(\mathcal{P})(B^*) \arrow[r,"(\gamma_B)^*",shift right=1.1ex,swap]  \arrow[r,"\varrho"{name=SD},shift left=1.1ex ]
&\mathscr{C}(\mathcal{P})\left((\mathscr{S}(\mathcal{P})(B))^*\right) \Bigg)~,
\end{tikzcd}
\]
where $\varrho$ is an arrow constructed using the comonadic structure of $\mathscr{C}(\mathcal{P})$ and the canonical inclusion of a dg module into its double linear dual. In particular, for any dg module $V$, we have an isomorphism:
\[
\left(\mathscr{S}(\mathcal{P})(V) \right)^\circ \cong \mathscr{C}(\mathcal{P})(V^*)~.
\]
\end{Remark}

\begin{Proposition}\label{prop: natural mono for Sweedler dual}
There is a natural monomorphism
\[
\epsilon: \mathrm{U}^{\mathsf{op}} \cdot (-)^\circ \rightarrowtail (-)^* \cdot \mathrm{U}~,
\]
which implies that the Sweedler dual is a sub-dg module of the linear dual functor. 
\end{Proposition}

\begin{proof}
Let $V$ be a dg module, there is a monomorphism
\[
\begin{tikzcd}[column sep=4pc,row sep=4pc]
\epsilon_V: \left(\mathscr{S}(\mathcal{P})(V) \right)^\circ \cong \mathscr{C}(\mathcal{P})(V^*) \arrow[r,rightarrowtail,"(p_1)_{V^*}"]
&\widehat{\mathscr{S}}^c(\mathcal{P})(V^*) \arrow[r,rightarrowtail]
&\left(\mathscr{S}(\mathcal{P})(V) \right)^*~,
\end{tikzcd}
\]
where the monomorphism $p_1$ is given by Theorem \ref{thm: existence of the cofree P cog} and the second monomorphism is induced by the natural monomorphism $(V^*)^{\otimes n} \rightarrowtail (V^{\otimes n})^*$, hence the proposition is true on free dg $\mathcal{P}$-algebras. Any dg $\mathcal{P}$-algebra can be written as a $\mathrm{U}$-split coequalizer of free dg $\mathcal{P}$-algebras. Both $\mathrm{U}^{\mathsf{op}} \cdot (-)^\circ$ and $(-)^* \cdot \mathrm{U}$ send these split coequalizers to split equalizers of dg modules, therefore the monomorphism $\epsilon$ extends to all dg $\mathcal{P}$-algebras.
\end{proof}

\begin{Proposition}\label{prop: iso de Beck-Chevalley}
Let $B$ be a dg $\mathcal{P}$-algebra which is degree-wise finite dimensional and bounded above or bounded below. There exists a canonical dg $\mathcal{P}$-coalgebra on $B^*$ and furthermore the natural map 
\[
\epsilon: \mathrm{U}^{\mathsf{op}} \cdot B^\circ \rightarrowtail B^* \cdot \mathrm{U}~
\]
is an isomorphism of dg modules.
\end{Proposition}

\begin{proof}
Under those hypothesis, the natural inclusion $(B^*)^{\otimes n} \rightarrowtail (B^{\otimes n})^*$ is an isomorphism, hence the natural map $\mathrm{Coend}_B \longrightarrow \mathrm{End}_{B^*}$ is an isomorphism. Therefore, the linear dual $B^*$ admits a dg $\mathcal{P}$-coalgebra structure and it can be checked that it satisfies the universal property of the Sweedler dual, hence the map of Proposition \ref{prop: natural mono for Sweedler dual} is indeed an isomorphism. 
\end{proof}

\begin{Remark}[Beck--Chevalley condition]\label{Rmk: Beck-Chevalley}
The adjunction $(-)^\circ \dashv (-)^*$ restricts to an anti-equivalence of categories between the category of total finite dimensional dg $\mathcal{P}$-algebra and the category of total finite dimensional dg $\mathcal{P}$-coalgebras, and more generally, between degree-wise finite dimensional bounded-above (resp. bounded-below) algebras and bounded-below (resp. bounded-above) coalgebras. 
\end{Remark}

\begin{Example}
Consider the operad $\mathcal{A}ss$ which encodes dg associative algebras as its algebras and dg coassociative coalgebras as its coalgebras. Then the adjunction 
\[
\begin{tikzcd}[column sep=7pc,row sep=3pc]
\mathsf{dg}~\mathcal{A}ss\text{-}\mathsf{alg}  \arrow[r, shift left=1.1ex, "(-)^\circ"{name=F}] 
& \left(\mathsf{dg}~\mathcal{A}ss\text{-}\mathsf{coalg}\right)^{\mathsf{op}} ~. \arrow[l, shift left=.75ex, "(-)^*"{name=U}] \arrow[phantom, from=F, to=U, , "\dashv" rotate=-90]
\end{tikzcd}
\]

coincides with the original Sweedler adjunction constructed in \cite{Sweedler69}.
\end{Example}

\subsection{Topological dual functor}
Let's turn to the other side of the Koszul duality, where $\mathcal{C}$ is a conilpotent curved cooperad.

\begin{lemma}\label{lemma: dual lineaire est absolute}
The linear duality defines a functor
\[
\begin{tikzcd}[column sep=4pc,row sep=0pc]
\left(\mathsf{pdg}~\mathcal{C}\text{-}\mathsf{coalg}\right)^{\mathsf{op}} \arrow[r,"(-)^*"] 
&\mathsf{pdg}~\mathcal{C}\text{-}\mathsf{alg}^{\mathsf{comp}}
\end{tikzcd}
\]
from the category of pdg $\mathcal{C}$-coalgebras to the category of complete pdg $\mathcal{C}$-algebras.
\end{lemma}

\begin{proof}
Let $(D,\Delta_D, d_D)$ be a pdg $\mathcal{C}$-coalgebra, where 
\[
\Delta_D: D \longrightarrow \bigoplus_{n \geq 0} \mathcal{C}(n) \otimes_{\mathbb{S}_n} D^{\otimes n} 
\]
is the structural morphism. By applying linear duality, we get a map
\[
\begin{tikzcd}[column sep=3.5pc,row sep=0pc]
\gamma_{D^*}: \displaystyle \prod_{n \geq 0} \mathrm{Hom}_{\mathbb{S}_n}\left(\mathcal{C}(n), (D^*)^{\otimes n} \right) \arrow[r,rightarrowtail]
&\displaystyle \prod_{n \geq 0} \mathrm{Hom}_{\mathbb{S}_n}\left(\mathcal{C}(n), (D^{\otimes n})^* \right) \arrow[r,"(\Delta_D)^* "]
&D^*~.
\end{tikzcd}
\]
One can check that it defines a pdg $\mathcal{C}$-algebra structure on $D^*$. Furthermore, let 
\[
\mathrm{F}_\omega D \coloneqq \mathrm{Ker}\left(\Delta_D^\omega: D \longrightarrow \bigoplus_{n \geq 0} \mathcal{C}/\mathscr{R}_\omega\mathcal{C}(n) \otimes_{\mathbb{S}_n} D^{\otimes n} \right)
\] 
be the canonical coradical filtration on $D$ induced by the coradical filtration on $\mathcal{C}$. Since $\mathcal{C}$ is conilpotent, the coradical filtration of any dg $\mathcal{C}$-coalgebra is exhaustive, therefore 
\[
D \cong \colim_{\omega}\mathrm{F}_\omega D~,
\]
which in turn implies that
\[
D^* \cong \lim_{\omega} ~ (\mathrm{F}_\omega D)^*~.
\]
One can check that $(\mathrm{F}_\omega D)^* \cong D^*/\mathrm{W}_\omega D^*$, therefore the image of the linear duality functor $(-)^*$ lies in the sub-category of complete pdg $\mathcal{C}$-coalgebras.
\end{proof}

\begin{Example}
Lemma \ref{lemma: dual lineaire est absolute} provides us with a wealth of examples of absolute algebras. Indeed, it shows that every time one takes the linear dual of some type of conilpotent coalgebras, one gets some type of absolute algebras.

\medskip

For instance, if one considers the linear dual of the global sections $\mathcal{O}_G(G)^*$, where $\mathcal{O}_G$ is the structural sheaf of a (possibly pro)-unipotent algebraic group $G$, then one gets an absolute associative algebra with a compatible coalgebraic structure. 
\end{Example}

\begin{Proposition}\label{prop: adjoint à gauche dual topo}
The linear duality functor
\[
\begin{tikzcd}[column sep=4pc,row sep=0pc]
\left(\mathsf{pdg}~\mathcal{C}\text{-}\mathsf{coalg}\right)^{\mathsf{op}} \arrow[r,"(-)^*"] 
&\mathsf{pdg}~\mathcal{C}\text{-}\mathsf{alg}^{\mathsf{comp}}
\end{tikzcd}
\]
admits a left adjoint. 
\end{Proposition}

\begin{proof}
We consider the following square of functors
\[
\begin{tikzcd}[column sep=4pc,row sep=4pc]
\left(\mathsf{pdg}~\mathcal{C}\text{-}\mathsf{coalg}\right)^{\mathsf{op}} \arrow[r,"(-)^*"]  \arrow[d,"\mathrm{U}^\mathsf{op}"{name=SD},shift left=1.1ex ]
&\mathsf{pdg}~\mathcal{C}\text{-}\mathsf{alg}^{\mathsf{comp}} \arrow[d,"\mathrm{U}"{name=LDC},shift left=1.1ex ] \\
\mathsf{pdg}~\mathsf{mod}^{\mathsf{op}} \arrow[r,"(-)^*"{name=CC},,shift left=1.1ex] \arrow[u,"\left(\mathscr{S}(\mathcal{C})(-)\right)^\mathsf{op}"{name=LD},shift left=1.1ex ] \arrow[phantom, from=SD, to=LD, , "\dashv" rotate=0]
&\mathsf{pdg}~\mathsf{mod}~,  \arrow[l,"(-)^*"{name=CB},shift left=1.1ex] \arrow[u,"\widehat{\mathscr{S}}^c(\mathcal{C})(-)"{name=TD},shift left=1.1ex] \arrow[phantom, from=TD, to=LDC, , "\dashv" rotate=0] \arrow[phantom, from=CC, to=CB, , "\dashv" rotate=90]
\end{tikzcd}
\] 
where $\mathscr{S}(\mathcal{C})(-)$ is the cofree pdg $\mathcal{C}$-coalgebra functor and where $\widehat{\mathscr{S}}^c(\mathcal{C})(-)$ is the free pdg $\mathcal{C}$-algebra functor, which is always complete. Again, vertical adjunctions are monadic and it is clear that $(-)^* \cdot \mathrm{U}^\mathsf{op} \cong \mathrm{U} \cdot (-)^*~.$ Thus we can apply the Adjoint Lifting Theorem \cite[Theorem 2]{AdjointLifting}, which concludes the proof. 
\end{proof}

\begin{Definition}[Topological dual functor]\label{def: topological dual functor}
The \textit{topological dual functor} 
\[
\begin{tikzcd}[column sep=4pc,row sep=0pc]
\mathsf{pdg}~\mathcal{C}\text{-}\mathsf{alg}^{\mathsf{comp}} \arrow[r,"(-)^\vee"]
&\left(\mathsf{pdg}~\mathcal{C}\text{-}\mathsf{coalg}\right)^{\mathsf{op}}
\end{tikzcd}
\]
is defined as the functor left adjoint to the linear dual functor.
\end{Definition}

\begin{Example}
Let us explain the name of this functor via an example. Let us consider a particular case of this adjunction 
\[
\begin{tikzcd}[column sep=7pc,row sep=3pc]
           \mathsf{abs}~\mathsf{assoc}\text{-}\mathsf{alg}^{\mathsf{comp}} \arrow[r,"(-)^\vee "{name=F}, shift left=1.1ex] 
           &\left(\mathsf{coassoc}\text{-}\mathsf{coalg}^{\mathsf{conil}}\right)^{\mathsf{op}}~, \arrow[l, shift left=.75ex, "(-)^*"{name=U}]
            \arrow[phantom, from=F, to=U, , "\dashv" rotate=-90]
\end{tikzcd}
\]
between conilpotent non-counital coassociative coalgebras and complete absolute associative algebras. See Subsection \ref{subsection: absolute associative} for more details about complete absolute associative algebras. Let $A$ be a \textit{finitely generated} complete absolute associative algebra, meaning $A/W_1A$ is finite dimensional over $\kk$. Then, for any $\omega \geq 1$, $A/W_\omega A$ is a finite dimensional nilpotent $\kk$-algebra and the topological dual of $A$ is given by 
\[
A^\vee = \colim_{\omega} \left(A/W_\omega A\right)^*~.
\]
In fact, this adjunction is an equivalence between finitely cogenerated conilpotent non-counital coassociative coalgebras and finitely generated complete absolute associative algebras.
\end{Example}

\begin{Remark}\label{Rmk: formule pour le dual topologique}
Given a complete pdg $\mathcal{C}$-algebra $(A, \gamma_A, d_A)$, its topological dual $A^\vee$ is given by the following equalizer:
\[
\begin{tikzcd}[column sep=4pc,row sep=4pc]
\mathrm{Eq}\Bigg(\mathscr{S}(\mathcal{C})(A^*) \arrow[r,"(\gamma_A)^*",shift right=1.1ex,swap]  \arrow[r,"\varrho"{name=SD},shift left=1.1ex ]
&\mathscr{S}(\mathcal{C})\left((\widehat{\mathscr{S}}^c(\mathcal{C})(A))^*\right) \Bigg)~,
\end{tikzcd}
\]
where $\varrho$ is an arrow constructed using the comonadic structure of $\mathscr{S}(\mathcal{C})(-)$ and the canonical inclusion of a pdg module into its double linear dual. In particular, for any pdg module $V$, we have an isomorphism: 
\[
\left(\widehat{\mathscr{S}}^c(\mathcal{C})(V)\right)^\vee \cong \mathscr{S}(\mathcal{C})(V^*)~.
\]
\end{Remark}

\begin{Proposition}\label{prop: natural mono for topo dual}
There is a natural monomorphism
\[
\epsilon: \mathrm{U}^{\mathsf{op}} \cdot (-)^\vee \rightarrowtail (-)^* \cdot \mathrm{U}~,
\]
which implies that the topological dual is a sub-pdg module of the linear dual functor. 
\end{Proposition}

\begin{proof}
This map is given by the mate of the natural isomorphism
\[
\epsilon^{\dagger}: \mathscr{S}(\mathcal{C}) \cdot (-)^* \longrightarrow (-)^\vee \cdot \widehat{\mathscr{S}}^c(\mathcal{C})
\]
which makes the square of adjunctions in the proof of Proposition \ref{prop: adjoint à gauche dual topo} commute. It suffices to check that $\epsilon$ is a monomorphism on free complete pdg $\mathcal{C}$-algebras. Indeed, since any pdg $\mathcal{C}$-algebra can be written as a $\mathrm{U}$-split coequalizer of free pdg $\mathcal{C}$-algebras and since both composites $\mathrm{U}^{\mathsf{op}} \cdot (-)^\vee$ and $(-)^* \cdot \mathrm{U}$ send these split coequalizers to split equalizers of pdg modules, then the map will be a monomorphism for all pdg $\C$-algebras. 

\medskip

Furthermore, we may assume that the conilpotent curved cooperad $\C$ is free as a graded $\mathbb{S}$-module, meaning there is an isomorphism $\C(n) \cong \C^{\mathrm{gr}}(n) \otimes \Bbbk[\mathbb{S}_n]$ of graded $\mathbb{S}_n$-modules for all $n \geq 0$. This is because over a characteristic zero field, there is always a free graded $\mathbb{S}$-module $M$ together with a monomorphism $\phi_\C: \C \rightarrowtail M$, by the fact that every graded $\mathbb{S}$-module is arity-wise projective. And since the following square 
\[
\begin{tikzcd}[column sep=3pc,row sep=3pc]
\mathscr{S}(\mathcal{C})(V^*) \arrow[d,"\mathscr{S}(\phi_\C)(\mathrm{id})",rightarrowtail,swap] \arrow[r]
&\left(\widehat{\mathscr{S}}^c(\mathcal{C})(V)\right)^* \arrow[d,"\left(\widehat{\mathscr{S}}^c(\phi_\C)(V)\right)^* ",rightarrowtail] \\
\mathscr{S}(M)(V^*) \arrow[r]
&\left(\widehat{\mathscr{S}}^c(M)(V)\right)^*
\end{tikzcd}
\]
commutes, if the bottom horizontal map is a monomorphism so is the top horizontal map. 

\medskip

So let's assume that $\C$ is free as a graded $\mathbb{S}$-module, the map $\epsilon_\C$ can be written down as the following composition
\[
\begin{tikzcd}[column sep=2.5pc,row sep=3pc]
\displaystyle \bigoplus_{n \geq 0} \C^{\mathrm{gr}}(n) \otimes_{\kk} (V^*)^{\otimes n} \arrow[r,"\epsilon_1"]
&\displaystyle \bigoplus_{n \geq 0} \mathrm{Hom}_{\kk}\left( \C^{\mathrm{gr}}(n),V^{\otimes n}\right)^* \arrow[r,"\epsilon_2"]
&\displaystyle \left(\prod_{n \geq 0} \mathrm{Hom}_{\kk}\left(\C^{\mathrm{gr}}(n),V^{\otimes n}\right)\right)^*~, 
\end{tikzcd}
\]
where the first map $\epsilon_1$ is obtained as the adjoint of the following composition of evaluation maps 
\[
\begin{tikzcd}[column sep=2.5pc,row sep=3pc]
\displaystyle \C^{\mathrm{gr}}(n) \otimes_{\kk} (V^*)^{\otimes n} \otimes_{\kk} \mathrm{Hom}_{\kk}\left( \C^{\mathrm{gr}}(n),V^{\otimes n}\right)^* \arrow[r]
&(V^*)^{\otimes n} \otimes_{\kk} V^{\otimes n}\arrow[r]
&\kk ~, 
\end{tikzcd}
\]
along the tensor-hom adjunction and where the second map $\epsilon_2$ is the canonical inclusion of the direct sum of linear duals to the linear dual of the product. The map $\epsilon_2$ is a monomorphism since any linear functional which is sent to zero must it self be zero, as it only acts in one component of the infinite product. Let us fix some $n\geq 0$ and let us take a linear basis for $\C^{\mathrm{gr}}(n)$, writing it as $\C^{\mathrm{gr}}(n) \cong \bigoplus_{i \in I} \kk$, then the map $\epsilon_1$ in the arity $n$ component is given by 
\[
\begin{tikzcd}[column sep=2.5pc,row sep=3pc]
\displaystyle \bigoplus_{i \in I} (V^*)^{\otimes n} \arrow[r]
&\displaystyle \prod_{i \in I} \left(V^{\otimes n}\right)^*
\end{tikzcd} 
\]
which is itself the composite of the canonical inclusion of the direct sum into the product and of the tensor product of the linear duals into the linear duals of the tensor products, which are both monomorphisms. Hence $\epsilon_1$ is also a monomorphism and the result follows. 
\end{proof}

\begin{Proposition}\label{prop: restriction aux curved}
The adjunction 
\[
\begin{tikzcd}[column sep=7pc,row sep=3pc]
           \mathsf{pdg}~\mathcal{C}\text{-}\mathsf{alg}^{\mathsf{comp}} \arrow[r,"(-)^\vee "{name=F}, shift left=1.1ex] 
           &\left(\mathsf{pdg}~\mathcal{C}\text{-}\mathsf{coalg}\right)^{\mathsf{op}}~, \arrow[l, shift left=.75ex, "(-)^*"{name=U}]
            \arrow[phantom, from=F, to=U, , "\dashv" rotate=-90]
\end{tikzcd}
\]
restricts to an adjunction 
\[
\begin{tikzcd}[column sep=7pc,row sep=3pc]
           \mathsf{curv}~\mathcal{C}\text{-}\mathsf{alg}^{\mathsf{comp}} \arrow[r,"(-)^\vee "{name=F}, shift left=1.1ex] 
           &\left(\mathsf{curv}~\mathcal{C}\text{-}\mathsf{coalg}\right)^{\mathsf{op}}~. \arrow[l, shift left=.75ex, "(-)^*"{name=U}]
            \arrow[phantom, from=F, to=U, , "\dashv" rotate=-90]
\end{tikzcd}
\]
\end{Proposition}

\begin{proof}
Denote $\Theta_\mathcal{C}: \C \longrightarrow \I$ the curvature of $\C$. Let $(D,\Delta_D,d_D)$ be a pdg $\mathcal{C}$-coalgebra. Recall that it is curved if the following diagram commutes
\[
\begin{tikzcd}[column sep=3pc,row sep=3pc]
D \arrow[r,"\Delta_D"] \arrow[rd,"-d_D^2",swap] 
&\mathscr{S}(\mathcal{C})(D) \arrow[d,"\mathscr{S}(\Theta_\mathcal{C})(\mathrm{id})"] \\
&D \cong \mathscr{S}(\I)(D)~.
\end{tikzcd}
\]
Therefore we have 
\[
\begin{tikzcd}[column sep=3pc,row sep=3pc]
\widehat{\mathscr{S}}^c(\mathcal{C})(D^*) \arrow[r,rightarrowtail] 
&\left(\mathscr{S}(\mathcal{C})(D)\right)^* \arrow[r,"(\Delta_D)^* "] 
&D^* \\
\widehat{\mathscr{S}}^c(\I)(D^*) \arrow[r,"\cong"] \arrow[u,"\widehat{\mathscr{S}}^c(\Theta_\mathcal{C})(\mathrm{id})"]
&\left(\mathscr{S}(\I)(D)\right)^* \cong D^*~. \arrow[ru,"(-d_D^2)^*",swap] \arrow[u,"(\mathscr{S}(\Theta_\mathcal{C})(\mathrm{id}))^*"]
\end{tikzcd}
\]
The left square commutes by naturality of the inclusion. The right triangle commutes since it is the image of a commutative triangle by the functor $(-)^*$. Thus the big diagram commutes. Finally, one can observe that $(-d_D^2)^* = d_{D^*}^2$. Therefore $D^*$ is indeed a complete curved $\mathcal{C}$-algebra. 

\medskip 

Let $(A, \gamma_A, d_A)$ be a complete curved $\mathcal{C}$-algebra. Consider the following diagram
\[
\begin{tikzcd}[column sep=3pc,row sep=3pc]
\left(\widehat{\mathscr{S}}^c(\mathcal{C})(A)\right)^* \arrow[dr,"\left(\widehat{\mathscr{S}}^c(\Theta_\mathcal{C})(\mathrm{id})\right)^* " {xshift= 0.35pc}]
&
&A^* \arrow[ll,"(\gamma_A)^* ",swap] \\
&\left(\widehat{\mathscr{S}}^c(\I)(A)\right)^* \cong A^* \arrow[ru,"(d_A^2)^* "]
& \\
\mathscr{S}(\C)(A^\vee) \arrow[uu,rightarrowtail] \arrow[dr,"\mathscr{S}(\Theta_\mathcal{C})(\mathrm{id}_{A^\vee})"]
&
&A^\vee  \arrow[ll, "\Delta_{A^\vee}" {xshift= -0.7pc},swap] \arrow[uu,rightarrowtail]  \\
&\mathscr{S}(\I)(A^\vee) \cong A^\vee~, \arrow[ru,"-d_{A^\vee}^2 "]  \arrow[uu,rightarrowtail, crossing over]
\end{tikzcd}
\]
in the category of pdg modules, where the vertical arrows are given by Proposition \ref{prop: natural mono for topo dual}. The top triangle commutes since it is the image of a commuting triangle via the functor $(-)^*$. Each of the vertical faces commutes as well, where $d_{A^\vee}^2$ is simply given by $(d_A)^\vee \circ (d_A)^\vee$. Since every vertical map is a monomorphism, the bottom triangle also commutes and thus $A^\vee$ is indeed a curved $\mathcal{C}$-coalgebra.
\end{proof}

\subsection{The algebraic duality square} Using the two duality adjunctions constructed so far, one can interrelate the "standard" bar-cobar adjunction with the complete bar-cobar adjunction in a square of commuting adjunctions.

\begin{theorem}[Duality square]\label{thm: magical square}
The square of adjunctions 
\[
\begin{tikzcd}[column sep=5pc,row sep=5pc]
\left(\mathsf{dg}~\mathcal{P}\text{-}\mathsf{alg}\right)^{\mathsf{op}} \arrow[r,"\mathrm{B}_\alpha^{\mathsf{op}}"{name=B},shift left=1.1ex] \arrow[d,"(-)^\circ "{name=SD},shift left=1.1ex ]
&\left(\mathsf{curv}~\mathcal{C}\text{-}\mathsf{coalg}\right)^{\mathsf{op}} \arrow[d,"(-)^*"{name=LDC},shift left=1.1ex ] \arrow[l,"\Omega_\alpha^{\mathsf{op}}"{name=C},,shift left=1.1ex]  \\
\mathsf{dg}~\mathcal{P}\text{-}\mathsf{coalg} \arrow[r,"\widehat{\Omega}_\alpha "{name=CC},shift left=1.1ex]  \arrow[u,"(-)^*"{name=LD},shift left=1.1ex ]
&\mathsf{curv}~\mathcal{C}\text{-}\mathsf{alg}^{\mathsf{comp}}~, \arrow[l,"\widehat{\mathrm{B}}_\alpha"{name=CB},shift left=1.1ex] \arrow[u,"(-)^\vee"{name=TD},shift left=1.1ex] \arrow[phantom, from=SD, to=LD, , "\dashv" rotate=0] \arrow[phantom, from=C, to=B, , "\dashv" rotate=-90]\arrow[phantom, from=TD, to=LDC, , "\dashv" rotate=0] \arrow[phantom, from=CC, to=CB, , "\dashv" rotate=-90]
\end{tikzcd}
\] 
commutes in the following sense: right adjoints going from the top right to the bottom left are naturally isomorphic.
\end{theorem}

\begin{proof}
Let $V$ be a graded module. There is an isomorphism
\[
\left(\mathscr{S}(\mathcal{P})(V)\right)^\circ \cong \mathscr{C}(\PP)(V^*)
\]
of graded $\PP$-coalgebras. Let $(D,\Delta_D,d_D)$ be a curved $\mathcal{C}$-coalgebra, let us check the above isomorphism commutes with the differential and thus that we have an isomorphism
\[
\left(\Omega_\alpha D\right)^\circ \cong \widehat{\mathrm{B}}_\alpha(D^*)
\]
of dg $\PP$-coalgebras natural in $D$. For that, let us compute the differential on $\left(\Omega_\alpha D\right)^\circ$. Since its underlying graded $\PP$-coalgebra is cofree, it suffices to check that the projection of this differential onto the cogenerators agrees with the projection of the complete bar construction on $D^*$. This in turn follows from the commutativity of the following diagram 
\[
\begin{tikzcd}[column sep=3pc,row sep=3pc]
\left(\mathscr{S}(\mathcal{P})(D)\right)^* \arrow[r,"(\mathscr{S}(\alpha))^*"] 
&\left(\mathscr{S}(\mathcal{C})(D)\right)^* \arrow[r,"(\Delta_D)^*"]
&D^* \\
\widehat{\mathscr{S}}^c(\mathcal{P})(D^*) \arrow[u,rightarrowtail] \arrow[r,"\widehat{\mathscr{S}}^c(\alpha)"]
&\widehat{\mathscr{S}}^c(\mathcal{C})(D^*) \arrow[u,rightarrowtail]  \arrow[r,"\gamma_{D^*}"]
&D^*  \arrow[u,equal] \\
\mathscr{C}(\PP)(D^*) \arrow[rr] \arrow[u,rightarrowtail]
& 
&D^* \arrow[u,equal]~,
\end{tikzcd}
\]
where the lower smaller square determines the projection of the differential of the complete bar construction and where the big outer square determines the projection of the differential of the Sweedler dual of the cobar construction, since the Sweedler dual is a sub-dg-module of the linear dual by Proposition \ref{prop: natural mono for Sweedler dual}. 
\end{proof}

\begin{Remark}
Let $C$ be a dg $\PP$-coalgebra. There is a natural isomorphism
\[
\mathrm{B}_\alpha C^* \cong \left(\widehat{\Omega}_\alpha C \right)^\vee
\]
of dg $\mathcal{C}$-coalgebras, which follows from the natural isomorphism between their left adjoints given in Theorem \ref{thm: magical square}. 
\end{Remark}

\begin{Proposition}
There is a natural monomorphism
\[
\zeta: \widehat{\Omega}_\alpha \cdot (-)^\circ \longrightarrow  (-)^* \cdot \mathrm{B}_\alpha~.
\]
of complete curved $\C$-algebras.
\end{Proposition}

\begin{proof}
This existence of this natural transformation is given by the mate of the natural isomorphism constructed in Theorem \ref{thm: magical square}. It is easily seen to be a monomorphism since on the level of graded modules, it can be identified with the following composition of natural inclusion 
\[
\widehat{\mathscr{S}}^c(\C)(B^\circ) \rightarrowtail \widehat{\mathscr{S}}^c(\C)(B^*) \rightarrowtail \left(\mathscr{S}(\mathcal{C})(B)\right)^*~,
\]
for any dg $\PP$-algebra $B$. 
\end{proof}

\begin{Proposition}\label{prop: finite dual commutes}
Let $B$ be a dg $\PP$-algebra degree-wise finite dimensional and bounded above or bounded below. The natural map 
\[
\zeta_B: \widehat{\Omega}_\alpha B^\circ \cong \left(\mathrm{B}_\alpha B\right)^*
\]
is an isomorphism of complete curved $\mathcal{C}$-algebras. 
\end{Proposition}

\begin{proof}
It can be directly checked at the level of graded modules that under these finiteness conditions, the map $\zeta_B$ is indeed an isomorphism. This essentially follows from Proposition \ref{prop: iso de Beck-Chevalley} and from the natural isomorphism $(B^{\otimes n})^* \cong (B^*)^{\otimes n}$. 
\end{proof}

\begin{Remark}
The subcategory of dg $\PP$-algebras which satisfy the \textit{Beck-Chevalley condition} with respect to the duality square of Theorem \ref{thm: magical square} contains the subcategory of degree-wise finite dimensional and bounded above or bounded below dg $\PP$-algebras. See also Remark \ref{Rmk: Beck-Chevalley}. 
\end{Remark}

\subsection{Homotopical duality square} The duality square of Theorem \ref{thm: magical square} behaves well with respect to model structures. We now restrict to the case where $\PP = \Omega \C$, where the curved twisting morphism considered is the canonical curved twisting morphism $\iota: \C \longrightarrow \Omega \C$, and where $\C$ is a conilpotent curved cooperad. This is in order to ensure the existence of a model category structure on the category of dg $\PP$-coalgebras, as in Theorem \ref{thm: model structure on P-cog}.

\begin{lemma}\label{lemma: Sweedler dual is a Quillen adjunction}
The adjunction 
\[
\begin{tikzcd}[column sep=7pc,row sep=3pc]
            \mathsf{dg}~\Omega \C\text{-}\mathsf{coalg} \arrow[r, shift left=1.1ex, "(-)^*"{name=F}] &\left(\mathsf{dg}~\Omega \C\text{-}\mathsf{alg}\right)^{\mathsf{op}} ~. \arrow[l, shift left=.75ex, "(-)^\circ"{name=U}]
            \arrow[phantom, from=F, to=U, , "\dashv" rotate=-90]
\end{tikzcd}
\]
is a Quillen adjunction. 
\end{lemma}

\begin{proof}
The left adjoint $(-)^*$ sends degree-wise monomorphisms to degree-wise epimorphisms. Thus it preserves cofibrations. It also preserves quasi-isomorphisms. Therefore we have a Quillen adjunction.
\end{proof}

\begin{theorem}[Homotopical properties of the duality square]\label{thm: homotopical magical square}
All the adjunctions in the square 
\[
\begin{tikzcd}[column sep=5pc,row sep=5pc]
\left(\mathsf{dg}~\Omega \C\text{-}\mathsf{alg}\right)^{\mathsf{op}} \arrow[r,"\mathrm{B}_\iota^{\mathsf{op}}"{name=B},shift left=1.1ex] \arrow[d,"(-)^\circ "{name=SD},shift left=1.1ex ]
&\left(\mathsf{curv}~\mathcal{C}\text{-}\mathsf{coalg}\right)^{\mathsf{op}} \arrow[d,"(-)^*"{name=LDC},shift left=1.1ex ] \arrow[l,"\Omega_\iota^{\mathsf{op}}"{name=C},,shift left=1.1ex]  \\
\mathsf{dg}~\Omega \C\text{-}\mathsf{coalg} \arrow[r,"\widehat{\Omega}_\iota "{name=CC},shift left=1.1ex]  \arrow[u,"(-)^*"{name=LD},shift left=1.1ex ]
&\mathsf{curv}~\mathcal{C}\text{-}\mathsf{alg}^{\mathsf{comp}}~, \arrow[l,"\widehat{\mathrm{B}}_\iota"{name=CB},shift left=1.1ex] \arrow[u,"(-)^\vee"{name=TD},shift left=1.1ex] \arrow[phantom, from=SD, to=LD, , "\dashv" rotate=0] \arrow[phantom, from=C, to=B, , "\dashv" rotate=-90]\arrow[phantom, from=TD, to=LDC, , "\dashv" rotate=0] \arrow[phantom, from=CC, to=CB, , "\dashv" rotate=-90]
\end{tikzcd}
\] 
are Quillen adjunctions.
\end{theorem}

\begin{proof}
The only thing left to check is that the adjunction 
\[
\begin{tikzcd}[column sep=7pc,row sep=3pc]
           \mathsf{curv}~\mathcal{C}\text{-}\mathsf{alg}^{\mathsf{comp}} \arrow[r,"(-)^\vee "{name=F}, shift left=1.1ex] 
           &\left(\mathsf{curv}~\mathcal{C}\text{-}\mathsf{coalg}\right)^{\mathsf{op}}~. \arrow[l, shift left=.75ex, "(-)^*"{name=U}]
            \arrow[phantom, from=F, to=U, , "\dashv" rotate=-90]
\end{tikzcd}
\]
is indeed a Quillen adjunction, where the model structure considered on curved $\C$-coalgebras is obtained by transfer along the adjunction $\Omega_\iota \dashv \mathrm{B}_\iota$ and where the model structure considered on the category of complete curved $\C$-algebras is obtained by transfer along the adjunction $\widehat{\Omega}_\iota \dashv \widehat{\mathrm{B}}_\iota$.

\medskip

Let us check that $(-)^*$ is a right Quillen functor. It sends monomorphisms to epimorphisms, thus preserves fibrations. Since every curved $\mathcal{C}$-coalgebra is cofibrant (fibrant in the opposite category), we are left to show that $(-)^*$ preserves weak equivalences of curved $\C$-coalgebras. Let 
\[
f: D_1 \qi D_2
\]
be a weak equivalence of curved $\C$-coalgebras, that is,
\[
\Omega_\iota(f): \Omega_\iota D_1 \qi \Omega_\iota D_2 
\]
is a quasi-isomorphism of dg $\Omega \C$-algebras, by Lemma \ref{lemma: Sweedler dual is a Quillen adjunction}, we know that the Sweedler dual functor $(-)^\circ$ is a right Quillen functor. Therefore it preserves weak equivalences between fibrant objects (i.e: quasi-isomorphisms between cofibrant dg $\Omega \C$-algebras). Thus 
\[
(\Omega_\iota(f))^\circ: (\Omega_\iota D_1)^\circ \qi (\Omega_\iota D_2)^\circ 
\]
is a quasi-isomorphism of dg $\Omega \C$-coalgebras. Using the commutativity of the square, we get that
\[
\widehat{\mathrm{B}}_\iota(f^*): \widehat{\mathrm{B}}_\iota D_1^* \longrightarrow \widehat{\mathrm{B}}_\iota D_2^*
\]
is also a quasi-isomorphism of dg $\Omega \C$-coalgebras. Therefore $f^*: D_1^* \longrightarrow D_2^*$ is a weak equivalence in the model category of complete curved $\C$-algebras. 
\end{proof}

\begin{Remark}
There are many examples where one ends up considering the linear dual of a bar construction $(\mathrm{B}_\alpha(-))^*$. The above square shows that this construction lands naturally in a category of absolute algebras. 

\medskip

For instance, taking the linear dual of the bar construction is crucial in the study of formal moduli problems. J. Lurie's original approach \cite{Lurie11} involves considering the linear dual of the Chevalley-Eilenberg complex. The linear dual of a bar construction also plays a crucial role in the approach of D. Calaque, R. Campos and J. Nuiten \cite{CCN19} to formal moduli problems over algebras over a dg operad. Using the above duality square to study formal moduli problems will be the subject of a future work. 
\end{Remark}

\medskip

The above duality square allows us to state a precise comparison theorem between the homotopy theory of dg algebras over a cofibrant operad and the homotopy theory of dg coalgebras over the same cofibrant operad. 

\begin{lemma}\label{Prop: comparaison homotopique}
Let $B$ be a dg $\Omega \C$-algebra whose homology is degree-wise finite dimensional and bounded above or bounded below. Then the derived unit of adjunction

\[
\mathbb{R}(\eta_B): B \qi \left(\mathbb{R}(B^\circ)\right)^*
\]
\vspace{0.2pc}

is a quasi-isomorphism. Therefore the functor $(-)^\circ$ is homotopically fully faithful on the full sub-$\infty$-category of degree-wise finite dimensional bounded above dg $\Omega \C$-algebras and on the full sub-$\infty$-category of degree-wise finite dimensional bounded below dg $\Omega \C$-algebras.
\end{lemma}

\begin{proof}
Let $B$ be a dg $\Omega \C$-algebra such that its homology is degree-wise finite dimensional and bounded above. The bounded below case is analogous. Since $\kk$ is a field of characteristic zero, using the Homotopy Transfer Theorem, one can replace $B$ by its homology $\mathrm{H}_*(B)$ in order to compute this derived unit of adjunction. Indeed, there is a direct quasi-isomorphism 
\[
\Omega(i_B): \Omega_\iota \mathrm{B}_\iota \mathrm{H}_*(B) \qi \Omega_\iota \mathrm{B}_\iota B
\]

of dg $\Omega \C$-algebras between their cofibrant resolutions given by the bar-cobar adjunction relative to $\iota$. Here $i_B: \mathrm{B}_\iota \mathrm{H}_*(B) \qi \mathrm{B}_\iota B$ is the $\infty$-quasi-isomorphism given by the choice of a homotopy contraction between $B$ and $\mathrm{H}_*(B)$. Therefore we obtain the following commutative square
\[
\begin{tikzcd}[column sep=4pc,row sep=4pc]
\Omega_\iota \mathrm{B}_\iota \mathrm{H}_*(B) \arrow[d,"\mathbb{R}(\eta_{\mathrm{H}_*(B)})",swap] \arrow[r,"\Omega(i_B)"]
&\Omega_\iota \mathrm{B}_\iota B    \arrow[d,"\mathbb{R}(\eta_B)"]        \\
((\Omega_\iota \mathrm{B}_\iota \mathrm{H}_*(B))^\circ)^* \arrow[r,"((\Omega(i_B))^\circ)^*"]
&((\Omega_\iota \mathrm{B}_\iota B)^\circ)^*~,
\end{tikzcd}
\]

Using 2 out of 3, we conclude that $\mathbb{R}(\eta_B)$ is a quasi-isomorphism of dg $\Omega \C$-algebras if and only if $\mathbb{R}(\eta_{\mathrm{H}_*(B)})$ is. We compute that 
\[
\left(\Omega_\iota \mathrm{B}_\iota \mathrm{H}_*(B)\right)^\circ \cong \widehat{\mathrm{B}}_\iota \left(\mathrm{B}_\iota \mathrm{H}_*(B)\right)^* \cong \widehat{\mathrm{B}}_\iota \widehat{\Omega}_\iota (\mathrm{H}_*(B)^\circ)~,
\]

first using the commutativity of the square of Theorem \ref{thm: homotopical magical square} and then by applying Proposition \ref{prop: finite dual commutes}, where $\mathrm{H}_*(B)^\circ$ coincides with the linear dual since $\mathrm{H}_*(B)$ is degree-wise finite dimensional and bounded above. And the unit of the complete bar-cobar adjunction $\widehat{\Omega}_\iota \dashv \widehat{\mathrm{B}}_\iota$ 
\[
\eta_{\mathrm{H}_*(B)^\circ}: \mathrm{H}_*(B)^\circ \qi \widehat{\mathrm{B}}_\iota \widehat{\Omega}_\iota (\mathrm{H}_*(B)^\circ)
\]
is a quasi-isomorphism of dg $\Omega \C$-coalgebras since this adjunction is a Quillen equivalence. We can apply the linear dual functor to this quasi-isomorphism, resulting in the following commutative square 

\[
\begin{tikzcd}[column sep=3pc,row sep=3pc]
\mathrm{H}_*(B) \arrow[r,"\simeq"] \arrow[rd,"\mathbb{R}(\eta_{\mathrm{H}_*(B)})",swap]
&\left((\mathrm{H}_*(B))^{\circ} \right)^{*}  \\
&\left(\widehat{\mathrm{B}}_\iota \widehat{\Omega}_\iota (\mathrm{H}_*(B)^\circ)\right)^{*}  \arrow[u,"\left(\eta_{\mathrm{H}_*(B)^\circ}\right)^*",swap]~,
\end{tikzcd}
\] 
which allows us to conclude that, by 2 out of 3, the derived unit $\mathbb{R}(\eta_{\mathrm{H}_*(B)})$ is indeed a quasi-isomorphism of dg $\Omega \C$-algebras, where the top arrow is an isomorphism by Proposition \ref{prop: iso de Beck-Chevalley}.
\end{proof}

\begin{theorem}\label{thm: equivalence infini cat cog et alg}
There is an equivalence of $\infty$-categories

\[
\begin{tikzcd}[column sep=5pc,row sep=3pc]
         \mathsf{dg}~\Omega \C\text{-}\mathsf{coalg}^{\mathsf{f.d},\mp}~\left[\mathsf{Q.iso}^{-1}\right] \arrow[r, shift left=1.1ex, "(-)^\ast"{name=A}]
         &\mathsf{dg}~\Omega \C \text{-}\mathsf{alg}^{\mathsf{f.d},\pm}~\left[\mathsf{Q.iso}^{-1}\right]^\mathsf{op}~, \arrow[l, shift left=.75ex, "(-)^\circ"{name=B}] \arrow[phantom, from=A, to=B, , "\dashv" rotate=-90]
\end{tikzcd}
\]

between the $\infty$-category of dg $\Omega \C$-algebras with degree-wise finite dimensional and bounded below (resp. bounded above) homology and the $\infty$-category of $\Omega \C$-coalgebras with degree-wise finite dimensional and bounded above (resp. bounded below) homology.
\end{theorem}

\begin{proof}
By Lemma \ref{Prop: comparaison homotopique}, we know that the Sweedler duality functor $(-)^\circ$ is homotopically fully faithful restricted to dg $\Omega \C$-algebras with degree-wise finite dimensional and bounded below (resp. bounded above) homology. Using the Homotopy Transfer Theorem for dg $\Omega \C$-coalgebras, which is a particular case of the properadic Homotopy Transfer Theorem of \cite{properadic}, we can transfer this coalgebra structure onto the homology. Furthermore, as explained in \cite[Section 12]{grignoulejay18}, the existence of an $\infty$-quasi-isomorphism of $\Omega \C$-coalgebras is equivalent to the existence of a zig-zag of quasi-isomorphisms of $\Omega \C$-coalgebras. 

\medskip

Thus we can reproduce the same arguments as in Lemma \ref{Prop: comparaison homotopique} in order to show that the linear duality functor $(-)^*$ is homotopically fully faithful restricted to dg $\Omega \C$-coalgebras with degree-wise finite dimensional and bounded above (resp. bounded below) homology. It is immediate to see that the adjunction $(-)^* \dashv (-)^\circ$ sends dg $\Omega \C$-coalgebras with degree-wise finite dimensional and bounded above (resp. bounded below) homology to dg $\Omega \C$-algebras with degree-wise finite dimensional and bounded below (resp. bounded above) homology and vice-versa.
\end{proof}

\begin{Remark}
Since any cofibrant dg operad $\mathcal{P}$ is a deformation retract of a dg operad of the form $\Omega \C$, Theorem \ref{thm: equivalence infini cat cog et alg} also holds for any cofibrant dg operad.
\end{Remark}

\vspace{1.5pc}

\section{Some comparison results between absolute and classical algebras}

\vspace{1.5pc}

The goal of this subsection is to establish some comparison results between the absolute and the classical version of the same algebraic structure and to record some useful statements regarding absolute algebras.

\subsection{Restriction functors and their left adjoints}
Let $\C$ be a conilpotent curved cooperad. Its linear dual $\C^*$ is a curved operad in the sense of \cite{lucio2022curved}. We start by constructing an adjunction between pdg $\C$-algebras and pdg $\C^*$-algebras and show that it restricts to curved algebras on both sides. 

\begin{Proposition}\label{Prop: raw abs res adjunction}
There is an adjunction 
\[
\begin{tikzcd}[column sep=7pc,row sep=3pc]
            \mathsf{pdg}~\C\text{-}\mathsf{alg}\arrow[r,"\mathrm{Res}"{name=F}, shift left=1.1ex] 
           &\mathsf{pdg}~\C^*\text{-}\mathsf{alg}~, \arrow[l, shift left=.75ex, "\overline{\mathrm{Abs}}"{name=U}]
            \arrow[phantom, from=F, to=U, , "\dashv" rotate=90]
\end{tikzcd}
\]
between pdg $\C$-algebras and pdg $\C^*$-algebras, where the right adjoint $\mathrm{Res}$ is given by restricting the structure to finite sums of operations.
\end{Proposition}

\begin{proof}
There is a natural inclusion
\[
\iota_{\C}: \mathscr{S}(\C^*)(-) = \bigoplus_{n \geq 0} \left(\C^*(n) \otimes (-)^{\otimes n}\right)_{\mathbb{S}_n} \rightarrowtail \prod_{n \geq 0} \left(\mathrm{Hom}(\C(n),(-)^{\otimes n})\right)^{\mathbb{S}_n} = \widehat{\mathscr{S}}^c(\C)(-)
\]
which is given by the composition of the natural inclusion of the direct sum into the direct product, the natural inclusion $V^* \otimes W \rightarrowtail \mathrm{Hom}(V,W)$ and the norm map from coinvariants to invariants, which over a characteristic zero field is an isomorphism. This natural inclusion can be easily checked to be a morphism of monads, and therefore it induces an adjunction between their respective categories of algebras. The right adjoint $\mathrm{Res}$ is given by restricting the structural morphism along the inclusion $\iota_{\C}$.
\end{proof}

\begin{lemma}
The restriction functor factors as
\[
\mathrm{Res}: \mathsf{curv}~\C\text{-}\mathsf{alg} \longrightarrow \mathsf{curv}~\C^*\text{-}\mathsf{alg}~,
\]
between the subcategories of curved algebras on both sides. 
\end{lemma}

\begin{proof}
Let us show that if $A$ is a curved $\C$-algebra, its restriction $\mathrm{Res}$ is also a curved $\C^*$-algebra. This, in turn, follows from the commutativity of the following diagram 
\[
\begin{tikzcd}[column sep=3pc,row sep=1.5pc]
\widehat{\mathscr{S}}^c(\mathcal{C})(A) \arrow[dr,swap, "\widehat{\mathscr{S}}^c(\Theta_\mathcal{C})(\mathrm{id})" {xshift= 0.35pc}]\arrow[rr,"\gamma_A"  {xshift= -2pc}] 
&
&A \\
&\widehat{\mathscr{S}}^c(\I)(A) \cong A \arrow[ru,"d_A^2"]
& \\
\mathscr{S}(\C^*)(A) \arrow[uu,"\iota_\C"] \arrow[dr,swap,"\mathscr{S}(\Theta_{\mathcal{C}^*})(\mathrm{id})"] \arrow[rr,"\gamma_{\mathrm{Res}(A)}" {xshift= -2pc}] 
&
&A  \arrow[uu,equal]  \\
&\mathscr{S}(\I)(A) \cong A~, \arrow[ru,"d_{A}^2"]  \arrow[uu,equal, crossing over]
\end{tikzcd}
\]
where $\iota_\C$ is the natural inclusion constructed in the proof of Proposition \ref{Prop: raw abs res adjunction}. 
\end{proof}

\begin{Proposition}\label{Prop: restriction and absolution adjunction}
There is an adjunction 
\[
\begin{tikzcd}[column sep=7pc,row sep=3pc]
            \mathsf{curv}~\C\text{-}\mathsf{alg}\arrow[r,"\mathrm{Res}"{name=F}, shift left=1.1ex] 
           &\mathsf{curv}~\C^*\text{-}\mathsf{alg}~, \arrow[l, shift left=.75ex, "\mathrm{Abs}"{name=U}]
            \arrow[phantom, from=F, to=U, , "\dashv" rotate=90]
\end{tikzcd}
\]
between curved $\C$-algebras and curved $\C^*$-algebras. 
\end{Proposition}

\begin{proof}
Since, in each case, curved algebras are a reflective subcategory of pdg algebras (see Proposition \ref{Prop: curved reflective subcategory} and \cite[Section 1]{lucio2022curved}), we can apply the Adjoint Lifting Theorem \cite[Theorem 2]{AdjointLifting} to the following square 
\[
\begin{tikzcd}[column sep=4pc,row sep=4pc]
\mathsf{curv}~\mathcal{C}\text{-}\mathsf{alg} \arrow[r,"\mathrm{Res}"]  \arrow[d,"\mathrm{U}"{name=SD},shift left=1.1ex ]
&\mathsf{curv}~\mathcal{C}^*\text{-}\mathsf{alg} \arrow[d,"\mathrm{U}"{name=LDC},shift left=1.1ex ] \\
\mathsf{pdg}~\mathcal{C}\text{-}\mathsf{alg} \arrow[r,"\mathrm{Res}"{name=CC},,shift left=1.1ex] \arrow[u,"\mathrm{Curv}"{name=LD},shift left=1.1ex ] \arrow[phantom, from=SD, to=LD, , "\dashv" rotate=0]
&\mathsf{pdg}~\mathcal{C}^*\text{-}\mathsf{alg}~,  \arrow[l,"\overline{\mathrm{Abs}}"{name=CB},shift left=1.1ex] \arrow[u,"\mathrm{Curv}"{name=TD},shift left=1.1ex] \arrow[phantom, from=TD, to=LDC, , "\dashv" rotate=0] \arrow[phantom, from=CC, to=CB, , "\dashv" rotate=90]
\end{tikzcd}
\] 
where $\mathrm{Curv}$ denotes the reflector from pdg algebras to curved algebras in each case.
\end{proof}

\begin{Remark}
The construction given in the above proposition entails that given a curved $\C^*$-algebra, the functor $\mathrm{Abs}$ can be computed by first forgetting that it is curved, applying the functor $\overline{\mathrm{Abs}}$ and finally applying the reflector $\mathrm{Curv}$. So, a priori, the underlying pdg modules of $\overline{\mathrm{Abs}}$ and $\mathrm{Abs}$ differ. However, we suspect that, like in Proposition \ref{prop: restriction aux curved}, applying the functor $\overline{\mathrm{Abs}}$ to a curved $\C^*$-algebra gives a pdg $\C$-algebra which is already curved and thus these two functors do in fact coincide. 
\end{Remark}

This adjunction between curved $\C$-algebras and curved $\C^*$-algebras also induces an adjunction between complete curved $\C$-algebras and curved $\C^*$-algebras. 

\begin{Proposition}\label{Prop: restriction and complete absolution adjunction}
There is an adjunction 
\[
\begin{tikzcd}[column sep=7pc,row sep=3pc]
            \mathsf{curv}~\C\text{-}\mathsf{alg}^{\mathsf{comp}} \arrow[r,"\mathrm{Res}"{name=F}, shift left=1.1ex] 
           &\mathsf{curv}~\C^*\text{-}\mathsf{alg}~, \arrow[l, shift left=.75ex, "\mathrm{cAbs}"{name=U}]
            \arrow[phantom, from=F, to=U, , "\dashv" rotate=90]
\end{tikzcd}
\]
between curved $\C$-algebras and curved $\C^*$-algebras. 
\end{Proposition}

\begin{proof}
It suffices to compose the adjunction of Proposition \ref{Prop: restriction and absolution adjunction} with the adjunction between curved $\C$-algebras and complete curved $\C$-algebras, since the later is a reflective subcategory of the former by Proposition \ref{prop: complete C algebras are a reflexiv subcat}. 
\end{proof}

\begin{Remark}
Again, we suspect that the functor $\mathrm{Abs}$ always lands in complete curved $\C$-algebras and hence the functors $\mathrm{Abs}$ and $\mathrm{cAbs}$ (which one may call the absolution and the complete absolution functors) do, in fact, coincide. 
\end{Remark}

\subsection{Morphisms between complete curved $\C$-algebras}
We show that morphisms between complete curved $\C$-algebras are complete characterized as a limit of morphisms between each of the stages of their canonical filtrations. 

\begin{Proposition}
Let $A,A'$ be two complete curved $\C$-algebras. The set of morphisms of $\C$-algebras is given by 
\[
\mathrm{Hom}_{\C\text{-}\mathsf{alg}}(A,A') \cong \lim_{\omega \in \mathbb{N}} \mathrm{Hom}_{\C\text{-}\mathsf{alg}}(A/\mathrm{W}_\omega A,A'/\mathrm{W}_\omega A')~,
\]
where $\mathrm{W}_\omega A$ (resp. $\mathrm{W}_\omega A'$) denotes the $\omega$-stage of the canonical filtration, as defined in Definition \ref{def canonical filtration}. 
\end{Proposition}

\begin{proof}
Since $A'$ is complete, it can be written down as the following limit 
\[
A' \cong \lim_{\omega \in \mathbb{N}} A'/\mathrm{W}_\omega A'~,
\]
in the category of complete curved $\C$-algebras and therefore we have that
\[
\mathrm{Hom}_{\C\text{-}\mathsf{alg}}(A,A') \cong \lim_{\omega \in \mathbb{N}} \mathrm{Hom}_{\C\text{-}\mathsf{alg}}(A,A'/\mathrm{W}_\omega A')~. 
\]
The canonical epimorphism $\pi_\omega: A \twoheadrightarrow A/\mathrm{W}_\omega A$ induces an injection of hom-sets by pre-composing
\[
\mathrm{Hom}_{\C\text{-}\mathsf{alg}}(A/\mathrm{W}_\omega A,A'/\mathrm{W}_\omega A') \hookrightarrow \mathrm{Hom}_{\C\text{-}\mathsf{alg}}(A,A'/\mathrm{W}_\omega A')~,
\]
and therefore, if it is also surjective at the level of hom-sets, we will be able to conclude. And this is indeed the case since any map $f: A \longrightarrow A'/\mathrm{W}_\omega A'$ sends $W_\omega A$ to $0$ since $f(W_\omega A) \subseteq W_\omega A'$, and therefore factors (uniquely) through $A/\mathrm{W}_\omega A$. 
\end{proof}

\begin{Remark}
In fact, the category of complete curved $\C$-algebras can be written as a pseudo-limit of the following diagram of categories 
\[
\cdots \longrightarrow \mathsf{curv}~\C\text{-}\mathsf{alg}^{\leq \omega} \longrightarrow \cdots \longrightarrow \mathsf{curv}~\C\text{-}\mathsf{alg}^{\leq 2} \longrightarrow \mathsf{curv}~\C\text{-}\mathsf{alg}^{\leq 1}~, 
\]
where $\mathsf{curv}~\C\text{-}\mathsf{alg}^{\leq \omega}$ denotes the subcategory of $\omega$-nilpotent complete curved $\C$-algebras, that is, those algebras $A$ which satisfy that $A \cong A/\mathrm{W}_\omega A$. 
\end{Remark}

\subsection{The absolute algebra structure on the homology of a dg absolute algebra}
Recall that when $\C$ is a conilpotent dg cooperad (a conilpotent curved cooperad with zero curvature), the category of curved $\C$-algebras is precisely the category of dg $\C$-algebras, and therefore one can wonder whether the homology of such a dg $\C$-algebra inherits a graded $\C$-algebra structure.

\begin{Proposition}\label{Prop: the homology has an absolute structure}
Let $\C$ be a conilpotent dg cooperad and $A$ a dg $\C$-algebra.

\medskip

\begin{enumerate}
\item The homology $\mathrm{H}_*(A)$ inherits a graded $\mathrm{H}_*(\C)$-algebra structure. 

\medskip

\item Assume $\C$ has a zero differential, meaning $\mathrm{H}_*(\C) \cong \C$. If the algebra $A$ is a complete dg $\C$-algebra, so is $\mathrm{H}_*(A)$.
\end{enumerate}
\end{Proposition}

\begin{proof}
Let us prove the first point. Let 
\[
\gamma_A: \prod_{n \geq 0} \left(\mathrm{Hom}(\C(n),A^{\otimes n})\right)^{\mathbb{S}_n} \longrightarrow A~,
\]
denote the structural map of the dg $\C$-algebra $A$. Since it is a morphism of dg modules, it induces a map 
\[
\gamma_{\mathrm{H}_*(A)}: \mathrm{H}_*\left(\prod_{n \geq 0} \left(\mathrm{Hom}(\C(n),A^{\otimes n})\right)^{\mathbb{S}_n}\right) \cong \prod_{n \geq 0} \left(\mathrm{Hom}(\mathrm{H}_*(\C)(n),(\mathrm{H}_*(A))^{\otimes n})\right)^{\mathbb{S}_n} \longrightarrow \mathrm{H}_*(A)~,
\]
where we are using that infinite products, invariants, hom-spaces and tensor products are all exact over a field of characteristic zero and thus commute with the homology functor. This structural morphism endows $A$ with a graded $\mathrm{H}_*(\C)$-algebra structure. 

\medskip

Now, let us prove the second point. First, we start assuming that $A$ is a weight-nilpotent dg $\C$-algebra, meaning that there exists an integer $\omega_0 \geq 0$ such that $A \cong A/\mathrm{W}_{\omega_0}A$. This is the case if and only if the structural morphism of $A$ factors through 
\[
\gamma_A: \prod_{n \geq 0} \left(\mathrm{Hom}(\mathscr{R}_{\omega_0}\C(n),A^{\otimes n})\right)^{\mathbb{S}_n} \longrightarrow A~,
\]
where $\mathscr{R}_{\omega_0}\C$ is the $\omega_0$-th stage of the coradical filtration of $\C$. This implies that the structural morphism of $\mathrm{H}_*(A)$ also factors as an algebra over $\mathscr{R}_{\omega_0}\C$, and therefore that $\mathrm{H}_*(A)$ is also weight-nilpotent of nilpotency class at most $\omega_0$. So, in particular, $\mathrm{H}_*(A)$ is a complete dg $\C$-algebra. Let us deal with the general case now. If $A$ is a complete dg $\C$-algebra, then 
\[
A \cong \lim_{\omega} A/W_{\omega}A~. 
\]
Since a tower of epimorphims of vector spaces indexed by $\mathbb{N}$ satisfies the Mittag-Leffler condition, taking the limit commutes with taking the homology and therefore we have: 
\[
\mathrm{H}_*(A) \cong \lim_{\omega} \mathrm{H}_*(A/W_{\omega}A)~,
\]
where each $\mathrm{H}_*(A/W_{\omega}A)$ is a complete dg $\C$-algebra. Since complete dg $\C$-algebras are stable under limits, this implies that $\mathrm{H}_*(A)$ is also a complete dg $\C$-algebra.
\end{proof}

\subsection{Algebras over finitely cogenerated dg cooperads}
We show further comparison results in the case where $\C$ is a finitely cogenerated dg cooperad in arities $\geq 2$. In that case, the weight filtration and the arity filtrations are commensurable, and many comparisons become easier. 

\medskip

For the rest of this subsection, we suppose that the cooperad $\C$ is cogenerated by a dg $\mathbb{S}$-module $M$ which is bounded in arities $\geq 2$, bounded in degrees and degree-wise finite dimensional. This essentially means that $\C$ can be written as an equalizer of the cofree conilpotent cooperad on $M$ by some corelations. Observe that this implies that $\C(n)$ is a bounded degree-wise finite dimensional dg module for any $n \geq 0$. It also implies that $\C$ can be written as the following colimit
\[
\tau^{\leq 1}\C \rightarrowtail \tau^{\leq 2}\C \rightarrowtail \cdots \rightarrowtail \tau^{\leq n}\C \rightarrowtail \cdots \rightarrowtail \colim_{n \in \mathbb{N}} \tau^{\leq n}\C \cong \C~,
\] 
where $\tau^{\leq n}\C$ is the suboperad obtained by truncating $\C$ in arities strictly larger that $n$. For any $n \geq 1$, there is a short exact sequence of dg $\mathbb{S}$-modules 
\[
\begin{tikzcd}
0 \arrow[r]
&\tau^{\leq n}\C \arrow[r,"i_n",hook]
&\C \arrow[r,"p_n"]
&\C / \tau^{\leq n}\C \arrow[r]
&0~.
\end{tikzcd}
\]

\begin{Definition}[Arity filtration]\label{def: arity filtration}
Let $A$ be a dg $\C$-algebra. The \textit{arity filtration} of $A$ is the decreasing filtration given by 
\[ 
\mathrm{F}_n A \coloneqq \mathrm{Im}\left(\gamma_A \circ \widehat{\mathscr{S}}^c(p_n)(\mathrm{id}_A): \widehat{\mathscr{S}}^c(\C / \tau^{\leq n}\C )(A) = \prod_{k \geq n} \mathrm{Hom}_{\mathbb{S}_k}(\C(k),A^{\otimes k}) \longrightarrow A \right)
\]
for all $n \geq 1~.$ Notice that we have 
\[
A = \mathrm{F}_1 A \supseteq \mathrm{F}_2 A \supseteq \mathrm{F}_3 A \supseteq \cdots \supseteq \mathrm{F}_n A \supseteq \cdots.
\]
\end{Definition}
	
Each step $\mathrm{F}_n A$ of the arity filtration of a dg $\C$-algebra $A$ is an ideal, and the quotient $A/\mathrm{F}_n A$ has a canonical dg $\C$-algebra structure induced by the one of $A$. This follows from the fact that $A/\mathrm{F}_n A$ fits into the following pushout diagram 
\[
\begin{tikzcd}[column sep=3.5pc,row sep=3.5pc]
\widehat{\mathscr{S}}^c(\C)(A) \arrow[d,"\gamma_A",swap] \arrow[r,"\widehat{\mathscr{S}}^c(i_n)"]  \arrow[dr, phantom,"\ulcorner" rotate=-180, very near end]
&\widehat{\mathscr{S}}^c(\tau^{\leq n}\C)(A) \arrow[d]  \\
A \arrow[r]
&A/\mathrm{F}_n A
\end{tikzcd}
\]
of dg $\C$-algebras. 

\begin{Definition}[Arity-completion]
Let $A$ be a dg $\C$-algebra. Its \textit{arity-completion} is given by 
\[ 
\widehat{(A)}_a \coloneqq \lim_{n} A/\mathrm{F}_n A~,
\]
where the limit is taken in the category of dg $\C$-algebras.
\end{Definition}

Any dg $\C$-algebra $A$ comes equipped with a canonical map $\varepsilon_A: A \longrightarrow \widehat{(A)}_a$ to its arity-completion. It is said to be \textit{arity-complete} if this map is an isomorphism.

\begin{Definition}[Arity-nilpotent dg $\C$-algebras]
Let $A$ be a dg $\C$-algebra. It is \textit{arity-nilpotent} if there exists a natural number $n_0$ such that $A \cong A/\mathrm{F}_n A$. The smallest such number $n_0$ is called the \textit{arity-nilpotency class} of the algebra $A$. 
\end{Definition}

\begin{Proposition}\label{prop: arity and weight topologies coincide}
Let $A$ be a dg $\C$-algebra. The topology induced by the arity filtration coincides with the topology induced by the canonical filtration. In particular, it is arity-complete if and only if it is complete. 
\end{Proposition}

\begin{proof}
For any $n \geq 2$, there is an inclusion $\tau^{\leq n}\C \subseteq \mathscr{R}_{n-1}\C$. Indeed, since $\C$ is cogenerated in arities $\geq 2$, any element in arity $n$ admits at most $n-1$ non-trivial decompositions. For any dg $\C$-algebra $A$, this inclusion gives an inclusion $\mathrm{W}_{n-1}A \subseteq \mathrm{F}_n A$ between the $(n-1)$-th term of the canonical filtration and the $n$-th term of the arity filtration. 

\medskip

The other way around, denote $M$ the dg $\mathbb{S}$-module of cogenerators of $\C$ and let $m$ be the maximal arity of $M$. Since $M$ is bounded in arity, then $m$ is a finite integer. There is an inclusion $\mathscr{R}_{\omega}\C \subseteq \tau^{\leq \omega.m}\C$, as any element in $\C$ which admits $\omega$ non-trivial partial decompositions can be at most of arity $\omega.m$. This inclusion induces an inclusion $\mathrm{F}_{\omega.m}A \subseteq \mathrm{W}_{\omega} A$ for any dg $\C$-algebra $A$. Since both filtrations are commensurable, they induce the same topology on $A$, and by a standard cofinality argument, the completions with respect to each filtration are isomorphic. 
\end{proof}

\begin{Remark}
The \textit{arity filtration} on dg $\C$-algebras can be defined as long as $\C(0) \cong 0$, since this condition guarantees that the arity truncation does define a series of sub-cooperads. However, in general, it will differ from the canonical filtration induced by the weight of the operations in $\C$, even when $\C$ is reduced. For example, if $\C$ is the bar construction on the commutative operad $\mathcal{C}om$ (which encodes absolute $\mathcal{L}_\infty$-algebras), there are infinitely-many operations in weight $1$ and in arbitrarily high arities. See \cite[Section 1]{lucio2022integration} for more details. 
\end{Remark}

\begin{Corollary}\label{cor: dense image in the product}
Let $V$ be a dg module. The image of the canonical inclusion
\[
\iota_{\C}(V): \mathscr{S}(\C^*)(V) = \bigoplus_{n \geq 0} \left(\C^*(n) \otimes V^{\otimes n}\right)_{\mathbb{S}_n} \rightarrowtail \prod_{n \geq 0} \left(\mathrm{Hom}(\C(n),V^{\otimes n})\right)^{\mathbb{S}_n} = \widehat{\mathscr{S}}^c(\C)(V)
\]
defines a dense subset of the free dg $\C$-algebra on $V$. 
\end{Corollary}

\begin{proof}
Notice that since $\C(n)$ is a degree-wise finite dimensional bounded complex and since we are working over a characteristic zero field, the right hand side is isomorphic to 
\[
\prod_{n \geq 0} \left(\mathrm{Hom}(\C(n),V^{\otimes n})\right)^{\mathbb{S}_n} \cong \prod_{n \geq 0} \left(\C^*(n) \otimes V^{\otimes n}\right)_{\mathbb{S}_n}~. 
\]
Then, the result follows directly from Proposition \ref{prop: arity and weight topologies coincide}, since the right hand side can be identified with the arity-completion of the left hand side, and the natural inclusion to the completion has a dense image. 
\end{proof}

\begin{theorem}\label{thm: restriction functor fully faithful}
The restriction functor 
\[
\mathrm{Res}: \mathsf{dg}~\C\text{-}\mathsf{alg}^{\mathsf{comp}} \longrightarrow \mathsf{dg}~\C^*\text{-}\mathsf{alg}
\]
is fully faithful. Therefore dg $\C$-algebras are a reflective subcategory of dg $\C^*$-algebras. 
\end{theorem}

\begin{proof}
It is immediate that it is faithful as the restriction functor does not change the underlying chain complex. Let us show that it is full. Let $A$ and $A'$ be two dg $\C$-algebras. We want to show that the map 
\[
\mathrm{Res}(A,A'): \mathrm{Hom}_{\mathsf{dg}~\C\text{-}\mathsf{alg}^{\mathsf{comp}}}(A,A') \longrightarrow \mathrm{Hom}_{\mathsf{dg}~\C^*\text{-}\mathsf{alg}}(\mathrm{Res}(A),\mathrm{Res}(A')) 
\]
at the level of hom-sets is a bijection. First, let us assume that $A$ is free as a (complete) dg $\C$-algebra, meaning there is an isomorphism
\[
A \cong \widehat{\mathscr{S}}^c(\C)(V)~,
\]
where $V$ is some dg module. Then, on the one hand, we have that 
\[
\mathrm{Hom}_{\mathsf{dg}~\C\text{-}\mathsf{alg}^{\mathsf{comp}}}(A,A') \cong \mathrm{Hom}_{\mathsf{dg}\text{-}\mathsf{mod}}(V,A')~. 
\]
On the other hand, pre-composition by the canonical inclusion 
\[
\iota_{\C}(V): \mathscr{S}(\C^*)(V) = \bigoplus_{n \geq 0} \left(\C^*(n) \otimes V^{\otimes n}\right)_{\mathbb{S}_n} \rightarrowtail \prod_{n \geq 0} \left(\mathrm{Hom}(\C(n),V^{\otimes n})\right)^{\mathbb{S}_n} = \widehat{\mathscr{S}}^c(\C)(V)
\]
induces a map 
\[
\iota_{\C}^*: \mathrm{Hom}_{\mathsf{dg}~\C^*\text{-}\mathsf{alg}}(\mathrm{Res}(A),\mathrm{Res}(A')) \longrightarrow \mathrm{Hom}_{\mathsf{dg}~\C^*\text{-}\mathsf{alg}}(\mathscr{S}(\C^*)(V),\mathrm{Res}(A')) \cong \mathrm{Hom}_{\mathsf{dg}\text{-}\mathsf{mod}}(V,A')~, 
\]
and since by Corollary \ref{cor: dense image in the product} the image of $\iota_{\C}$ is dense, then pre-composing by $\iota_{\C}$ is injective. This, in turn, implies that $\mathrm{Res}(A,A')$ is bijective when $A$ is free. The general case follows from the free case using the following. First, we use that any complete dg $\C$-algebra can be written down as a $\mathrm{U}$-split coequalizer of free dg $\C$-algebras and secondly we use that the restriction functor preserves $\mathrm{U}$-split coequalizers since it commutes with the forgetful functors. 
\end{proof}

\begin{Remark}\label{Rmk: absolute algebras are algebras over an idempotent monad}
A direct consequence of Theorem \ref{thm: restriction functor fully faithful} is that complete dg $\C$-algebras can be identified with the category of algebras over an idempotent monad in the category dg $\C^*$-algebras, which is the monad induced by the adjunction in Proposition \ref{Prop: restriction and complete absolution adjunction}.
\end{Remark}

\subsubsection{Comparison with nilpotent algebras and with completions of classical algebras}\label{subsection: comparison between the filtrations}
It is natural to wonder, particularly in the light of Theorem \ref{thm: restriction functor fully faithful}, whether the notion of nilpotence and of completion that one may define at the level of dg $\C^*$-algebras allows us to recover the category of complete dg $\C$-algebras. We will show that while nilpotent algebras coincide, the "classical" completions does \textit{not} recover absolute algebras. We start by defining the arity filtration on dg $\C^*$-algebras. 

\begin{Definition}[Arity filtration on dg $\C^*$-algebras]\label{def: arity filtration on classical algebras}
Let $B$ be a dg $\C^*$-algebra. The \textit{arity filtration} of $B$ is the decreasing filtration of given by 
\[ 
\tilde{\mathrm{F}}_n B \coloneqq \mathrm{Im}\left(\gamma_B|_{\geq n}: \bigoplus_{k \geq n} \C^*(k) \otimes_{\mathbb{S}_k} V^{\otimes k}\longrightarrow B \right)
\]
for all $n \geq 1~.$ Notice that we have 
\[
B = \tilde{\mathrm{F}}_1 B \supseteq \tilde{\mathrm{F}}_2 B \supseteq \tilde{\mathrm{F}}_3 B \supseteq \cdots \supseteq \tilde{\mathrm{F}}_n B \supseteq \cdots.
\]
\end{Definition}

Each step $\tilde{\mathrm{F}}_n B$ of the arity filtration of a dg $\C^*$-algebra $B$ defines an ideal and the quotient $B/\tilde{\mathrm{F}}_n B$ has a canonical dg $\C^*$-algebra structure induced by that of $B$.

\begin{Definition}[Arity completion of dg $\C^*$-algebras]
Let $B$ be a dg $\C^*$-algebra. Its \textit{arity-completion} is given by the limit
\[ 
\widehat{(B)}_a \coloneqq \lim_{n} B/\tilde{\mathrm{F}}_n B~,
\]
in the category of dg $\C^*$-algebras. 
\end{Definition}

An algebra $B$ is said to be \textit{arity-complete} if the canonical map $B \longrightarrow \widehat{(B)}_a$ is an isomorphism. 

\begin{Example}[$I$-adic completion]\label{Example: I-adic completion}
Let $\mathcal{C}om$ be the commutative operad, which encodes non-unital commutative algebras. For any augmented dg commutative algebra $B$, its augmentation ideal $I$ inherits a dg $\mathcal{C}om$-algebra, and the $I$-adic completion coincides with the arity completion with respect to its dg $\mathcal{C}om$-algebra structure with an added unit, meaning we have
\[
\widehat{B}_I = \left(\lim_{n} I/I^n\right) \oplus \kk ~,
\]
where the limit is precisely the augmentation ideal of $\widehat{B}_I$. 
\end{Example}

\begin{Definition}[Nilpotent dg $\C^*$-algebras]
Let $B$ be a dg $\C^*$-algebra. It is \textit{nilpotent} of there exists an natural number $n_0$ such that $B \cong B/\tilde{\mathrm{F}}_{n_0} B$. The smallest such number $n_0$ is called the \textit{nilpotency class} of the algebra $B$. 
\end{Definition}

\begin{Remark}
Any nilpotent dg $\C^*$-algebra is in particular arity complete.
\end{Remark}

\begin{Proposition}\label{Prop: équivalences entre arity nilpotentes}
The adjunction 
\[
\begin{tikzcd}[column sep=7pc,row sep=3pc]
            \mathsf{dg}~\C\text{-}\mathsf{alg}^{\mathsf{comp}} \arrow[r,"\mathrm{Res}"{name=F}, shift left=1.1ex] 
           &\mathsf{dg}~\C^*\text{-}\mathsf{alg}~, \arrow[l, shift left=.75ex, "\mathrm{cAbs}"{name=U}]
            \arrow[phantom, from=F, to=U, , "\dashv" rotate=90]
\end{tikzcd}
\]

restricts to the identity between the subcategories of arity nilpotent algebras of a fixed nilpotency class on both sides, which are therefore equivalent.  
\end{Proposition}

\begin{proof}
Let $n_0 \geq 1$ be a natural number. A dg $\C$-algebra $A$ is arity-nilpotent of nilpotency class $n_0$ if and only if its structural morphism factors through the submonad 
\[
\prod_{n \geq 0} \left(\mathrm{Hom}\left(\tau^{\leq n_0}\C(n),A^{\otimes n}\right)\right)^{\mathbb{S}_n} \cong \bigoplus_{n \geq 0}^{n_0} \C^*(n) \otimes_{\mathbb{S}_n} A^{\otimes n}~. 
\]
Likewise, a dg $\C^*$-algebra $B$ is arity-nilpotent of nilpotency class $n_0$ if and only if its structural morphism factors through 
\[
\bigoplus_{n \geq 0} \tau^{\leq n_0}\C^*(n) \otimes_{\mathbb{S}_n} B^{\otimes n} \cong \bigoplus_{n \geq 0}^{n_0} \C^*(n) \otimes_{\mathbb{S}_n} B^{\otimes n}~. 
\]
Therefore both subcategories are given as algebras over the same monad. Furthermore, the functor $\mathrm{Res}$ restricted to these subcategories can be identified with the identity functor of the category of algebras over this monad. It follows directly from Theorem \ref{thm: restriction functor fully faithful} that the functor $\mathrm{Abs}$ must also be the identity on arity-nilpotent dg $\C^*$-algebras.
\end{proof}

\begin{Proposition}\label{Prop: arity completion defines a functor}
The arity completion of a dg $\C^*$-algebra admits a canonical complete dg $\C$-algebra structure. It defines a functor
\[
\widehat{(-)}_a: \mathsf{dg}~\C^*\text{-}\mathsf{alg} \longrightarrow \mathsf{dg}~\C\text{-}\mathsf{alg}^{\mathsf{comp}}~.
\]
\end{Proposition}

\begin{proof}
By Proposition \ref{Prop: équivalences entre arity nilpotentes}, the arity completion is given as a limit of complete dg $\C$-algebras, which can be computed in the underlying category of dg modules. Since complete dg $\C$-algebras are stable under limits, the result is again a complete dg $\C$-algebra. Any map between dg $\C^*$-algebras respects the arity filtration and thus induces a map between the respective limits.
\end{proof}

\textbf{Absolute algebras and arity completions.} By Theorem \ref{thm: restriction functor fully faithful}, the category of complete dg $\C$-algebra is given as the category of algebras over the idempotent monad $\mathrm{Res}\mathrm{Abs}$, which lives in the category of dg $\C^*$-algebras. One might be tempted to think that the left adjoint $\mathrm{Abs}$ to the restriction functor $\mathrm{Res}$ is given by the arity-completion functor of dg $\C^*$-algebras of Proposition \ref{Prop: arity completion defines a functor} and thus that complete dg $\C^*$-algebras are equivalent to arity-complete dg $\C^*$-algebras. Turns out that both of these statements are false. In fact, the arity-completion functor is not even idempotent in general. 

\medskip

Let us illustrate this with a concrete example. Let us consider $\C = \mathcal{C}om^*$. Then, as explained in Example \ref{Example: I-adic completion}, dg $\mathcal{C}om$-algebras correspond to augmented dg commutative algebras and the functor in Proposition \ref{Prop: arity completion defines a functor} can be identified with the $I$-adic completion of the augmentation ideal. 

\medskip

The $I$-adic completion functor is not idempotent and therefore can not coincide with the functor $\mathrm{Abs}$. Consider $\kk[x_1,x_2,\cdots]$ the free polynomial algebra in a countable number of variables. Then its $I$-adic completion with respect to its augmentation ideal is given by the algebra of formal power series $\kk[[x_1,x_2,\cdots]]$ in a countable number of variables. This algebra is not $I$-adically complete with respect to its augmentation ideal, as it is explained in \cite[110.7 Noncomplete completion]{stackproject}. However, since any finitely generated $\kk$-algebra is Noetherian, the $I$-adic completion is indeed idempotent on finitely generated dg $\mathcal{C}om$-algebras and thus coincides with the functor $\mathrm{Abs}$ in this particular case. 

\vspace{1.5pc}

\section{Examples: contramodules and absolute associative algebras}

\vspace{1.5pc}

The goal of this section is to illustrate the new notion of an absolute algebras through various examples. First, we show that the theory of \textit{contramodules} is a particular case of the theory of absolute algebras. Then, we explore the notion of absolute associative algebras and absolute Lie algebras and explain how well-known constructions are instances of these structures.

\subsection{Contramodules}
The notion of a \textit{contramodule} over a dg counital coassociative coalgebra appears naturally as the dual definition of a comodule. This notion was introduced in the seminal work of S. Eilenberg and J. C. Moore in \cite{eilenbergmoore65}. After being almost forgotten for a number of years, this notion reemerged in the work of L. Positselski, see \cite{positselski2021contramodules} for a extensive account. In this subsection, we show that contramodules appear as algebras over cooperads concentrated in arity one. This allows us to subsume the theory of contramodules, obtaining many examples of interest of "absolute algebras". To the best of our knowledge, the duality square stated in the previous section provides also a new result in the context of contramodules.

\begin{Definition}[Contramodule]
Let $(C, \Delta, \epsilon, d_C)$ be a dg counital coassociative coalgebra. A \textit{dg} $C$-\textit{contramodule} $M$ is the data $(M,\gamma_M,d_M)$ of a  dg module $(M,d_M)$ together with a morphism of dg modules 
\[
\gamma_M: \mathrm{Hom}(C,M) \longrightarrow M~,
\]

such that the following diagrams commute
\[
\begin{tikzcd}[column sep=3pc,row sep=3pc]
\mathrm{Hom}(C \otimes C, M) \arrow[r,"\Delta^* "] \arrow[d,"\varsigma",swap]
& \mathrm{Hom}(C, M) \arrow[dd,"\gamma_M"] \\
\mathrm{Hom}(C, \mathrm{Hom}(C, M)) \arrow[d,"(\gamma_M)_*",swap]
& \\
\mathrm{Hom}(C, M) \arrow[r,"\gamma_M"]
&M ~,
\end{tikzcd} \quad \quad
\begin{tikzcd}[column sep=3pc,row sep=3pc]
M \cong \mathrm{Hom}(\kk, M)  \arrow[r,"\epsilon^*"] \arrow[rd,"\cong",swap]
&\mathrm{Hom}(C, M) \arrow[d,"\gamma_M"]\\
&M~.
\end{tikzcd}
\]
\end{Definition}

\begin{Remark}
If $(C, \Delta, \epsilon, d_C)$ is a dg counital coassociative coalgebra, one can show that the endofunctor
\[
\mathrm{Hom}(C,-): \mathsf{dg}~\mathsf{mod} \longrightarrow \mathsf{dg}~\mathsf{mod}
\]
admits a monad structure such that dg $C$-contramodules are algebras over this monad. In fact, it admits \textit{two} monad structures, depending on the choice of $\varsigma$ in the above definition. Indeed, there are two isomorphisms
\[
\mathrm{Hom}(C \otimes C, M) \cong \mathrm{Hom}(C, \mathrm{Hom}(C, M)) ~,
\]
depending on the choice of either the left or the right factor in $C \otimes C$. One should speak of \textit{left} or \textit{right} $C$-contramodules in each of these cases. We omit these subtleties when possible.
\end{Remark}

\begin{Example}
Let $C$ be a dg counital coassociative coalgebra and $V$ be a dg module. The \textit{free} dg $C$-contramodule on $V$ is given by $\mathrm{Hom}(C,V)$ with the obvious structure maps. 
\end{Example}

\begin{Example}
Let $C$ be a dg counital coassociative coalgebra and let $M$ be a right dg $C$-comodule. For any dg module $V$, the dg module $\mathrm{Hom}(M,V)$ inherits a canonical left dg $C$-contramodule structure. 
\end{Example}

\begin{Proposition} \leavevmode
\begin{enumerate}
\item The data of a dg counital coassociative coalgebra $(C, \Delta, \epsilon, d_C)$ is equivalent to the data of a dg cooperad $(\mathcal{C},\Delta_{\mathcal{C}},\epsilon, d_{\C})$ whose dg $\mathbb{S}$-module $\mathcal{C}$ is concentrated in arity one.

\medskip

\item In the case above, the definition of a dg $C$-contramodule is equivalent to the definition of a dg $\C$-algebra.
\end{enumerate}
\end{Proposition}

\begin{proof}
It is a straightforward computation from the definitions.
\end{proof}

This equivalence is a great source of examples for algebras over cooperads. Moreover, it is also a great source of counter-examples.

\begin{Example}[Contramodules over formal power series]
Consider the cofree conilpotent coalgebra $\kk[t]^c$ cogenerated by a single element $t$ of degree zero. In this case, one can see that the data of $\kk[t]^c$-contramodule structure on $M$ is equivalent to the data of a structural map
\[
\gamma_M: \kk[[t]] ~\widehat{\otimes}~ M \cong \mathrm{Hom}(\kk[t]^c,M) \longrightarrow M~.
\]
which satisfies unital and associativity conditions with respect to the algebra of formal power series in one variable $\kk[[t]]$. Therefore, for any formal power series
\[
\sum_{n \geq 0} m_n \otimes t^n ~~\quad~~\text{in} ~~\quad~~ \kk[[t]] ~\widehat{\otimes}~ M~,
\]
the structural morphism $\gamma_M$ gives a well-defined image $\gamma_M\left(\sum_{n \geq 0} m_n \otimes t^n \right)$ in $M$, without \textit{presupposing} any topology on $M$.
\end{Example}

Let $C$ be a dg \textit{conilpotent} counital coassociative coalgebra and let $\overline{C}$ be its coaugmentation ideal. Any dg $C$-contramodule $M$ admits a \textit{canonical decreasing filtration} given by 
\[
\mathrm{W}_\omega M \coloneqq \mathrm{Im}\left(\gamma_M^{ \geq \omega}: \mathrm{Hom}\left(\overline{C}/\mathscr{R}_\omega \overline{C}, M\right) \longrightarrow M \right)~,
\]
where $\mathscr{R}_\omega \overline{C}$ denotes the $\omega$-stage of the coradical filtration. One says that $M$ is \textit{complete} if the canonical morphism of dg $C$-contramodules 
\[
\varphi_M: M \twoheadrightarrow \lim_{\omega} M/\mathrm{W}_\omega M
\]
is an isomorphism. Again, this is a particular case of the more general notion of canonical filtration for algebras over conilpotent cooperads given in Definition \ref{def canonical filtration}.

\begin{Counterexample}
Consider again the cofree conilpotent coalgebra $\kk[t]^c$ cogenerated by a single element $t$ of degree zero. There exists a $\kk[t]^c$-contramodule $M$ and a family of elements $\{m_n\}_{n \geq 0}$ in $M$ such that 
\[
\gamma_M\left(m_n \otimes t^n \right) = 0 \quad \text{for all} \quad n\geq 0 \quad \text{and} \quad \gamma_M\left(\sum_{n \geq 0} m_n \otimes t^n \right) \neq 0~.
\]

The contramodule $M$ is constructed as follows. Let $V = \bigoplus_{n \geq 1}\kk.e_n$ be the vector space spanned by a countable basis and let us consider the free $\kk[t]^c$-contramodule on $V$ given by $\kk[[t]] ~\widehat{\otimes}~ V$. We define the following injective endomorphism of this free contramodule 
\[
\sum_{n \geq 1}f_n(t)e_n \mapsto \sum_{n \geq 1}t^n.f_n(t)e_n~,
\]
where $\{f_n(t)\}_{n \geq 0}$ are the coefficients in $\kk[[t]]$, converging to $0$ in the topology of $\kk[[t]]$, of a generic element in this free contramodule. The contramodule $M$ is given by the cokernel of this injective endormophism. The family of elements $\{m_n\}_{n \geq 0}$ is then given by $m_n = \overline{e}_n$, where $\overline{e}_n$ is the image of the basis $e_n$ in this cokernel. Then $t_n. m_n= 0 $ is clearly zero in this cokernel, however $\sum_{n \geq 1} t_n.m_n$ is not zero. This is because the element $\sum_{n \geq 1} e_n$ is not well-defined in $\kk[[t]] ~\widehat{\otimes}~ V$. It follows that the topology induced by the canonical filtration on $M$ is not complete. See \cite[Section 1.5]{positselski2021contramodules} for more details on this construction.
\end{Counterexample}

Let $A$ be a dg unital associative algebra (viewed as a dg operad concentrated in arity one), and let $C$ be a dg conilpotent counital coassociative algebra (viewed as a conilpotent dg cooperad concentrated in arity one). Let $\alpha: C \longrightarrow A$ be a twisting morphism between them. It induces a first adjunction 

\[
\begin{tikzcd}[column sep=7pc,row sep=3pc]
           \mathsf{dg}~C\text{-}\mathsf{comod}_{l} \arrow[r,"A \otimes_{\alpha} - "{name=F}, shift left=1.1ex] 
           &\mathsf{dg}~A\text{-}\mathsf{mod}_{l} ~. \arrow[l, shift left=.75ex, "C \otimes_{\alpha} - "{name=U}]
            \arrow[phantom, from=F, to=U, , "\dashv" rotate=-90]
\end{tikzcd}
\]
\vspace{0.3pc}

between the category of left dg $C$-comodules and the category of left dg $A$-modules, where $\otimes_\alpha$ denotes the twisted tensor product, see \cite[Section 2.1]{LodayVallette12} for an account of this twisted tensor product. One can check that this adjunction is nothing but the classical bar-cobar adjunction relative to $\alpha$ when we view $A$ as a dg operad concentrated in arity one and $C$ as a dg conilpotent cooperad concentrated in arity one. When we endow the category of left dg $A$-modules with the \textit{projective model structure} and the category of left dg $C$-comodules with the transferred model structure, this becomes a Quillen adjunction. 

\begin{Remark}
In fact, when one considers the above adjunction relative to the canonical twisting morphism $\iota: C \longrightarrow \Omega(C)$, L. Positselski showed that it induces a Quillen equivalence even when the coalgebra $C$ is not conilpotent. For these results, see \cite[Section 6.7]{PositselskiTwoKinds}.
\end{Remark}

\begin{lemma}
Let $A$ be a dg unital associative algebra viewed as a dg operad concentrated in arity one. The category of dg $A$-coalgebras is isomorphic to the category of right dg $A$-modules. Furthermore, the cofree right dg $A$-module functor is given by $\mathrm{Hom}(A,-)$.
\end{lemma}

\begin{proof}
It is a straightforward computation.
\end{proof}

\begin{Remark}
When $A$ is the group algebra $\kk[G]$ for some group $G$, the functor $\mathrm{Hom}(\kk[G],-)$ is usually called the \textit{coinduced representation}.
\end{Remark}

The twisting morphism $\alpha: C \longrightarrow A$ induces a second adjunction 

\[
\begin{tikzcd}[column sep=7pc,row sep=3pc]
           \mathsf{dg}~A\text{-}\mathsf{mod}_{r}\arrow[r," \mathrm{Hom}^{\alpha}(C \text{,}-) "{name=F}, shift left=1.1ex] 
           &\mathsf{dg}~C\text{-}\mathsf{contra}_{r}^{\mathsf{comp}} ~. \arrow[l, shift left=.75ex, "\mathrm{Hom}^{\alpha}(A \text{,}-)"{name=U}]
            \arrow[phantom, from=F, to=U, , "\dashv" rotate=-90]
\end{tikzcd}
\]
\vspace{0.3pc}

between the category of complete right dg $C$-contramodules and the category of right dg $A$-modules, where $\mathrm{Hom}^\alpha(-,-)$ denotes the twisted hom space, see \cite[Section 2.1]{LodayVallette12} for an account of this twisted hom space. Once again, one can check that this adjunction corresponds to the complete bar-cobar adjunction relative to $\alpha$. When we endow the category of right dg $A$-modules with the \textit{injective model structure} and the category of complete right dg $C$-contramodules with the transferred model structure, this becomes a Quillen adjunction. 

\medskip

The two structures defined above are compatible under the following duality square.

\begin{Proposition}
The following square diagram 
\[
\begin{tikzcd}[column sep=5pc,row sep=5pc]
\left(\mathsf{dg}~A\text{-}\mathsf{mod}_{l}\right)^{\mathsf{op}} \arrow[r,"C \otimes_{\alpha} -"{name=B},shift left=1.1ex] \arrow[d,"(-)^\circ "{name=SD},shift left=1.1ex ]
&\left(\mathsf{dg}~C\text{-}\mathsf{comod}_{l}\right)^{\mathsf{op}} \arrow[d,"(-)^*"{name=LDC},shift left=1.1ex ] \arrow[l,"A \otimes_{\alpha} -"{name=C},,shift left=1.1ex]  \\
\mathsf{dg}~A\text{-}\mathsf{mod}_{r} \arrow[r,"\mathrm{Hom}^{\alpha}(C \text{,}-)"{name=CC},shift left=1.1ex]  \arrow[u,"(-)^*"{name=LD},shift left=1.1ex ]
&\mathsf{dg}~C\text{-}\mathsf{contra}_{r}^{\mathsf{comp}} \arrow[l,"\mathrm{Hom}^{\alpha}(A \text{,}-)"{name=CB},shift left=1.1ex] \arrow[u,"(-)^\vee"{name=TD},shift left=1.1ex] \arrow[phantom, from=SD, to=LD, , "\dashv" rotate=0] \arrow[phantom, from=C, to=B, , "\dashv" rotate=-90]\arrow[phantom, from=TD, to=LDC, , "\dashv" rotate=0] \arrow[phantom, from=CC, to=CB, , "\dashv" rotate=-90]
\end{tikzcd}
\] 

is commutative and is made of Quillen adjunctions.
\end{Proposition}

\begin{proof}
This is a direct consequence of Theorem \ref{thm: homotopical magical square}.
\end{proof}

One can view the adjunction 
\[
\begin{tikzcd}[column sep=7pc,row sep=3pc]
           \left(\mathsf{dg}~C\text{-}\mathsf{comod}_{l}\right)^{\mathsf{op}} \arrow[r,"(-)^*"{name=F}, shift left=1.1ex] 
           &\mathsf{dg}~C\text{-}\mathsf{contra}_{r}^{\mathsf{comp}} ~. \arrow[l, shift left=.75ex, "(-)^\vee"{name=U}]
            \arrow[phantom, from=F, to=U, , "\dashv" rotate=90]
\end{tikzcd}
\]

as a contravariant version of the \textit{co-contra correspondence} constructed of \cite{PositselskiTwoKinds}. Indeed, this adjunction identifies finitely generated cofree $C$-comodules and finitely generated $C$-contramodules: let $M$ be a dg module which is degree-wise finite dimensional and bounded above or below, we have that 

\[
(C \otimes M)^* \cong \mathrm{Hom}(C,M^*) \quad \text{and} \quad \left(\mathrm{Hom}(C,M) \right)^\vee \cong C \otimes M~.
\]
\vspace{0.2pc}

Furthermore, when both of these categories are endowed with the canonical model structure transferred using the twisting morphism $\iota: C \longrightarrow \Omega C$, then this correspondences becomes a correspondence between the finitely generated fibrant-cofibrant objects of each of these categories. Therefore it gives a contravariant version of the \textit{derived co-contra correspondence}, see again \cite{PositselskiTwoKinds} for more on this. 

\begin{Remark}
The data of a curved cooperad concentrated in arity one is equivalent to the data of a curved counital coassociative coalgebra. This subsection can be generalized \textit{mutatis mutandis} to the curved setting.
\end{Remark}

\subsection{Absolute associative algebras and absolute Lie algebras}\label{subsection: absolute associative}
In this subsection, we explore the notion of algebras over the conilpotent cooperad $\mathcal{A}ss^*$, which we call \textit{absolute associative algebras}. We compare this category with the standard notion of non-unital associative algebras, encoded by the operad $\mathcal{A}ss$. Furthermore, we also introduce \textit{absolute Lie algebras}, encoded by the cooperad $\mathcal{L}ie^*$, and construct the universal enveloping absolute algebra functor. These two examples provide supplementary intuition on the notion of an algebra over a cooperad, in one of its simplest cases. 

\begin{Definition}[dg absolute associative algebra]
A \textit{dg absolute associative algebra} $(A,\gamma_A,d_A)$ is the data of a dg $\mathcal{A}ss^*$-algebra. 
\end{Definition}

Let us unravel this definition. Recall that a dg $\mathcal{A}ss^*$-algebra $(A,\gamma_A,d_A)$ is the data of a dg module $(A,d_A)$ together with a structural morphism of dg modules
\[
\gamma_A: \displaystyle \prod_{n \geq 0} \mathrm{Hom}_{\mathbb{S}_n}(\mathcal{A}ss^*(n), A^{\otimes n}) \longrightarrow A~,
\]
which satisfies the axioms of Definition \ref{def pdg C algebra}. 

\begin{lemma}
There is an isomorphism of dg modules

\[
\displaystyle \prod_{n \geq 0} \mathrm{Hom}_{\mathbb{S}_n}(\mathcal{A}ss^*(n), A^{\otimes n}) \cong \displaystyle \prod_{n \geq 1} A^{\otimes n}~.
\]
\end{lemma}

\begin{proof}
The $\mathbb{S}_n$-module $\mathcal{A}ss^*(n)$ is given by the regular representation of $\mathbb{S}_n$, for $n \geq 1$, and by $0$, when $n=0$.
\end{proof}

\begin{Remark}
This implies that there is a monad structure on the endofunctor 

\[
\prod_{n \geq 1} (-)^{\otimes n}: \mathsf{dg}~\mathsf{mod} \longrightarrow \mathsf{dg}~\mathsf{mod}~,
\]

such that a dg absolute associative algebra is the data of an algebra over this monad.
\end{Remark}

The structural map $\gamma_A$ of an dg absolute associative algebra $A$ is equivalently given by a morphism of dg modules 

\[
\gamma_A: \displaystyle \prod_{n \geq 1} A^{\otimes n} \longrightarrow A~.
\]

It associates to any series 
\[
\displaystyle \sum_{n \geq 1} \sum_{i \in \mathrm{I}_n } a_1^{(i)} \otimes \cdots \otimes a_n^{(i)}~,
\]
where the $a_j$ are elements in $A$ and where $\mathrm{I}_n$ is a finite set, a well-defined element
\[
\gamma_A \left(\displaystyle \sum_{n \geq 1} \sum_{i \in \mathrm{I}_n } a_1^{(i)} \otimes \cdots \otimes a_n^{(i)} \right) \quad \text{in} \quad A~,
\]
without \textbf{presupposing any topology} on the dg module $A$. 

\begin{Remark}
For a general dg absolute associative algebra $(A,\gamma_A,d_A)$, notice that 
\[
\gamma_A \left(\displaystyle \sum_{n \geq 1} \sum_{i \in \mathrm{I}_n } a_1^{(i)} \otimes \cdots \otimes a_n^{(i)} \right) \neq \displaystyle \sum_{n \geq 1} \sum_{i \in \mathrm{I}_n } \gamma_A\left(a_1^{(i)} \otimes \cdots \otimes a_n^{(i)}\right)~, 
\]
as the latter sum is not even well-defined in $A$.
\end{Remark}

\textbf{Differential condition:}
The condition that $\gamma_A$ is a morphism of dg modules can be rewritten as 

\begin{equation*}\label{condition: dg}
\begin{small}
\gamma_A \left(\displaystyle \sum_{n \geq 1} \sum_{i \in \mathrm{I}_n } \sum_{j = 0}^n (-1)^j  a_1^{(i)} \otimes \cdots \otimes d_A\left(a_j^{(i)}\right)\otimes \cdots \otimes a_n^{(i)} \right) = d_A \left(\gamma_A \left(\displaystyle \sum_{n \geq 1} \sum_{i \in \mathrm{I}_n } a_1^{(i)} \otimes \cdots \otimes a_n^{(i)} \right) \right)~,
\end{small}
\end{equation*}

for any series in  $\prod_{n \geq 1} A^{\otimes n}$.

\medskip

\textbf{Associativity condition:}
The condition that the structural morphism $\gamma_A$ defines an algebra over the monad $\prod_{n \geq 1} (-)^{\otimes n}$ can be rewritten as 

\begin{align*}\label{condition: associativité}
\gamma_A \left(\displaystyle \sum_{k \geq 1} \gamma_A \left(\displaystyle \sum_{i_1 \geq 1} \sum_{j_1 \in \mathrm{I}^{(1)}_{i_1} } a_1^{(1,j_1)} \otimes \cdots \otimes a_{i_1}^{(1,j_1)} \right) \otimes \cdots \otimes \gamma_A \left(\displaystyle \sum_{i_k \geq 1} \sum_{j_k \in \mathrm{I}^{(k)}_{i_k} } a_1^{(k,j_k)} \otimes \cdots \otimes a_{i_k}^{(k,j_k)} \right) \right)\\
=\gamma_A \left(\displaystyle \sum_{n \geq 1} \sum_{k \geq 1} \sum_{i_1 + \cdots + i_k = n} \sum_{j_1 \in  \mathrm{I}^{(1)}_{i_1}~, \cdots~, j_k \in \mathrm{I}^{(k)}_{i_k}} a_1^{(1,j_1)} \otimes \cdots \otimes a_{i_k}^{(1,j_1)} \otimes \cdots \otimes a_{n - i_k}^{(k,j_k)} \otimes \cdots \otimes a_{n}^{(k,j_k)} \right)~.
\end{align*}

\begin{Example}
Let $(V,d_V)$ be a dg module. The \textit{free} dg absolute associative algebra on $V$ is given by 
\[
\overline{\mathcal{T}}^{~\wedge}(V) \coloneqq \prod_{n \geq 1} V^{\otimes n}~,
\]

that is, the completed non-unital tensor algebra on $V$. If $V$ is a $\kk$-vector space of dimension $n$, this algebra is the algebra of non-commutative formal power series in $n$ variables without constant terms. 
\end{Example}

\begin{Example}[Convolution absolute associative algebra]\label{Example: convolution absolute algebra}
Let $(D,\Delta,d_D)$ be a non-counital \textit{conilpotent} dg coassociative coalgebra and let $(B,\mu,d_B)$ denote a non-unital dg associative algebra. It is well-known that the dg module of graded morphisms 
\[
\left(\mathrm{hom}(D,B),\partial \right)
\]
admits a \textit{convolution dg associative algebra structure}, where the product of $f:D \longrightarrow B$ and $g: D \longrightarrow B$ is given by 
\[
f \star g \coloneqq \left\{
\begin{tikzcd}[column sep=3.5pc,row sep=0pc]
D \arrow[r,"\Delta"] 
&D \otimes D  \arrow[r,"f ~\otimes~ g"] 
&B  \otimes B \arrow[r,"\mu"] 
&B
\end{tikzcd} 
\right\}~.
\]

This convolution product can be extended into an dg \textit{absolute} associative algebra structure: let $\mu^n$ (resp. $\Delta^n$) denote the $n$-iterated product (resp. coproduct), we define
\[
\gamma_{\mathrm{hom}(D,B)} \coloneqq  \left\{
\begin{tikzcd}[column sep=4pc,row sep=0.5pc]
\displaystyle \prod_{n \geq 1} \mathrm{hom}(D,B)^{\otimes n} \arrow[r]
&\mathrm{hom}(D,B) \\
\displaystyle \sum_{n \geq 1} \sum_{i \in \mathrm{I}_n } f_1^{(i)} \otimes \cdots \otimes f_n^{(i)} \arrow[r,mapsto]
&\displaystyle \sum_{n \geq 1} \sum_{i \in \mathrm{I}_n } \mu^n \circ \left( f_1^{(i)} \otimes \cdots f_n^{(i)} \right) \circ \Delta^n~.
\end{tikzcd} 
\right.
\]

The latter infinite sum is a well-defined linear map since, for any element $d$ in $D$, the sum is finite. Indeed, $D$ is conilpotent, and therefore the image of $d$ under iterated coproducts is eventually zero. One can check that $\gamma_{\mathrm{hom}(D,B)}$ satisfies the axioms of a dg absolute associative algebra. 

\medskip

Convolution algebras are an essential part of the Koszul duality for associative algebras defined by S. Priddy in \cite{Priddy70} since they encode twisting morphisms. For the same reason, they play a role in defining twisted tensor products as done by E. Brown in \cite{Brown}. They also appear in Quantum Algebra, see \cite{Kassel}.
\end{Example}

\begin{Remark}[Eilenberg--Mazur swindle and units]
Suppose that there was a "unital" version of absolute associative algebras, suppose $A$ was one of those and take $a \in A$. If we denote $1$ the unit of $A$, then we can consider the image of the following sum in $A$
\[
\gamma_A \left(\sum_{n \geq 1} \underbrace{1 \otimes \cdots \otimes 1}_{n-1} \otimes a \right) = \gamma_A \left(\sum_{n \geq 1} a \right)~.
\]
Now, using the associativity, we can prove that
\[
\gamma_A \left(\sum_{n \geq 1} a \right) = \gamma_A\left(\gamma_A \left(\sum_{n \geq 1} a \right) \otimes 1\right) = \gamma_A \left(\sum_{n \geq 2} a \otimes 1 \right) = \gamma_A \left(\sum_{n \geq 2} a \right)~. 
\]
Therefore we get that 
\[
\gamma_A \left(\sum_{n \geq 1} a \right) = a + \gamma_A \left(\sum_{n \geq 2} a \right) =a + \gamma_A \left(\sum_{n \geq 1} a \right)~,
\]
which implies that $a = 0$. 
\end{Remark}

Let us now compare \textit{absolute associative algebras} with (non-unital) associative algebras.

\begin{lemma}\label{lemma: morphism of monads}
The inclusion 

\[
\iota: \bigoplus_{n \geq 1} (-)^{\otimes n} \hookrightarrow \prod_{n \geq 1} (-)^{\otimes n}
\]

defines a morphism of monads in the category of dg modules.
\end{lemma}

\begin{proof}
It is straightforward to check. 
\end{proof}

\begin{Proposition}\label{prop: adjunction absolute envelope}
There is an adjunction 

\[
\begin{tikzcd}[column sep=7pc,row sep=3pc]
            \mathsf{dg}~\mathsf{abs}~\mathsf{assoc}\text{-}\mathsf{alg}\arrow[r,"\mathrm{Res}"{name=F}, shift left=1.1ex] 
           &\mathsf{dg}~\mathsf{assoc}\text{-}\mathsf{alg}~, \arrow[l, shift left=.75ex, "\mathrm{Abs}"{name=U}]
            \arrow[phantom, from=F, to=U, , "\dashv" rotate=90]
\end{tikzcd}
\]

between the category of dg absolute associative algebras and the category of dg associative algebras.
\end{Proposition}

\begin{proof}
This follows from Lemma \ref{lemma: morphism of monads}.
\end{proof}

Let us describe these functors. Let $(A,\gamma_A,d_A)$ be a dg absolute associative algebra. Its restriction algebra $(\mathrm{Res}(A), \mathrm{Res}(\gamma_A), d_A)$ is given by the dg module $(A,d_A)$ together with the structural map
\[
\begin{tikzcd}
\displaystyle \mathrm{Res}(\gamma_A): \bigoplus_{n \geq 1} A^{\otimes n} \arrow[r,"\iota_A"]
&\displaystyle \prod_{n \geq 1} A^{\otimes n} \arrow[r,"\gamma_A"]
&A~,
\end{tikzcd}
\]
which endows $\mathrm{Res}(A)$ with a dg associative algebra structure. Let $(B,\gamma_B,d_B)$ be a dg associative algebra, the image $\mathrm{Abs}(B)$ is given by the following coequalizer 
\[
\begin{tikzcd}[column sep=4pc,row sep=4pc]
\mathrm{Coeq}\Bigg(\displaystyle \prod_{n \geq 1} \Bigg(\bigoplus_{k \geq 1} B^{\otimes k}\Bigg)^{\otimes n} \arrow[r,"\prod_{n \geq 0} (\gamma_B)^{\otimes n}",shift right=1.1ex,swap]  \arrow[r,"\psi_B"{name=SD},shift left=1.1ex ]
&\prod_{n \geq 1} B^{\otimes n}\Bigg)~,
\end{tikzcd}
\]

in the category of dg modules, where $\psi_B$ is given by
\[
\begin{tikzcd}[column sep=4pc,row sep=4pc]
\displaystyle \prod_{n \geq 1} \left(\bigoplus_{k \geq 1} B^{\otimes k}\right)^{\otimes n} \arrow[r,"\prod_{n \geq 0} (\iota_B)^{\otimes n}"]
&\displaystyle\prod_{n \geq 1} \left(\prod_{k \geq 1} B^{\otimes k}\right)^{\otimes n} \arrow[r, "\gamma_{\prod_{k \geq 1} B^{\otimes k}}"]
&\displaystyle \prod_{n \geq 1} B^{\otimes n}~.
\end{tikzcd}
\]

\textbf{Completeness:}
Any dg absolute associative algebra $(A,\gamma_A,d_A)$ comes equipped with a canonical decreasing filtration given by 
\[
\mathrm{W}_\omega A \coloneqq \mathrm{Im}\left(\gamma_A^{\geq \omega}: \prod_{n \geq \omega + 1} A^{\otimes n} \longrightarrow A \right)~.
\]

This filtration is the same filtration defined in Definition \ref{def canonical filtration} since the cooperad $\mathcal{A}ss^*$ is a binary cooperad. A dg absolute associative algebra $(A,\gamma_A,d_A)$ is therefore said to be \textit{complete} if the canonical epimorphism
\[
\varphi_A: A \twoheadrightarrow \lim_{\omega} A/\mathrm{W}_\omega A
\]

is an isomorphism of dg absolute associative algebras. Composing the adjunction of Proposition \ref{prop: adjunction absolute envelope} with the completion functor produces an adjunction 
\[
\begin{tikzcd}[column sep=7pc,row sep=3pc]
            \mathsf{dg}~\mathsf{abs}~\mathsf{assoc}\text{-}\mathsf{alg}^{\mathsf{comp}}\arrow[r,"\mathrm{Res}"{name=F}, shift left=1.1ex] 
           &\mathsf{dg}~\mathsf{assoc}\text{-}\mathsf{alg}~, \arrow[l, shift left=.75ex, "\mathrm{cAbs}"{name=U}]
            \arrow[phantom, from=F, to=U, , "\dashv" rotate=90]
\end{tikzcd}
\]
between complete dg absolute associative algebras and dg associative algebras. By applying Theorem \ref{thm: restriction functor fully faithful}, we get that the restriction functor $\mathrm{Res}$ here is fully faithful. Thus complete dg absolute associative algebras are a reflective subcategory of dg associative algebras.

\begin{Remark}
Let $(A,\gamma_A,d_A)$ be a complete dg absolute associative algebra, then  
\[
\gamma_A \left(\displaystyle \sum_{n \geq 1} \sum_{i \in \mathrm{I}_n } a_1^{(i)} \otimes \cdots \otimes a_n^{(i)} \right) = \displaystyle \sum_{n \geq 1} \sum_{i \in \mathrm{I}_n } \gamma_A\left(a_1^{(i)} \otimes \cdots \otimes a_n^{(i)}\right)~, 
\]
using the completeness of $A$. 
\end{Remark}

\textbf{Completion of associative algebras.} Similarly, when $(B,\gamma_B,d_B)$ is a dg associative algebra, one has a canonical decreasing filtration as well given by
\[
\mathrm{F}_\omega B \coloneqq \mathrm{Im}\left(\gamma_B^{\geq \omega}: \bigoplus_{n \geq \omega + 1} B^{\otimes n} \longrightarrow B \right)~.
\]

A dg associative algebra $(B,\gamma_B,d_B)$ is said to be \textit{complete} if the canonical morphism 
\[
\lambda_B: B \longrightarrow \widehat{B} = \lim_{\omega} B/ \mathrm{F}_\omega B 
\]
is an isomorphism of dg associative algebras, where the right hand side $\widehat{B}$ is its \textit{completion}.

\medskip

\textbf{Comparison between the two notions of completion.} So, since on the one hand, one has complete dg absolute associative algebras and on the other, one can complete dg associative algebras, it is natural to wonder if these two constructions agree. We show that while both notions of nilpotency agree, these two categories are not equivalent in general.

\begin{Definition}[Nilpotent dg absolute associative algebra]
A dg absolute associative algebra $(A,\gamma_A,d_A)$ is said to be \textit{nilpotent} if there exists $\omega_{0} \geq 1$ such that $\mathrm{W}_{\omega_0} A = 0$. The integer $\omega_0$ is called the \textit{nilpotency degree}.
\end{Definition}

\begin{Remark}
A nilpotent dg absolute associative algebra is in particular complete.
\end{Remark}

\begin{Definition}[Nilpotent dg associative algebra]
A dg associative algebra $(B,\gamma_B,d_B)$ is said to be \textit{nilpotent} if there exists $\omega_{0} \geq 1$ such that $\mathrm{F}_{\omega_0} B = 0$. The integer $\omega_0$ is called the nilpotency degree.
\end{Definition}

\begin{Remark}
This definition coincides with the standard notion of nilpotency for associative algebras present in the literature.
\end{Remark}

\begin{Proposition}\label{prop: les nilpotentes sont des absolues}
The data of a nilpotent dg absolute associative algebra is equivalent to the data of a nilpotent dg associative algebra with the same nilpotency degree.
\end{Proposition}

\begin{proof}
Let $\omega_0 \geq 1$. A nilpotent dg absolute associative algebra $(A,\gamma_A,d_A)$ with nilpotency degree $\omega_0$ amounts to the data of an algebra over the monad $\prod_{n = 1}^{\omega_0 + 1} (-)^{\otimes n}$. On the other hand, a nilpotent dg associative algebra $(B,\gamma_B,d_B)$ amounts to the data of an algebra over the monad $\bigoplus_{n = 1}^{\omega_0 + 1} (-)^{\otimes n}$. There is an isomorphism of monads 
\[
\prod_{n = 1}^{\omega_0 + 1} (-)^{\otimes n} \cong \bigoplus_{n = 1}^{\omega_0 + 1} (-)^{\otimes n}~,
\]
since the product only involves a finite amount of terms.
\end{proof}

\begin{Remark}\label{Remark: nilpotent Yoneda}
The image of a nilpotent dg associative algebra $(B,\gamma_B,d_B)$ via the complete absolute envelope functor $\mathrm{cAbs}$ is simply given by $(B,\gamma_B,d_B)$ considered as a complete dg absolute associative algebra. This follows from the general case treated in Proposition \ref{Prop: équivalences entre arity nilpotentes}.
\end{Remark}

\begin{Corollary}
The completion of dg associative algebras defines a functor 
\[
\widehat{(-)} : \mathsf{dg}~\mathsf{assoc}\text{-}\mathsf{alg} \longrightarrow\mathsf{dg}~\mathsf{abs}~\mathsf{assoc}\text{-}\mathsf{alg}^{\mathsf{comp}}
\]
from the category of dg associative algebras to the category of complete dg absolute associative algebras. 
\end{Corollary}

\begin{proof}
Follows from the fact that for any dg associative algebra $B$, its completion 
\[
\widehat{B} = \lim_{\omega} B/ \mathrm{F}_\omega B
\]
is given as a limit of nilpotent dg absolute associative algebras, which are complete dg absolute associative algebras, using that complete dg absolute associative algebras are stable under limits. See Proposition \ref{Prop: arity completion defines a functor} for the general case.
\end{proof}

So the completion of a dg associative algebra is always a complete dg absolute associative algebra. However, the inverse is not always true, see Subsection \ref{subsection: comparison between the filtrations} for an explanation of the general case. In terms of filtrations, there is a simple comparison between the two filtrations, but they do not always agree.  

\begin{Proposition}
Let $(A,\gamma_A,d_A)$ be a dg absolute associative algebra. There is a monomorphism of dg modules 
\[
\mathrm{F}_\omega \mathrm{Res}(A) \rightarrowtail \mathrm{W}_\omega A~.
\]

Therefore, if $A$ is a complete dg absolute associative algebra, the topology induced by the canonical filtration on $(\mathrm{Res}(A), \mathrm{Res}(\gamma_A), d_A)$ is separated. 
\end{Proposition}

\begin{proof}
It is straightforward to notice that $\mathrm{F}_\omega \mathrm{Res}(A)$ consists of finite sums of products of arity greater than $\omega +1$, hence it is included in $\mathrm{W}_\omega A$. For the second point, notice that if $(A,\gamma_A,d_A)$ is complete, then 

\[
\bigcap_{\omega \geq 0}\mathrm{F}_\omega \mathrm{Res}(A) \hookrightarrow \bigcap_{\omega \geq 0}\mathrm{W}_\omega A = \{0\}~,
\]
hence the canonical morphism of dg associative algebras 
\[
\lambda_{\mathrm{Res}(A)}: \mathrm{Res}(A) \longrightarrow \lim_{\omega} \left(\mathrm{Res}(A)/\mathrm{F}_\omega \mathrm{Res}(A)\right)
\]
is a monomorphism and the topology induced by $\mathrm{F}_\omega \mathrm{Res}(A)$ is separated.
\end{proof}

\begin{Counterexample}
Suppose $(A,\gamma_A,d_A)$ is a complete dg absolute associative algebra, then $\mathrm{Res}(A)$ might not be complete as a dg associative algebra. This is for instance the case of the free absolute associative algebra on a $\kk$-vector space $V \cong \bigoplus_{i \geq 1} \kk.x_i$ with a countable basis. Its restriction algebra $\mathrm{Res}(\prod_{n \geq  1} V^{\otimes n})$ is not complete as a dg associative algebra. Indeed, one can check that the canonical morphism $\lambda_{\prod_{n \geq 1}V^{\otimes n}}$ is not an epimorphism in this case, as formal power series 
\[
\sum_{i \geq 1} x_i^{i+1} = x_1^2 + x_2^3 + x_3^4 + \cdots 
\]
is in $\mathrm{W}_2 A$ but not in $\mathrm{F}_2 \mathrm{Res}(A)$. This is also very similar to what happens in \cite[110.7 Noncomplete completion]{stackproject}.
\end{Counterexample}

\textbf{A duality square for dg associative algebras.} Let us describe explicitly a duality square one can construct in order to relate dg associative algebras with dg absolute associative algebras.

\medskip

In the following, we consider dg associative algebras endowed with the standard model structure where weak equivalences are given by quasi-isomorphisms and fibrations by epimorphisms. We endow conilpotent dg coassociative coalgebras with the transferred model structure via the bar-cobar adjunction constructed by K. Lefèvre-Hasegawa in \cite{LefevreHasegawa03}. We endow \textit{non-necessarily conilpotent} dg coassociative coalgebras with the model structure constructed by E. Getzler and P. Goerss in \cite{GetzlerGoerss99}, where weak equivalences are given by quasi-isomorphisms and cofibrations by monomorphisms. This latter structure can be obtained directly using the coassociative coalgebra structure on the interval object of dg modules and the theory developed in \cite{lefttransfer}. Finally, this structure can be transferred via the complete bar-cobar adjuntion onto dg absolute associative algebras, as it was shown in \cite{grignoulejay18}.

\begin{Proposition}
The square of Quillen adjunctions 
\[
\begin{tikzcd}[column sep=5pc,row sep=5pc]
\left(\mathsf{dg}~\mathsf{assoc}\text{-}\mathsf{alg}\right)^{\mathsf{op}} \arrow[r,"\mathrm{B}^{\mathsf{op}}"{name=B},shift left=1.1ex] \arrow[d,"(-)^\circ "{name=SD},shift left=1.1ex ]
&\left(\mathsf{dg}~\mathsf{coassoc}\text{-}\mathsf{coalg}^{\mathsf{conil}}\right)^{\mathsf{op}} \arrow[d,"(-)^*"{name=LDC},shift left=1.1ex ] \arrow[l,"\Omega^{\mathsf{op}}"{name=C},,shift left=1.1ex]  \\
\mathsf{dg}~\mathsf{coassoc}\text{-}\mathsf{coalg} \arrow[r,"\widehat{\Omega} "{name=CC},shift left=1.1ex]  \arrow[u,"(-)^*"{name=LD},shift left=1.1ex ]
&\mathsf{dg}~\mathsf{abs}~\mathsf{assoc}\text{-}\mathsf{alg}^{\mathsf{comp}}~, \arrow[l,"\widehat{\mathrm{B}}"{name=CB},shift left=1.1ex] \arrow[u,"(-)^\vee"{name=TD},shift left=1.1ex] \arrow[phantom, from=SD, to=LD, , "\dashv" rotate=0] \arrow[phantom, from=C, to=B, , "\dashv" rotate=-90]\arrow[phantom, from=TD, to=LDC, , "\dashv" rotate=0] \arrow[phantom, from=CC, to=CB, , "\dashv" rotate=-90]
\end{tikzcd}
\] 
commutes in the following sense: right adjoints going from the top right to the bottom left are naturally isomorphic.
\end{Proposition}

\begin{proof}
The existence of this commuting square of adjunctions follows directly from Theorem \ref{thm: magical square} applied to the twisting morphism $\kappa: \mathcal{A}ss^{\ac} \longrightarrow \mathcal{A}ss$.

\medskip

The fact that these adjunctions are Quillen adjunction can be shown using the same arguments as in the proof of Theorem \ref{thm: homotopical magical square} applied to this case.
\end{proof}

\begin{Remark}
This duality square intertwines the classical bar construction of \cite{CartanEilenberg56} and the classical cobar construction of \cite{Adams56}, and the Sweedler dual functor constructed in \cite{Sweedler69} with the theory of dg absolute algebras. The first two have played a seminal role in Algebra and in Algebraic Topology since their introduction. The Sweedler dual functor plays an important role in the theory of Hopf algebras. Notice that it was not known to be a right Quillen functor before. Also notice that this square shows that the linear dual of the classical bar construction for dg associative algebras is \textit{naturally a dg absolute associative algebra}.
\end{Remark}

\begin{Remark}
It was pointed out to us by L. Positselski that the complete bar construction only depends on the underlying associative algebra structure of the absolute associative algebra, as the twisting morphism $\kappa: s\mathcal{A}ss^* \longrightarrow \mathcal{A}ss^*$ only sees the binary product. It therefore coincides with the bar construction of M. Anel and A. Joyal of \cite{AnelJoyal}. Although this complete bar-cobar construction is not known to be a Quillen equivalence, it was used by A. Guan and A. Lazarev in \cite{GuanLazarev} to establish Koszul duality type of equivalences between exotic derived categories of dg $A$ modules and dg $\widehat{\mathrm{B}}(A)$-comodules. 
\end{Remark}

\begin{Remark}
An analogous duality square can be constructed by replacing dg (co)associative (co)algebras with $\mathcal{A}_\infty$-(co)algebras on the left hand side of the duality square. In this case, the complete bar-cobar adjunction is known to be a Quillen equivalence.
\end{Remark}

\textbf{Universal enveloping absolute algebra functor.} As with operads, a morphism of cooperads induces an adjunction between their respective categories of algebras. This allows us to construct the universal enveloping absolute algebra of a dg absolute Lie algebra. 

\begin{Definition}[dg absolute Lie algebra]
A \textit{dg absolute Lie algebra} $(\mathfrak{g},\gamma_{\mathfrak{g}},d_\mathfrak{g})$ is the data of a dg $\mathcal{L}ie^*$-algebra.
\end{Definition}

There is a morphism of operads $\mathcal{S}kew: \mathcal{L}ie \longrightarrow \mathcal{A}ss$ given by the skew-symmetrization of the associative product. By taking the linear dual $\mathcal{S}kew^*: \mathcal{A}ss^* \longrightarrow \mathcal{L}ie^*$, we obtain a morphism of conilpotent cooperads. 

\begin{Proposition}\label{prop: adjunction universal enveloping absolute algebra}
There is an adjunction
\[
\begin{tikzcd}[column sep=7pc,row sep=3pc]
            \mathsf{dg}~\mathsf{abs}~\mathsf{Lie}\text{-}\mathsf{alg}^{\mathsf{comp}} \arrow[r,"\widehat{\mathfrak{U}}"{name=F}, shift left=1.1ex] 
           &\mathsf{dg}~\mathsf{abs}~\mathsf{assoc}\text{-}\mathsf{alg}^{\mathsf{comp}}~, \arrow[l, shift left=.75ex, "\mathrm{Res}_{\mathcal{S}kew^*}"{name=U}]
            \arrow[phantom, from=F, to=U, , "\dashv" rotate=-90]
\end{tikzcd}
\]

between the category of complete dg absolute Lie algebras and the category of complete dg absolute associative algebras. The left adjoint functor $\widehat{\mathfrak{U}}$ is called the universal enveloping absolute algebra functor. 
\end{Proposition}

\begin{proof}
The morphism $\mathcal{S}kew^*: \mathcal{A}ss^* \longrightarrow \mathcal{L}ie^*$ induces a natural transformation 

\[
\widehat{\mathscr{S}}^c(\mathcal{S}kew^*): \widehat{\mathscr{S}}^c(\mathcal{L}ie^*) \longrightarrow \widehat{\mathscr{S}}^c(\mathcal{A}ss^*)
\]

between the corresponding monads encoding dg absolute associative algebras and dg absolute Lie algebras.
\end{proof}

\begin{Remark}
Let $(\mathfrak{g},\gamma_{\mathfrak{g}},d_\mathfrak{g})$ be a dg absolute Lie algebra. The universal enveloping absolute algebra $\widehat{\mathfrak{U}}(\mathfrak{g})$ of $\mathfrak{g}$ is given by 
\[
\begin{tikzcd}[column sep=4pc,row sep=4pc]
\mathrm{Coeq}\Bigg(\displaystyle  \overline{\mathcal{T}}^{~\wedge}\bigg(\widehat{\mathcal{L}ie}(\mathfrak{g})\bigg) \arrow[r,"\overline{\mathcal{T}}^{~\wedge}(\gamma_{\mathfrak{g}})",shift right=1.1ex,swap]  \arrow[r,"\psi_\mathfrak{g}"{name=SD},shift left=1.1ex ]
&\overline{\mathcal{T}}^{~\wedge}(\mathfrak{g})\Bigg)~,
\end{tikzcd}
\]

where $\widehat{\mathcal{L}ie}(\mathfrak{g})$ denotes the free completed Lie algebra on $\mathfrak{g}$ and where $\psi_\mathfrak{g}$ is given by the composition
\[
\begin{tikzcd}[column sep=4.5pc,row sep=4pc]
\psi_\mathfrak{g}: \overline{\mathcal{T}}^{~\wedge}(\widehat{\mathcal{L}ie}(\mathfrak{g})) \arrow[r,"\overline{\mathcal{T}}^{~\wedge}(\mathcal{S}kew^*)"]
&\overline{\mathcal{T}}^{~\wedge}(\overline{\mathcal{T}}^{~\wedge}(\mathfrak{g})) \arrow[r, "\gamma_{\overline{\mathcal{T}}^{~\wedge}(\mathfrak{g})}"]
&\overline{\mathcal{T}}^{~\wedge}(\mathfrak{g})~.
\end{tikzcd}
\]

Hence $\widehat{\mathfrak{U}}(\mathfrak{g})$ is quotient of the completed tensor algebra $\overline{\mathcal{T}}^{~\wedge}(\mathfrak{g})$ on $\mathfrak{g}$ where one not only identifies $x \otimes y - (-1)^{|x|} y \otimes x$ with the Lie bracket $[x,y]$, but this identification has also to be done for all formal power series of brackets of elements of $\mathfrak{g}$.
\end{Remark}

\begin{Example}[Convolution absolute Lie algebra]
Recall the \textit{convolution absolute algebra} structure on $\mathrm{hom}(D,B)$ constructed in Example \ref{Example: convolution absolute algebra}, where $D$ is a conilpotent dg coassociative coalgebra and $B$ a dg associative algebra. Its skew-symmetrization is naturally a dg absolute Lie algebra, which encodes twisting morphisms between $D$ and $B$ as its Maurer-Cartan elements.
\end{Example}

\begin{Proposition}\label{prop: carré qui commute}
Let $(\mathfrak{g},\gamma_{\mathfrak{g}},d_\mathfrak{g})$ be a nilpotent dg absolute Lie algebra. There is an isomorphism 

\[
\widehat{\mathfrak{U}}(\mathfrak{g}) \cong \frac{\overline{\mathcal{T}}^{~\wedge}(\mathfrak{g})}{\left(x \otimes y - (-1)^{|x||y|} y \otimes x - [x,y] \right)}
\]
\vspace{0.2pc}

of complete dg absolute associative algebras, where on the right hand side we consider the free absolute associative algebra on $\mathfrak{g}$, modulo the ideal generated by the expression inside the parenthesis.  
\end{Proposition}

\begin{proof}
The following square of adjunctions 
\[
\begin{tikzcd}[column sep=5pc,row sep=5pc]
\mathsf{dg}~\mathsf{assoc}\text{-}\mathsf{alg} \arrow[r,"\mathrm{cAbs}"{name=B},shift left=1.1ex] \arrow[d,"\mathrm{Skew}"{name=SD},shift left=1.1ex ]
&\mathsf{dg}~\mathsf{abs}~\mathsf{assoc}\text{-}\mathsf{alg}^{\mathsf{comp}} \arrow[d,"\mathrm{Skew}"{name=LDC},shift left=1.1ex ] \arrow[l,"\mathrm{Res}"{name=C},,shift left=1.1ex]  \\
\mathsf{dg}~\mathsf{Lie}\text{-}\mathsf{alg} \arrow[r,"\mathrm{cAbs}"{name=CC},shift left=1.1ex]  \arrow[u,"\mathfrak{U}"{name=LD},shift left=1.1ex ]
&\mathsf{dg}~\mathsf{abs}~\mathsf{Lie}\text{-}\mathsf{alg}^{\mathsf{comp}} \arrow[l,"\mathrm{Res}"{name=CB},shift left=1.1ex] \arrow[u,"\widehat{\mathfrak{U}}"{name=TD},shift left=1.1ex] \arrow[phantom, from=SD, to=LD, , "\dashv" rotate=0] \arrow[phantom, from=C, to=B, , "\dashv" rotate=-90]\arrow[phantom, from=TD, to=LDC, , "\dashv" rotate=0] \arrow[phantom, from=CC, to=CB, , "\dashv" rotate=-90]
\end{tikzcd}
\] 

commutes, where $\mathfrak{U}$ is the classical universal enveloping algebra functor. Indeed, one easily checks that 
\[
\mathrm{Res} \cdot \mathrm{Skew} \cong \mathrm{Skew} \cdot \mathrm{Res}~.
\]

This implies that $\widehat{\mathfrak{U}} \cdot \mathrm{cAbs} \cong \mathrm{cAbs} \cdot \mathfrak{U}$. Therefore, by Remark \ref{Remark: nilpotent Yoneda}, we have that 
\[
\widehat{\mathfrak{U}}(\mathfrak{g}) \cong \mathrm{cAbs}\left(\frac{\overline{\mathcal{T}}(\mathfrak{g})}{\left(x \otimes y - (-1)^{|x||y|} y \otimes x - [x,y] \right)} \right) \cong \frac{\overline{\mathcal{T}}^{~\wedge}(\mathfrak{g})}{\left(x \otimes y - (-1)^{|x||y|} y \otimes x - [x,y] \right)}~.
\]
\end{proof}

\begin{Remark}
The theory of ideals of algebras over cooperads is explained in \cite[Sections 4.2 and 4.3]{grignoulejay18}. In particular, for the ideal generated by a subobject, we refer to \cite[Definition 4.8 and Theorem 4.9]{grignoulejay18}.
\end{Remark}

\vspace{1.5pc}

\section{Applications to absolute Lie theory}

\vspace{1.5pc}

In this section, we give a generalization of one the main results obtained by R. Campos, D. Petersen, D. Robert-Nicoud and F. Wierstra in \cite{campos2020lie}, by rephrasing and reinterpreting their constructions with the language of algebras over cooperads, and by using the formalism developed so far. We now consider the dg operads $\Omega \mathcal{L}ie^*$ and $\Omega \mathcal{A}ss^*$ which encode, respectively, \textit{shifted} $\mathcal{C}_\infty$ and \textit{shifted} $\mathcal{A}_\infty$ algebras (and coalgebras). From now on, the adjective \textit{shifted} will be implicit.

\begin{lemma}
There is a morphism of dg operads
\[
\varphi: \Omega \mathcal{A}ss^* \longrightarrow \Omega \mathcal{L}ie^*~,
\]
given by $\varphi \coloneqq \Omega(\mathcal{S}kew^*)$.
\end{lemma}

\begin{proof}
It is straightforward from the definition.
\end{proof}

\begin{Proposition}
The morphism of dg operads $\varphi: \Omega \mathcal{A}ss^* \longrightarrow \Omega \mathcal{L}ie^*$ induces a Quillen adjunction
\[
\begin{tikzcd}[column sep=7pc,row sep=3pc]
           \mathcal{C}_\infty\text{-}\mathsf{coalg} \arrow[r,"\mathrm{Res}_{\varphi}"{name=F}, shift left=1.1ex] 
           &\mathcal{A}_\infty\text{-}\mathsf{coalg}~. \arrow[l, shift left=.75ex, "\mathrm{Coind}_{\varphi}"{name=U}]
            \arrow[phantom, from=F, to=U, , "\dashv" rotate=-90]
\end{tikzcd}
\]

between the model category of $\mathcal{C}_\infty$-coalgebras and the model category of $\mathcal{A}_\infty$-coalgebras.
\end{Proposition}

\begin{proof}
Since both dg operads $\Omega \mathcal{A}ss^*$ and $\Omega \mathcal{L}ie^*$ are cofibrant, their respective categories of dg coalgebras admit a left-transferred model structure where weak equivalences are given by quasi-isomorphisms and where cofibrations are given by degree-wise monomorphisms, by Theorem \ref{thm: model structure on P-cog}. Both categories are comonadic by Theorem \ref{thm: existence of the cofree P cog}, and the morphism $\varphi$ induces a morphism between their respective comonads. Thus it induces an adjunction. Finally, it is straightforward to check that $\mathrm{Res}_{\varphi}$ preserves cofibrations and quasi-isomorphisms, since it does not change the underlying chain complex of the $\mathcal{C}_\infty$-coalgebra.
\end{proof}

\begin{theorem}[{\cite[Theorem 4.27]{campos2020lie}}]\label{thm: cited thm about C infinity}
Let $C_1$ and $C_2$ be two $\mathcal{C}_\infty$-coalgebras. There exists a zig-zag of quasi-isomorphisms of $\mathcal{C}_\infty$-coalgebras 
\[
C_1 \lqi \cdot \qi C_2
\]
if and only if there exists a zig-zag of $\mathcal{A}_\infty$-coalgebras 
\[
\mathrm{Res}_{\varphi}(C_1) \lqi \cdot \qi \mathrm{Res}_{\varphi}(C_2)~.
\]
\end{theorem}

The above result is the main technical result of \cite{campos2020lie}. The main idea here is to pass to appropriate Koszul dual categories, namely those that arise as algebras over the Koszul dual cooperads, to derive an a priori stronger version \cite[Theorem B]{campos2020lie}. 

\begin{Proposition}\label{Cor: commutative of the square of absolute Lie}
There is a commuting square
\[
\begin{tikzcd}[column sep=5pc,row sep=5pc]
\mathcal{A}_\infty\text{-}\mathsf{coalg} \arrow[r,"\widehat{\Omega}_\iota"{name=B},shift left=1.1ex] \arrow[d,"\mathrm{Coind}_{\varphi} "{name=SD},shift left=1.1ex ]
&\mathsf{dg}~\mathsf{abs}~\mathsf{assoc}\text{-}\mathsf{alg}^{\mathsf{comp}} \arrow[d,"\mathrm{Skew}"{name=LDC},shift left=1.1ex ] \arrow[l,"\widehat{\mathrm{B}}_\iota"{name=C},,shift left=1.1ex]  \\
\mathcal{C}_\infty\text{-}\mathsf{coalg} \arrow[r,"\widehat{\Omega}_\iota "{name=CC},shift left=1.1ex]  \arrow[u,"\mathrm{Res}_{\varphi}"{name=LD},shift left=1.1ex ]
&\mathsf{dg}~\mathsf{abs}~\mathsf{Lie}\text{-}\mathsf{alg}^{\mathsf{comp}} \arrow[l,"\widehat{\mathrm{B}}_\iota"{name=CB},shift left=1.1ex] \arrow[u,"\widehat{\mathfrak{U}}"{name=TD},shift left=1.1ex] \arrow[phantom, from=SD, to=LD, , "\dashv" rotate=0] \arrow[phantom, from=C, to=B, , "\dashv" rotate=-90]\arrow[phantom, from=TD, to=LDC, , "\dashv" rotate=0] \arrow[phantom, from=CC, to=CB, , "\dashv" rotate=-90]
\end{tikzcd}
\] 

of Quillen adjunctions.
\end{Proposition}

\begin{proof}
It is immediate to check that the square 
\[
\begin{tikzcd}[column sep=3pc,row sep=3pc]
\mathcal{A}ss^* \arrow[r,"\iota "] \arrow[d,"\mathcal{S}kew^*",swap] 
&\Omega \mathcal{A}ss^* \arrow[d,"\varphi"]\\
\mathcal{L}ie^* \arrow[r,"\iota"]
&\Omega \mathcal{L}ie^*~.
\end{tikzcd}
\]

is commutative. The result follows from Proposition \ref{prop: compatibility conditions} (where here conilpotent dg cooperad are considered as conilpotent curved cooperads with zero curvature). 
\end{proof}

\begin{theorem}\label{thm: inclusion infini cat des Lie absolues}
Let $\mathfrak{g}$ and $\mathfrak{h}$ be two complete dg absolute Lie algebras. There exists a zig-zag of weak equivalences of dg absolute Lie algebras
\[
\mathfrak{g} \lqi \cdot \qi \mathfrak{h}
\]
if and only if there exists a zig-zag of weak equivalences of dg absolute associative algebras
\[
\widehat{\mathfrak{U}}(\mathfrak{g}) \lqi \cdot \qi \widehat{\mathfrak{U}}(\mathfrak{h})  ~.
\]
\end{theorem}

\begin{proof}
It is a direct consequence of Theorem \ref{thm: cited thm about C infinity}, using the fact that the horizontal Quillen adjunctions of Proposition \ref{Cor: commutative of the square of absolute Lie} are Quillen equivalences.
\end{proof}

Let us make more explicit what these weak equivalences of dg absolute associative algebras or dg absolute Lie algebras look like. We state the analogue of \cite[Proposition 2.5]{Vallette14} for the model structure on algebras over a conilpotent dg cooperad transferred along the complete bar-cobar adjunction.

\begin{Proposition}\label{thm: weak equiv inclues dans les quasi-isos}
The category of complete absolute dg Lie algebras admits a model structure right-transferred from dg modules, where fibrations are given by degree-wise epimorphisms and weak equivalences by quasi-isomorphisms. Furthermore, the identity induces a Quillen adjunction 
\[
\begin{tikzcd}[column sep=5pc,row sep=5pc]
\left(\mathsf{dg}~\mathsf{abs}~\mathsf{Lie}\text{-}\mathsf{alg}^{\mathsf{comp}}, \mathsf{W.eq} \right) \arrow[r,"\mathrm{Id}"{name=LDC},shift left=1.1ex ] 
&\left(\mathsf{dg}~\mathsf{abs}~\mathsf{Lie}\text{-}\mathsf{alg}^{\mathsf{comp}}, \mathsf{Q.iso} \right)~. \arrow[l,"\mathrm{Id}"{name=TD},shift left=1.1ex] \arrow[phantom, from=TD, to=LDC, , "\dashv" rotate=90] 
\end{tikzcd}
\] 
In particular, any weak equivalence of complete dg absolute Lie algebras transferred from $\mathcal{C}_\infty$-coalgebras is a quasi-isomorphism. 
\end{Proposition}

\begin{proof}
First notice that the dg operad $\Omega \mathcal{L}ie^*$ is augmented, that is, there is a morphism of dg operads $\varepsilon: \Omega \mathcal{L}ie^* \longrightarrow \I$. Now we consider the following commutative square 
\[
\begin{tikzcd}[column sep=3pc,row sep=3pc]
\mathcal{L}ie^* \arrow[r,"\iota "] \arrow[d,"\mathrm{id}_{\mathcal{L}ie^*}",swap] 
&\Omega \mathcal{L}ie^* \arrow[d,"\varepsilon"]\\
\mathcal{L}ie^* \arrow[r,"\epsilon_{*}\iota"]
&\I~,
\end{tikzcd}
\]
where $\varepsilon_{*}\iota$ is the push forward of the twisting morphism $\iota$ by $\varepsilon$. It induces a commutative square of Quillen adjunctions 
\[
\begin{tikzcd}[column sep=5pc,row sep=5pc]
\mathcal{C}_\infty\text{-}\mathsf{coalg} \arrow[r,"\widehat{\Omega}_\iota"{name=B},shift left=1.1ex] \arrow[d,"\mathrm{Coind}_{\varepsilon} "{name=SD},shift left=1.1ex ]
&\mathsf{dg}~\mathsf{abs}~\mathsf{Lie}\text{-}\mathsf{alg}^{\mathsf{comp}} \arrow[d,"\mathrm{Id}"{name=LDC},shift left=1.1ex ] \arrow[l,"\widehat{\mathrm{B}}_\iota"{name=C},,shift left=1.1ex]  \\
\mathsf{dg}\text{-}\mathsf{mod} \arrow[r,"\widehat{\Omega}_{\varepsilon_{*}\iota}"{name=CC},shift left=1.1ex]  \arrow[u,"\mathrm{Res}_{\varepsilon}"{name=LD},shift left=1.1ex ]
&\mathsf{dg}~\mathsf{abs}~\mathsf{Lie}\text{-}\mathsf{alg}^{\mathsf{comp}}~. \arrow[l,"\widehat{\mathrm{B}}_{\varepsilon_{*}\iota}"{name=CB},shift left=1.1ex] \arrow[u,"\mathrm{Id}"{name=TD},shift left=1.1ex] \arrow[phantom, from=SD, to=LD, , "\dashv" rotate=0] \arrow[phantom, from=C, to=B, , "\dashv" rotate=-90]\arrow[phantom, from=TD, to=LDC, , "\dashv" rotate=0] \arrow[phantom, from=CC, to=CB, , "\dashv" rotate=-90]
\end{tikzcd}
\] 

by the functionality of the model structures on complete algebras over cooperads of \cite[Theorem 10.32]{grignoulejay18}. One can easily see that the bottom adjunction is the free-forgetful adjunction. The category of dg modules is equivalent to the category of coalgebras over the trivial operad $I$ and the bottom complete bar-cobar is just the free-forgetful adjunction. Thus the model structure transferred along this adjunction has quasi-isomorphisms as weak equivalences. Finally, since the identity functor 

\[
\mathrm{Id}: \left(\mathsf{dg}~\mathsf{abs}~\mathsf{Lie}\text{-}\mathsf{alg}^{\mathsf{comp}}, \mathsf{W.eq} \right) \longrightarrow \left(\mathsf{dg}~\mathsf{abs}~\mathsf{Lie}\text{-}\mathsf{alg}^{\mathsf{comp}}, \mathsf{Q.iso} \right) 
\]
\vspace{0.2pc}

is a right Quillen functor and since every object is fibrant, then it sends weak equivalences of dg absolute Lie algebras to quasi-isomorphisms and therefore any weak equivalence of dg absolute Lie algebras is in particular a quasi-isomorphisms.
\end{proof}

\begin{Remark}
On the other hand, any filtered quasi-isomorphism of complete absolute dg Lie algebras is a weak equivalence in the first model structure by \cite[Theorem 10.26]{grignoulejay18}.
\end{Remark}

\begin{Remark}
The above proposition can be generalized \textit{mutatis mutandis} to the case of any twisting morphism $\alpha: \mathcal{C} \longrightarrow \mathcal{P}$ between a reduced cofibrant dg operad $\mathcal{P}$ and a conilpotent dg cooperad $\mathcal{C}$. In particular, this also holds for weak equivalences of dg absolute associative algebras.
\end{Remark}

\begin{Remark}
Model category structures on algebras over a conilpotent dg cooperads behave in an analogous way to what happens with model category structures on coalgebras over conilpotent dg cooperads, as described in \cite{DrummondColeHirsh14}.
\end{Remark}

Let us explain how to recover the results of \cite{campos2020lie} from Theorem \ref{thm: inclusion infini cat des Lie absolues}. The first thing to notice is that the $\mathrm{cAbs} \dashv \mathrm{Res}$ adjunction between complete dg Lie algebras and dg absolute Lie algebras 
\[
\begin{tikzcd}[column sep=7pc,row sep=3pc]
            \mathsf{dg}~\mathsf{abs}~\mathsf{Lie}\text{-}\mathsf{alg}^{\mathsf{comp}} \arrow[r,"\mathrm{Res}"{name=F}, shift left=1.1ex] 
           &\mathsf{dg}~\mathsf{Lie}\text{-}\mathsf{alg}~, \arrow[l, shift left=.75ex, "\mathrm{cAbs}"{name=U}]
            \arrow[phantom, from=F, to=U, , "\dashv" rotate=90]
\end{tikzcd}
\]

is a Quillen adjunction since $\mathrm{Res}$ preserves fibrations and, by Proposition \ref{thm: weak equiv inclues dans les quasi-isos}, it also preserves weak equivalences. A dg Lie algebra $\mathfrak{g}$ is \textit{homotopy complete} precisely when the derived unit of adjuction
\[
\mathbb{L}(\eta): \mathfrak{g} \qi \mathrm{Res}~\mathbb{L}\mathrm{cAbs}(\mathfrak{g}) 
\]
is a quasi-isomorphism. Indeed, this derived unit can be computed as 
\[
\mathbb{L}(\eta): \mathfrak{g} \longrightarrow \widehat{\Omega}_\iota \mathrm{B}_\pi \mathfrak{g}~, 
\]
where the right hand side term is indeed the completion of the standard cofibrant resolution $\Omega_\pi\mathrm{B}_\pi \mathfrak{g} \qi \mathfrak{g}$ given the classical bar-cobar construction, hence this map is a quasi-isomorphism if and only if this standard resolution is complete. See \cite{HarperHess} for more details on homotopy completeness. 

\medskip

Suppose that $\mathfrak{g}, \mathfrak{h}$ are two cofibrant dg Lie algebras such that $\mathfrak{U}(\mathfrak{g}) \simeq \mathfrak{U}(\mathfrak{h})$. Then $\widehat{\mathfrak{U}}(\mathrm{cAbs}(\mathfrak{g})) \simeq \widehat{\mathfrak{U}}(\mathrm{cAbs}(\mathfrak{h}))$, since they are cofibrant. Now, using the commutativity of the square of adjunctions in the proof of Proposition \ref{prop: carré qui commute} together with Theorem \ref{thm: inclusion infini cat des Lie absolues}, it follows that $\mathrm{cAbs}(\mathfrak{g})$ and $\mathrm{cAbs}(\mathfrak{h})$ are weakly-equivalent as complete dg absolute Lie algebras. So, in particular, when $\mathfrak{g}, \mathfrak{h}$ are also homotopy complete, it follows that they are quasi-isomorphic, and we recover the result of \cite{campos2020lie}. 

\begin{theorem}\label{thm: iso envelopantes Lie absolues}
Let $\mathfrak{g}$ and $\mathfrak{h}$ be two complete graded absolute Lie algebras. They are isomorphic as complete graded absolute Lie algebras if and only if their universal enveloping absolute algebras are isomorphic.
\end{theorem}

\begin{proof}
Using Theorems \ref{thm: inclusion infini cat des Lie absolues}, we know that there is a zig-zag of weak equivalences between these complete graded absolute Lie algebras if and only if there is between their universal enveloping absolute algebras. Now, by Proposition \ref{thm: weak equiv inclues dans les quasi-isos}, we know that these weak equivalences are in particular quasi-isomorphisms. Finally, by Proposition \ref{Prop: the homology has an absolute structure}, we can pass to the homology and obtain a direct isomorphism in both cases.
\end{proof}

\begin{Remark}
Any nilpotent graded Lie algebra is a complete graded absolute Lie algebra by the analogue of Proposition \ref{prop: les nilpotentes sont des absolues} for absolute Lie algebras. Therefore this result is a generalization of \cite[Corollary 0.12]{campos2020lie}.
\end{Remark}

This approach also allows us to generalize the above theorems to the universal enveloping absolute $\mathcal{A}_\infty$-algebra of an absolute $\mathcal{L}_\infty$-algebra. We change our \textit{shifting conventions}, considering this time unshifted $\mathcal{C}_\infty$ or $\mathcal{A}_\infty$-coalgebras, and therefore shifting absolute $\mathcal{L}_\infty$-algebras and absolute $\mathcal{A}_\infty$-algebras.

\begin{theorem}\label{thm: inclusion infini cat des L infinies}
Let $\mathfrak{g}$ and $\mathfrak{h}$ be two complete absolute $\mathcal{L}_\infty$-algebras. There exists a zig-zag of weak equivalences of absolute $\mathcal{L}_\infty$-algebras
\[
\mathfrak{g} \lqi \cdot \qi \mathfrak{h}
\]
if and only if there exists a zig-zag of weak equivalences of absolute $\mathcal{A}_\infty$-algebras
\[
\widehat{\mathfrak{U}}_\infty(\mathfrak{g}) \lqi \cdot \qi \widehat{\mathfrak{U}}_\infty(\mathfrak{h})  ~.
\]
\end{theorem}

\begin{proof}
We consider the dg operad $\Omega \mathrm{B} \mathcal{C}om$. Since it is a cofibrant resolution for the operad $\mathcal{C}om$, there exists a quasi-isomorphism of dg operads $f: \Omega \mathrm{B} \mathcal{C}om \qi \Omega s\mathcal{L}ie^*$. Therefore there is a Quillen equivalence between dg $\Omega \mathrm{B} \mathcal{C}om$-coalgebras and $\mathcal{C}_\infty$-coalgebras. Likewise, there is a Quillen equivalence between dg $\Omega \mathrm{B} \mathcal{A}ss$-coalgebras and $\mathcal{A}_\infty$-coalgebras. We consider the adjunction 
\[
\begin{tikzcd}[column sep=7pc,row sep=3pc]
           \mathsf{dg}~\Omega \mathrm{B} \mathcal{C}om\text{-}\mathsf{coalg} \arrow[r,"\mathrm{Res}_{\rho}"{name=F}, shift left=1.1ex] 
           &\mathsf{dg}~\Omega \mathrm{B} \mathcal{A}ss\text{-}\mathsf{coalg}~. \arrow[l, shift left=.75ex, "\mathrm{Coind}_{\rho}"{name=U}]
            \arrow[phantom, from=F, to=U, , "\dashv" rotate=-90]
\end{tikzcd}
\]

induced by the morphism of dg operads $\rho: \Omega \mathrm{B} \mathcal{A}ss \longrightarrow \Omega \mathrm{B} \mathcal{C}om$. Two dg $\Omega \mathrm{B} \mathcal{C}om$-coalgebras $C_1$ and $C_2$ are linked by a zig-zag of quasi-isomorphisms if and only if $\mathrm{Res}_{\rho}(C_1)$ and $\mathrm{Res}_{\rho}(C_2)$ are linked by a zig-zag of quasi-isomorphisms of dg $\Omega \mathrm{B} \mathcal{A}ss$-coalgebras. Using the commutative square 

\[
\begin{tikzcd}[column sep=5pc,row sep=5pc]
\mathsf{dg}~\Omega \mathrm{B} \mathcal{A}ss\text{-}\mathsf{coalg} \arrow[r,"\widehat{\Omega}_\iota"{name=B},shift left=1.1ex] \arrow[d,"\mathrm{Coind}_{\rho} "{name=SD},shift left=1.1ex ]
&\mathsf{abs}~\mathcal{A}_\infty\text{-}\mathsf{alg}^{\mathsf{comp}} \arrow[d,"\mathrm{Skew}"{name=LDC},shift left=1.1ex ] \arrow[l,"\widehat{\mathrm{B}}_\iota"{name=C},,shift left=1.1ex]  \\
\mathsf{dg}~\Omega \mathrm{B} \mathcal{C}om\text{-}\mathsf{coalg} \arrow[r,"\widehat{\Omega}_\iota "{name=CC},shift left=1.1ex]  \arrow[u,"\mathrm{Res}_{\rho}"{name=LD},shift left=1.1ex ]
&\mathsf{abs}~\mathcal{L}_\infty\text{-}\mathsf{alg}^{\mathsf{comp}} \arrow[l,"\widehat{\mathrm{B}}_\iota"{name=CB},shift left=1.1ex] \arrow[u,"\widehat{\mathfrak{U}}"{name=TD},shift left=1.1ex] \arrow[phantom, from=SD, to=LD, , "\dashv" rotate=0] \arrow[phantom, from=C, to=B, , "\dashv" rotate=-90]\arrow[phantom, from=TD, to=LDC, , "\dashv" rotate=0] \arrow[phantom, from=CC, to=CB, , "\dashv" rotate=-90]
\end{tikzcd}
\] 

and the fact that the horizontal adjunction are Quillen equivalences concludes the proof.
\end{proof}

\begin{Proposition}\label{prop: weak equiv de L infinie absolues sont des quasi-isos}
Let $f: \mathfrak{g} \qi \mathfrak{h}$ be a weak equivalence of complete absolute $\mathcal{L}_\infty$-algebras. It is in particular a quasi-isomorphism.
\end{Proposition}

\begin{proof}
The arguments are the same as in the proof of Theorem \ref{thm: weak equiv inclues dans les quasi-isos}.
\end{proof}

\begin{Definition}[Minimal absolute $\mathcal{L}_\infty$-algebra]
Let $(\mathfrak{g},\gamma_\mathfrak{g},d_\mathfrak{g})$ be an absolute $\mathcal{L}_\infty$-algebra. It is \textit{minimal} if the differential $d_\mathfrak{g}$ is equal to zero.
\end{Definition}

\begin{theorem}\label{thm: isos envelopantes absolues de L infinies}
Let $\mathfrak{g}$ and $\mathfrak{h}$ be two complete minimal absolute $\mathcal{L}_\infty$-algebras. If their universal enveloping absolute $\mathcal{A}_\infty$-algebras are weakly equivalent, there exists an $\infty$-isomorphism of $\mathcal{L}_\infty$-algebras between them. 
\end{theorem}

\begin{proof}
By Theorem \ref{thm: inclusion infini cat des L infinies}, there exists a zig-zag of weak equivalences of absolute $\mathcal{L}_\infty$-algebras between the two. Now, since by Proposition \ref{prop: weak equiv de L infinie absolues sont des quasi-isos}, any weak equivalence is a quasi-isomorphism, the aforementioned zig-zag is in particular a zig-zag of quasi-isomorphisms of absolute $\mathcal{L}_\infty$-algebras. By applying the restriction functor to $\mathcal{L}_\infty$-algebras to this zig-zag, we get a zig-zag of quasi-isomorphisms of $\mathcal{L}_\infty$-algebras. The existence of such a zig-zag is equivalent to the existence of a direct $\infty$-quasi-isomorphism, which in this case, since both algebras have zero differential, is an $\infty$-isomorphism. See \cite[Section 10.2.2]{LodayVallette12} for more details on $\infty$-(quasi)-isomorphisms. 
\end{proof}

\begin{Example}[Arity-wise nilpotent $\mathcal{L}_\infty$-algebras]
Nilpotent $\mathcal{L}_\infty$-algebras in the sense of \cite{Getzler09} are particular examples of absolute $\mathcal{L}_\infty$-algebras. Therefore the above theorem applies to \textit{minimal} nilpotent $\mathcal{L}_\infty$-algebras without any degree restriction.
\end{Example}

\begin{Remark}
The analogues of Theorems \ref{thm: inclusion infini cat des Lie absolues} and \ref{thm: inclusion infini cat des L infinies} should also hold when we replace the categories of absolute Lie/$\mathcal{L}_\infty$-algebras by their \textit{curved} counterparts. Indeed, using \cite[Remark 3.14]{lucio2022curved}, one should get analogue statements as in \cite{campos2020lie} concerning the deformation complexes of unital $\mathcal{C}_\infty$-coalgebras and unital $\mathcal{A}_\infty$-coalgebras. Then applying the same formalism is straightforward.
\end{Remark}

\bibliographystyle{alpha}
\bibliography{bibe}

\end{document}